\documentclass[aos,preprint]{imsart}

\RequirePackage{amsthm,amsmath,amsfonts,amssymb,mathrsfs}
\RequirePackage[round,authoryear]{natbib}
\RequirePackage[colorlinks,citecolor=blue,urlcolor=blue]{hyperref}
\RequirePackage{graphicx,subfigure}
\usepackage{bbm,bm}
\usepackage{enumerate}
\usepackage{caption}
\usepackage{xcolor}
\usepackage{threeparttable}

\startlocaldefs
\theoremstyle{plain}

\newtheorem{thm}{Theorem}[section]
\newtheorem{lemma}[thm]{Lemma}
\newtheorem{remark}{Remark}[section]

\theoremstyle{remark}

\endlocaldefs

\newcommand{\bbV}{{\bf V}}
\newcommand{\bbU}{{\bf U}}
\newcommand{\bbD}{{\bf D}}
\newcommand{\bbA}{{\bf A}}

\newcommand{\bbX}{{\bf X}}
\newcommand{\bbP}{{\bf P}}

\newcommand{\bbZ}{{\bf Z}}

\newcommand{\bbS}{{\bf S}}
\newcommand{\bbB}{{\bf B}}

\newcommand{\bSi}{\pmb \Sigma}

\newcommand{\bbalp}{{\bf \alpha}}

\newcommand{\bgl}{{\bf \lambda}}

\newcommand{\bGma}{{\bf \Gamma}}
\newcommand{\bPsi}{{\bf \Psi}}

\newcommand{\bgO}{{\bf \Omega}}
\newcommand{\bgL}{{\bf \Lambda}}

\newcommand{\bUps}{{\bf \Upsilon}}

\newcommand{\bbI}{{\bf I}}
\newcommand{\bbY}{{\bf Y}}
\newcommand{\bbT}{{\bf T}}

\newcommand{\bqn}{\begin{eqnarray*}}
	\newcommand{\eqn}{\end{eqnarray*}}

\newcommand{\rtr}{{\textrm{tr}}}
\newcommand{\rdiag}{{\textrm{diag}}}

\newcommand{\bqa}{\begin{eqnarray}}
	\newcommand{\eqa}{\end{eqnarray}}

\newcommand{\ep}{\varepsilon}

\newcommand{\al}{\alpha}

\newcommand{\CYRS}{\CYRS}

\def\Var{\mbox{var}}

\begin{document}

\begin{frontmatter}
\title{A CLT for the LSS of large dimensional sample covariance matrices with diverging spikes}
\runtitle{CLT for LSS with diverging spikes}

\begin{aug}
\author{\fnms{Zhijun} \snm{Liu}\ead[label=e1,mark]
{liuzj037@nenu.edu.cn}},
\author{\fnms{Jiang} \snm{Hu}\ead[label=e2,mark]{huj156@nenu.edu.cn}},
\author{\fnms{Zhidong} \snm{Bai}\ead[label=e3,mark]{baizd@nenu.edu.cn}},
\and 
\author{\fnms{Haiyan}
\snm{Song}\ead[label=e4,mark]{songhy716@nenu.edu.cn}}
\address{KLASMOE and School of Mathematics and Statistics, Northeast Normal University, China.
\printead{e1,e2,e3,e4}}

\end{aug}

\begin{abstract}
\textcolor{black}{ In this paper, we establish the central limit theorem (CLT) for linear spectral statistics (LSSs) of a large-dimensional sample covariance matrix when the population covariance matrices are involved with diverging spikes.
This constitutes a nontrivial extension of the Bai-Silverstein theorem (BST) (Ann Probab 32(1):553--605, 2004), a theorem that has strongly influenced the development of high-dimensional statistics, especially in the applications of random matrix theory to statistics. 
Recently, there has been a growing realization that the assumption of uniform boundedness of the population covariance matrices in the BST is not satisfied in some fields, such as economics, where the variances of principal components may diverge as the dimension tends to infinity. 
Therefore,  in this paper, we aim to eliminate this obstacle to applications of the BST.  Our new CLT accommodates spiked eigenvalues, which may either be bounded or tend to infinity. 
A distinguishing feature of our result is that the variance in the new CLT is related to both spiked eigenvalues and bulk eigenvalues, with dominance being determined by the divergence rate of the largest spiked eigenvalues.
The new CLT for LSS is then applied to test the hypothesis that the population covariance matrix is the identity matrix or a generalized spiked model. The asymptotic distributions of the corrected likelihood ratio test statistic and the corrected Nagao's trace test statistic are derived under the alternative hypothesis. Moreover, we present power comparisons between these two LSSs and Roy's largest root test. In particular, we demonstrate that except for the case in which the number of spikes is equal to one, the LSSs could exhibit higher asymptotic power than Roy's largest root test.}

\end{abstract}


\end{frontmatter}

\section{Introduction}

We consider the general sample covariance matrix  $ \bbB_n=\frac{1}{n}\bbT_p\bbX_n\bbX_n^{\ast}\bbT_p^{\ast} $, where $ \bbX_n $ is a $ p\times n $ matrix with independent and identically distributed (i.i.d.) standardized entries $ \left\lbrace x_{ij}\right\rbrace _{1\leq i\leq p, 1\leq j \leq n} $, $ \bbT_p $ is a $p \times p$ deterministic matrix, $\bbT_p\bbX_n$ is considered a random sample from the population with the population covariance matrix $\bbT_p\bbT_p^{\ast}=\bSi$, and $^*$ represents
the complex conjugate transpose. In the sequel, we simply write $\bbB\equiv\bbB_n$, $\bbT\equiv\bbT_p$ and $\bbX\equiv\bbX_n$ when there is no confusion. \textcolor{black}{Let $ \lambda_{1}\geq\dots\geq\lambda_{p} $ be the eigenvalues of $\bbB$.} For a known test function $f$, we call $\sum_{j=1}^{p}f\left( \bgl_{j}\right) $ a linear spectral statistic (LSS) of $\bbB$.  Because most of the classical test statistics in multivariate statistical analysis are associated with the eigenvalues of sample covariance matrices,  LSSs are remarkable tools in many statistical problems (see  \cite{Anderson03I,YaoB15L} for details). \textcolor{black}{ Through extensive study of high-dimensional data, 
 it has been discovered that the distributions of LSSs significantly differ between low-dimensional and high-dimensional data.}
For example, in the low-dimensional setting, Wilks' theorem (see \cite{Wilks38L}) provides the $\chi^2$ approximation for the likelihood ratio test  (LRT) statistic, which is a kind of LSS. However,  when $ p $ is large compared with the sample size $ n $, the LRT statistic exhibits Gaussian fluctuations (see \cite{10.1214/09-AOS694,JiangY13C}). 
\textcolor{black}{More generally, \cite{10.1214/aop/1078415845}  established the central limit theorem (CLT) for the LSSs of a high-dimensional $ \bbB $ under Gaussian-like moments condition by employing random matrix theory (RMT). Here the term `Gaussian-like moments' refers to the population second-order and fourth-order moments are the same as those of real or complex Gaussian population.} We refer to this CLT as the Bai--Silverstein theorem (BST) for brevity. Following the work of \cite{10.1214/aop/1078415845}, many extensions have been developed under different settings.  \textcolor{black}{\cite{pan2008central} generalized the BST by relaxing the Gaussian-like moments condition of $ x_{ij} $, which at the price of adding a structural condition on $ \bbT $.  } \cite{zheng2012central}, \cite{YangP15I} and \cite{BaoH22S}  extended the BST to multivariate $F$ matrices, canonical correlation matrices and block correlation matrices, respectively. \cite{pan2014comparison} presented the CLT for the LSS of noncentered sample covariance matrices, and \cite{ 10.1214/14-AOS1292} studied the case of an unbiased sample covariance matrix when the population mean is unknown. \cite{chen2015clt} focused on the ultrahigh dimensional case in which the dimension $p$ is much larger than the sample size $n$. \cite{GaoH17H} and \cite{LiW21C} studied the CLTs for the LSSs of high-dimensional Spearman and Kendall's rank correlation  matrices, respectively. Without attempting to be comprehensive, we also refer readers to other extensions \citep{BaiM07A, BaiH15C, BaiL19C, ZhengC19T, BannaN20C, najim2016gaussian, baik2018ferromagnetic, hu2019high, JiangB21G}.

\textcolor{black}{Almost all the literature mentioned above have traditionally assumed that the spectral norms of $ \bSi $ are bounded in $ n $.}
This assumption limits the applications in data analysis because in many fields, such as economics and wireless communication networks, the leading eigenvalues may tend to infinity. We present two examples here.
\begin{itemize}
	 
\item \textit{\textbf{Signal detection}} (\cite{10.1093/biomet/asw060}): We consider a single signal model:
	\begin{equation*} 
		{\bm x}=\chi_{s}^{1/2}u{\bm h} +\sigma\bm{v},   
	\end{equation*} 
	where $ {\bm h} $ is an unknown $ p $-dimensional unit vector, $ u $ is a random variable distributed as $ N(0,1) $, $  \chi_{s}$ is the signal strength, $\sigma  $ is the noise level, and $\bm {v} $ is
 a random noise vector that is independent of $ u $ and follows a multivariate Gaussian distribution $N_p (0, \bm\Sigma_{v})$.
 It is easy to check that the covariance matrix of $\bm x$ is $\bm\Sigma_x=\sigma^2\bm\Sigma_{v}+\chi_s{\bm h}{\bm h}^{\top}$. When the noise level is low, 
but the signal strength is large and sometimes tends to infinity, it is illogical to assume the boundedness of $\bm\Sigma_x$.

\textcolor{black}{\item \textit{\textbf{Factor model }}(\cite{bai02}): Many economic analyses, such as arbitrage pricing theory and analyses of the rank of a demand system, align naturally within the framework of the factor model:
\begin{align*} \mathop{\bm x_{t}}\limits_{\left( N\times1\right) }=\mathop{\bgL}\limits_{\left( N\times r\right) }\mathop{\bm f_{t}}\limits_{\left( r\times 1\right) }+\mathop{\bm\ep_{t}}\limits_{\left( N\times 1\right) } ~~t=1,\dots,T.
\end{align*}
where  $ \bm x_{t} $ is the observed data, $ N $ represents the number of cross-sections,  $ T $ is a large time dimension, and $ \bm f_{t}, \bgL $ and  $  \bm\ep_{t} $ represent the common factors, the factor loadings and the idiosyncratic error term, respectively. To ensure the identification of the model, several conventional assumptions are needed, such as $ \mathbb{E}\bm f_{t}=\bf 0$, $ \mathbb{E}\left(\bm f_{t}\bm f_{t}^{\top} \right) =\bbI_{r} $, $  \bm\ep_{t} $ is independent of $ \bm f_{t}$ with $ \mathbb{E}\bm\ep_{t}=\bf 0$ and $ \mathbb{E}\bm\ep_{t}\bm\ep_{t}^{\top}=\bSi_{\ep}>0$. 
Then the covariance matrices of $ \bm x_{t} $ can be expressed as $ \bSi_{x}=\bgL\bgL^{\top}+\bSi_{\ep}$. A pervasive assumption is that the variances of the principal components $ \bgL\bm f_{t} $ can diverge as $N$ increases to infinity (see Assumption B of \cite{bai02}). Therefore, the spectral norms of $\bSi_{x}$ are unbounded.
} 
\end{itemize}

For these reasons, it is of practical value to obtain the asymptotic properties of the LSS when $ \bSi $ is unbounded. Therefore, in this paper,
we focus on a generalized CLT for the LSSs of a spiked covariance matrix with the following structure:
\begin{align}\label{ds}
  \bSi=\mathbf{V}\left(\begin{array}{cc}
	\bbD_{1} & 0 \\
	0 & \bbD_{2}
\end{array}\right) \mathbf{V}^{\ast},
\end{align}
where $\mathbf{V}$ is a unitary matrix, $\bbD_{1}$ is a diagonal matrix consisting of the descending unbounded eigenvalues, and $\bbD_{2}$ is the diagonal matrix of the bounded eigenvalues. 
\textcolor{black}{As an application, the established CLT is employed to study the asymptotic behaviors of two special LSSs, i.e., the likelihood ratio (LR) statistic and Nagao's trace (NT) statistic,  under the hypothesis 
\begin{align} \label{hytest1}
		H_{0}:\bSi=\bbI_{p}\quad \text{vs.} \quad
		H_{1}:\bSi=\mathbf{V}\left(\begin{array}{cc}
			\bbD_{1} & 0 \\
			0 & \bbI_{p-M}
		\end{array}\right) \mathbf{V}^{\ast},
\end{align} 
where $M$ is a constant.
We also derive the asymptotic power of Roy's largest root test to detect the above hypothesis and make a comparison with these two LSSs. }

 The setting \eqref{ds} is attributed to the famous spiked model
in which a few large eigenvalues of the population covariance matrix are assumed to be well separated from the remaining eigenvalues \citep{10.1214/aos/1009210544}. The spiked model has served as the foundation for a rich theory of principal component analysis through the performance of extreme eigenvalues, as discussed in \cite{BAIK20061382, 10.2307/24307692, 10.1214/07-AIHP118, Nadler08F,JungM09P,  BAI2012167, OnatskiM14S, BloemendalK16P, WangY17E, DonohoG18O, PerryW18O, JohnstoneP18P, YangJ18E, YaoZ18E, Dobriban20P, JohnstoneO20T, cai2020limiting, JiangB21G}. 
There are also several works that have considered the asymptotic behaviors of various quantities as the spike strengths tend to infinity. Specifically, \cite{Zhou15} focused on the consistency of the sample eigenvector, corresponding to the largest eigenvalue of the sample covariance matrix, under high dimension and low sample size settings, when $ \bSi $ is unbounded and the data set is Gaussian. \textcolor{black}{\cite{Wang17} derived the asymptotic distributions of the spiked eigenvalues and eigenvectors when $ \bSi $ is diagonal and unbounded, and the data set is sub-Gaussian.
}  
Recently, \cite{li2020asymptotic}, \cite{yin2021spectral} and \cite{zhang2022asymptotic} investigated the trace of a large sample covariance matrix under the spiked model assumption.


To summarize, the contributions of this paper are as follows.
\textcolor{black}{
\begin{enumerate}
\item We demonstrate a nontrivial extension of the BST to the situation in which the spectral norms of the population covariance matrices are allowed to diverge as $\min\{p,n\}\to\infty$. In particular, we show how the test function $ f $ and the divergence rate of the population spectral norm affect the new CLT. 
\item It was previously reported that Gaussian-like moments or diagonality of the population covariance matrix are necessary for the CLT of the LSS (e.g., \cite{10.1214/14-AOS1292}). Nevertheless, we prove that these restrictions can be completely removed by normalizing the LSS. More importantly, even if no limit exists on the variance of the LSS, the new CLT could still hold.
\item The entire technical part of this paper is built on the decomposition of the LSS $ \sum_{j=1}^{p}f\left(\lambda_{j} \right)  = \sum_{j=1}^{M}f\left(\lambda_{j} \right)  + \sum_{j=M+1}^{p}f\left(\lambda_{j} \right)  $.  Because the classical delta method cannot be applied to the unbounded part $\sum_{j=1}^{M}f\left(\lambda_{j} \right) $ and the bounded part $ \sum_{j=M+1}^{p}f\left(\lambda_{j} \right)  $  is not a strict LSS of a sample covariance matrix, the results of \cite{10.1214/aop/1078415845} and \cite{JiangB21G} cannot be adopted directly. In this paper, 
we leverage a `generalized delta method' and employ skillful transformations to prove the CLTs for the unbounded and bounded parts, respectively. Moreover, we prove that the unbounded and bounded parts are asymptotically independent, which leads to the establishment of the new CLT. 
\item 
We verify that  Roy's largest root test is most powerful among the common tests when the alternative \eqref{hytest1} has only one spiked eigenvalue, which has also been mentioned by \cite{Olson74,10.1093/biomet/asw060}. Furthermore, 
we demonstrate that when the number of spikes is larger than one, the LSSs could exhibit higher asymptotic power than Roy's largest root test when the divergence rates of spiked eigenvalues are higher than $ \sqrt{n} $. 
\end{enumerate}}

The remaining sections are organized as follows: Section \ref{section 2} presents a detailed description of our notations and assumptions. The main results for the CLT for the LSSs of sample covariance matrices are stated in Section \ref{section 3}. In Section \ref{section 5}, we explore an application of our main results. 
We also report the results of numerical studies in Section \ref{simulation}.  
Technical proofs are presented in Section \ref{section 6}.
\textcolor{black}{This paper is also accompanied by an online supplementary file that includes the following materials: (i) some postponed proofs for Theorems \ref{thm1}--\ref{RLRT power}; (ii) some additional simulation results, and (iii) some useful lemmas. }

\section{Notations and assumptions}\label{section 2}

 Throughout the paper, we use bold capital letters and bold italic lowercase letters to represent matrices and vectors, respectively. Scalars are represented by regular letters.  $\boldsymbol{e}_{i}$ denotes a standard basis vector whose components are all zero, except the $i$-th component,
which is equal to 1. We use tr$ (\bbA) $, $ \bbA^{\top} $ and $ \bbA^{\ast} $ to denote the trace, transpose and conjugate transpose of matrix $ \bbA $, respectively. We also use $f'$ to denote the derivative of function $f$, and we use $ \frac{\partial}{\partial z_{1}}f(z_{1},z_{2}) $ to denote the partial derivative of function $ f $ with respect to $ z_{1} $. 
Let $ \left[\bbA \right]_{ij}  $ denote the $ (i,j) $-th entry of the matrix $ \bbA $ and $ \oint_{\mathcal{C}}f(z)dz $ denote the contour integral of $ f(z) $ on the contour $ \mathcal{C} $. Let $ \lambda_{i}^{\bbA} $  be the $ i $th largest eigenvalue of matrix $ \bbA $. 
Weak convergence is denoted by $ \stackrel{d}{\rightarrow}$.  Throughout this paper,  we use $o(1)$ (resp. $o_p (1)$) to denote a negligible scalar (resp. in probability), and the notation $C$ represents a generic constant that may vary from line to line.

\textcolor{black}{ We adopt the notation $ \bbX=(\boldsymbol x_{1},\ldots,\boldsymbol x_{n})=(x_{ij}) $, $ 1\leq i\leq p $, $ 1\leq j\leq n $. Let  $\rho_{ 1}\geq\cdots \geq \rho_{ p}$ be the eigenvalues of $ \bSi $ and  
 the singular value decomposition of $\bbT$ be
\begin{equation}\label{decT}
	\bbT=\bbV\bbD^{1/2}\bbU^*=
	(\bbV_1,\bbV_2)
	\left( 
	\begin{array}{cc}
		\bbD_{1}^\frac{1}{2} & 0\\
		0 & \bbD_{2}^\frac{1}{2}
	\end{array}
	\right)
	(\bbU_1,\bbU_2)^{\ast}. 
\end{equation} 
Here  $\bbU$ and $\bbV$ are unitary matrices, and  $\bbD_{1}=\rdiag(\underbrace{\alpha_1,\dots,\alpha_1}_{d_{1}},\underbrace{\alpha_2,\dots,\alpha_2}_{d_{2}},\dots,\underbrace{\alpha_K,\dots,\alpha_K}_{d_{K}})$ is a diagonal matrix whose diagonal elements tend to infinity. To avoid confusion, we refer to $\{\alpha_i,i=1,\dots,K\}$ as the diverging spikes in the following. } Assume $ d_{1}+\cdots+d_{K}=M$.  $\bbD_{2}$ is the diagonal matrix of the eigenvalues with bounded components, including bounded spiked eigenvalues and bulk eigenvalues. 
Moreover, let $d_0=0$ and  $ J_{k}=\left\lbrace \sum_{i=0}^{k-1}d_i+1, \ldots, \sum_{i=0}^{k}d_i\right\rbrace $, thus $  \rho_{i} = \al_{k} $ if $ i \in J_{k} $.   Then, 
 the corresponding sample covariance matrix $\bbB=\frac{1}{n}\bbT\bbX\bbX^{\ast}\bbT^{\ast}$
is the so-called generalized spiked sample covariance matrix. 
{Corresponding to the decomposition of $\bbD$, we decompose $ \bbV=\left(\bbV_{1},\bbV_{2} \right)$,  $ \bbU=\left(\bbU_{1},\bbU_{2} \right)$, and denote $\bGma=\bbV_{2}\bbD_{2}^{1/2}\bbU_{2}^{\ast} $, $ \boldsymbol r_{j}=\frac{1}{\sqrt{n}}\bGma\boldsymbol x_{j} $, and $ \bbA_{j}=\frac{1}{n}\bGma\bbX\bbX^{*}\bGma^{*}-z\bbI-\boldsymbol r_{j}\boldsymbol r_{j}^{*}. $ Let $ \mathbb{E}_{j} $ be the conditional expectation with respect to the $ \sigma $-field generated by $ \boldsymbol r_{1}, \dots, \boldsymbol r_{j} $. 
%
For any matrix $\bbA $ with real eigenvalues, the empirical spectral distribution of $\bbA $ is denoted by
\begin{equation*}
	F^{\bbA}\left(x \right)=\frac{1}{p}\left(\text{number of eigenvalues of }  \bbA \leq x \right).  	
\end{equation*}	
For any function of bounded variation $F$ on the real line, its Stieltjes transform is defined by
$$
m_{F}(z)=\int \frac{1}{\lambda-z} \mathrm{~d} F(\lambda), \quad z \in \mathbb{C}^{+} :=\{z \in \mathbb{C}: \Im z>0\}.
$$

The assumptions used to obtain the results in this paper are as follows:
\newtheorem{assumption}{Assumption}[]
\textcolor{black}{\begin{assumption} \label{ass1}
	$ \{x_{ij},  1\leq i\leq p ,  1\leq j\leq n \} $ {are i.i.d. random variables with common moments} $$ \mathbb{E}x_{ij}=0,\quad  \mathbb{E}\left| x_{ij}\right| ^{2}=1, \quad  \beta_{x}= \mathbb{E}\left|x_{ij}\right| ^{4}- \left|\mathbb{E}x_{ij}^{2}\right|^{2}-2,  \quad \alpha_{x}=\left|\mathbb{E}x_{ij}^{2}\right|^{2}.  $$ 
\end{assumption} }
\begin{assumption} \label{ass2}
	{ As $\min\{p,n\}\to\infty$, the ratio of the dimension-to-sample size (RDS)} $ c_{n}:={p}/{n}\rightarrow c>0. $
\end{assumption} 
\begin{remark}
	Assumptions \ref{ass1} and \ref{ass2} are standard in RMT. If $\mathbb{E}x_{ij}\neq 0$, we can use the centralized sample covariance matrices and $n-1
	$ instead of $\bbB_n$ and $n$, respectively, and the following results also hold. Details can be found in \citep{10.1214/14-AOS1292}. Therefore, in the sequel, we assume that $\mathbb{E}x_{ij}= 0$ without loss of generality. 
\end{remark}
\begin{assumption}\label{ass3}
	 {$ \bbT $ is nonrandom. As $\min\{p,n\}\to\infty$, $\alpha_K\to\infty$ and $H_n:=F^{\bGma\bGma^{*}}\stackrel{d}{\rightarrow}H$, where $H$ is a distribution function on the real line.  $M$ is fixed.}

\end{assumption} 

\begin{remark}
Similar to \cite{Silverstein95S} that under Assumptions 1-3, $F^{\bbB}\stackrel{d}{\rightarrow}F^{c,H}$ almost surely, where $F^{c,H}$ is a nonrandom distribution function whose Stieltjes transform $m:=m_{F^{c,H}}(z)  $ satisfies the following equation:
\begin{align}\label{Sequation}
  m=\int \frac{1}{t(1-c-c z m)-z} d H(t).
\end{align}
In the sequel, we call $F^{c,H}$ the limiting spectral distribution (LSD) of $\bbB$.
Moreover, because the matrix $\underline{\bbB}=\frac{1}{n}\bbX^*\bbT^{*}\bbT\bbX$ shares the same nonzero eigenvalues as $\bbB$, equation \eqref{Sequation} can be rewritten as
$$\underline{m}=-\left(z-c \int \frac{t }{1+t \underline{m}}d H(t)\right)^{-1},$$ 
where $\underline{m}:=m_{\underline{F}^{c,H}}(z)$ represents the Stieltjes transform of the LSD of $\underline{\bbB}$.
\end{remark} 
\begin{assumption}\label{ass4}
	Test functions $ f_{1},\dots, f_{h} $ are analytic on a connected open region of the complex plane containing the support of $ F^{c_{n},H_{n}}$ for almost all $ n. $ Moreover, we suppose that for any $l=1,\dots,h$,  $$\lim_{\{x_{n},y_n\}\to\infty\atop {x_{n}}/{y_{n}}\rightarrow 1}\frac{f_{l}'\left(x_{n} \right) }{f_{l}'\left(y_{n} \right)}= 1 .$$
\end{assumption}    
\begin{remark}
	In fact, Assumption \ref{ass4} is not highly restrictive in practice, as many common functions such as logarithmic and polynomial functions satisfy it. \textcolor{black}{However, it is worth noting that the exponential function does not satisfy this assumption.}
\end{remark} 

\textcolor{black}{ For convenience of description, we introduce some notations before presenting the main results in the next section. Let $ \underline{F}^{c,H} $ denote the LSD of matrix $ n^{-1}\bbX^{\ast}\bbU_{2}\bbD_{2}\bbU_{2}^{*}\bbX$,  
$ \bbU_{1}=\left(u_{ij} \right)_{i=1,\dots,p;j=1,\dots,M}  $,  ~$\mathcal{U}_{i_{1}j_{1}i_{2}j_{2}}=\sum_{t=1}^{p}\overline{u}_{ti_{1}}u_{tj_{1}}u_{ti_{2}}\overline{u}_{tj_{2}}$, $\phi_n\left(x \right)=x\left(1+c_n\int\frac{t}{x-t}dH_n\left(t \right)  \right)$,$$ \phi_{k}=\phi\left(x \right)\mid_{x=\al_{k}}=\al_{k}\left(1+c\int\frac{t}{\al_{k}-t}dH\left(t \right)  \right),~~
\theta_{k}=\phi_{k}^{2}\underline{m}_{2}\left( \phi_{k}\right), ~~\nu_{k}=\phi_{k}^{2} \underline{m}^{2}\left(\phi_{k}\right),$$$$  \underline{m}\left( \lambda\right)=\int\frac{1}{x- \lambda }d\underline{F}^{c,H}\left( x\right),~~\underline{m}_{2}\left( \lambda\right)=\int\frac{1}{\left( \lambda-x\right) ^{2}}d\underline{F}^{c,H}\left( x\right), $$ 
$$  c_{nM}=\dfrac{p-M}{n},~~H_{2n}=F^{\bbD_{2}},  ~~ \bbP_{n}(z)=\left( (1-c_{nM})\bGma\bGma^{*}-zc_{nM}m_{2n0}(z)\bGma\bGma^{*}-z\bbI_{p}\right)^{-1},  $$
$$ \varpi_{nkl}=\frac{\phi_{n}\left(\al_{k} \right)}{\sqrt{n}} f_{l}'\left(\phi_{n}\left(\al_{k} \right)  \right), ~~
	s_{k}^{2}= \frac{\left(\al_{x}+1 \right)d_{k} }{\theta_{k}}+\frac{ \beta_{x}\nu_{k} \sum_{j_{1}, j_{2}\in J_{k}}\mathcal{U}_{j_{1}j_{1}j_{2}j_{2}} }{\theta_{k}^{2}}, $$
	$$\vartheta_{n}^{2}=\Theta_{0,n}(z_{1},z_{2})+\al_{x}\Theta_{1,n}(z_{1},z_{2})+\beta_{x}\Theta_{2,n}(z_{1},z_{2}),$$
		$$\Theta_{0,n}(z_{1},z_{2})=\dfrac{\underline{m}_{2n0}^{\prime}(z_{1}) \underline{m}_{2n0}^{\prime}(z_{2})  }{(\underline{m}_{2n0}(z_{1})-\underline{m}_{2n0}(z_{2}) )^{2} }-\dfrac{1}{(z_{1}-z_{2})^{2}},$$
		$$\Theta_{1,n}(z_{1},z_{2})=\frac{\partial}{\partial z_{2}}\left\lbrace \dfrac{\partial \mathcal{A}_{n}(z_{1},z_{2})}{\partial z_{1}}\dfrac{1}{1-\al_{x}\mathcal{A}_{n}(z_{1},z_{2})} \right\rbrace ,$$
		$${\mathcal{A}_{n}(z_{1},z_{2})}=\dfrac{z_{1}z_{2}}{n}\underline{m}_{2n0}(z_{1}) \underline{m}_{2n0}(z_{2})\mathrm{tr}{\bGma^{*}\bbP_{n}(z_{1})\bGma\bGma^{\top}\bbP_{n}(z_{2})^{\top} \bar{\bGma}},$$
		$$\Theta_{2,n}(z_{1},z_{2})=\dfrac{z_{1}^{2}z_{2}^{2}\underline{m}_{2n0}^{\prime}(z_{1}) \underline{m}_{2n0}^{\prime}(z_{2})}{n}\sum_{i=1}^{p}\left[ \bGma^{*}\bbP_{n}^{2}(z_{1})\bGma\right]  _{ii}\left[ \bGma^{*}\bbP_{n}^{2}(z_{2})\bGma\right]  _{ii},$$
		 \begin{align*}
    	\mu_{l}=&\nonumber-\frac{\alpha_{x}}{2 \pi i}\cdot\oint_{\mathcal{C}}\frac{  c_{nM} f_{l}(z)\int \underline{m}_{2n0}^{3}(z)t^{2}\left(1+t \underline{m}_{2n0}(z)\right)^{-3} d H_{2n}(t)}{\left(1-c_{nM} \int \frac{\underline{m}_{2n0}^{2}(z) t^{2}}{\left(1+t \underline{m}_{2n0}(z)\right)^{2}} d H_{2n}(t)\right)\left(1-\alpha_{x} c_{nM} \int \frac{\underline{m}_{2n0}^{2}(z) t^{2}}{\left(1+t \underline{m}_{2n0}(z)\right)^{2}} d H_{2n}(t)\right) }dz \\
    	&-\frac{\beta_{x}}{2 \pi i} \cdot \oint_{\mathcal{C}} \frac{c_{nM}  f_{l}(z)\int \underline{m}_{2n0}^{3}(z) t^{2}\left(1+t \underline{m}_{2n0}(z)\right)^{-3} d H_{2n}(t)}{1-c_{nM} \int \underline{m}^{2}_{2n0}(z) t^{2}\left(1+t \underline{m}_{2n0}(z)\right)^{-2} d H_{2n}(t)} dz,\quad l=1,\dots,h.	
    \end{align*}
	Here, $ m_{2n0}(z) $ is the Stieltjes transform of $ F^{c_{nM},H_{2n}} $, where $F^{c_{nM},H_{2n}}$ is the LSD $F^{c,H}$ with $\{c,~H\}$ replaced by $\{c_{nM},~ H_{2n}\}$, $ \underline{m}_{2n0}(z)=-\frac{1-c_{nM}}{z}+c_{nM}m_{2n0}(z) $ and $\mathcal{C}$ is a closed contour in the complex plane enclosing the support of  $ F^{c_n, H_n}$ and it is also enclosed in the analytic area of $ f_l $.} For clarity, $ m_{1n0}(z) $ denotes the Stieltjes transform of $ F^{c_{n},H_{n}} $, $m_{n}=\frac{1}{p}\mathrm{tr}\left(\bbB-z\bbI_{p} \right)^{-1} $, and $m_{2n}=\frac{1}{p-M}\mathrm{tr}\left(\bbS_{22}-z\bbI_{p-M} \right)^{-1}$. 

Note that $$\sum_{j=1}^{p}f\left( \bgl_{j}\right)=p\int f\left(x \right)dF^{\bbB}(x). $$ 
Thus, for brevity, we define the normalized LSSs as
$$  Y_{l}= \int f_{l}\left(x \right)dG_{n}\left( x\right)-{\sum_{k=1}^{K}d_{k}f_{l}\left(\phi_{n}\left(\al_{k} \right)  \right)}-\frac{M}{2\pi i}\oint_{\mathcal C}f_{l}\left(z \right)\frac{\underline{m}_{2n0}'(z)}{\underline{m}_{2n0}(z)}dz,\quad l=1,2,\dots,h, $$
where $$ G_{n}\left( x\right)=p[F^{\bbB}\left(x \right)-F^{c_{n},H_{n}}\left(x \right)] .$$ 

\section{Main results}  \label{section 3}
Now, we are in a position to present our main theorems and their proofs are provided in Section \ref{section 6} and the supplementary material. We first establish a CLT for an LSS
without any restrictions imposed on the Gaussian moments 
or on the structures of the population covariance matrix  by normalizing the LSS.
 

\begin{thm}\label{thm1}
Under Assumptions
	 \ref{ass1}--\ref{ass4}, we have $$ \frac{Y_{1}-\mu_{1}}{\varsigma_{1}}\stackrel{d}{\rightarrow}N\left(0, 1\right),    $$  where
\textcolor{black}{\begin{align}
		\varsigma_{1}^{2}=\sum_{k=1}^{K}\varpi_{nk1}^{2}s_{k}^{2} -\frac{1}{4\pi^{2}}\oint_{\mathcal{C}_{1}}\oint_{\mathcal{C}_{2}}f_{1}\left(z_{1} \right)f_{1}\left(z_{2} \right)\vartheta_{n}^{2}dz_{1}dz_{2},\label{thm3.1cov}
	\end{align}
	$\mathcal{C}_{1}$ and $\mathcal{C}_{2}$ are nonoverlapping and closed contours in the complex plane enclosing the support of  $ F^{c_{n}, H_{n}}$.  $\mathcal{C}_{1}$ and $\mathcal{C}_{2}$ are also enclosed in the analytic area of $ f_1. $ }
\end{thm}

\begin{remark}\label{remark3.1}
\textcolor{black}{Recall the definitions $s_{k}^{2}= \frac{\left(\al_{x}+1 \right)d_{k} }{\theta_{k}}+\frac{ \beta_{x}\nu_{k} \sum_{j_{1}, j_{2}\in J_{k}}\mathcal{U}_{j_{1}j_{1}j_{2}j_{2}} }{\theta_{k}^{2}} $ and
	$\vartheta_{n}^{2}=\Theta_{0,n}(z_{1},z_{2})+\al_{x}\Theta_{1,n}(z_{1},z_{2})+\beta_{x}\Theta_{2,n}(z_{1},z_{2}).$ Notably, the term $ \Theta_{0,n}(z_{1},z_{2}) $  has a limitation under Assumptions
	 \ref{ass1}--\ref{ass4}, which has already been discussed in \cite{10.1214/aop/1078415845}.   Additionally, if $ \bSi $ is complex, the convergence of $ \Theta_{1,n}(z_{1},z_{2}) $ is not guaranteed.
The term $ \Theta_{2,n}(z_{1},z_{2}) $ involves the quantities $ \left[ \bGma^{*}\bbP_{n}^{2}(z_{i})\bGma\right]_{ii} $, which depend not only on the eigenvalues of $ \bbD_{2} $ but also on their associated eigenvectors.
Furthermore, the term $ s_{k}^{2}$ indicates that the variance is influenced by the second and fourth moments of $ x_{ij} $, spiked eigenvalues, and their associated eigenvectors. The limit of $ \varpi_{nk1}^{2} $ is allowed to be infinite. }
\end{remark}
\begin{remark}
\textcolor{black}{After a closer look at the variance  (\ref{thm3.1cov}), we can also find that the first part of the formula (\ref{thm3.1cov}) is the variance containing diverging spikes and the second part is the variance containing the bounded eigenvalues. When $ \phi_{n}\left(\al_{1} \right)f^{\prime}_{1}\left(\phi_{n}\left(\al_{1} \right) \right) =o\left(\sqrt{n} \right),   $ the first term in formula (\ref{thm3.1cov}) tends to 0, and the variance is mainly affected by the second part. When $ \phi_{n}\left(\al_{1} \right)f^{\prime}_{1}\left(\phi_{n}\left(\al_{1} \right) \right)  $ is of order $ \sqrt{n} $, the two parts of formula (\ref{thm3.1cov}) are of the same order, and the variance is affected by both. When the order of $ \phi_{n}\left(\al_{1} \right)f^{\prime}_{1}\left(\phi_{n}\left(\al_{1} \right) \right)  $ is higher than $ \sqrt{n} $,  the first part of formula (\ref{thm3.1cov}) is much larger than the second part; therefore, the spiked part dominates the variance value.}
\end{remark}

As a minor price for the removal of the bounded spectrum condition, the new CLT described above applies only to a single LSS. To guarantee that the new CLT will apply to multiple normalized LSSs,  structural
assumptions about the population covariance matrices are needed. 
\textcolor{black}{
\begin{assumption}\label{ass5}
	$ \bbT $  {is real or the variables} $x_{ij}$ {are complex satisfying} $\alpha_{x}=0 .$ 
\end{assumption}
\begin{assumption}\label{ass6}
	$ \bbT^{*}\bbT $  {is diagonal or} $ \beta_{x}=0. $
\end{assumption}}	
\begin{remark}\label{remark2.4}
Assumptions \ref{ass5} and \ref{ass6} are used as a replacement for the Gaussian-like moments condition.   It has been proved by \cite{10.1214/14-AOS1292} that these two structural assumptions regarding the population matrices are necessary for their results when the Gaussian-like moments condition in the BST does not hold. 
	
\end{remark}

The following theorem is a nontrivial extension of the BST:
\begin{thm}\label{thm2}
	\textit{Under Assumptions \ref{ass1}--\ref{ass6}}, the random vector  $$ \left( \frac{Y_{1}-\mu_{1}}{\sigma_{1}},\dots,\frac{Y_{h}-\mu_{h}}{\sigma_{h}}\right)^{\top}\stackrel{d}{\rightarrow}N_{h}\left(0, \bPsi \right),    $$
    with     variance
\textcolor{black}{	\begin{align*}
		\sigma_{l}^{2}&=\sum_{k=1}^{K}\varpi_{nkl}^{2} s_{k}^{2}-\kappa_{nll}, \quad l=1,\dots,h,
    \end{align*} 
     and covariance matrix $  \bPsi =\left( \psi_{st}\right) _{h\times h} $, where  $\psi_{st}=\lim_{n\to\infty}\psi_{nst}$,
     	\begin{align*}
		\psi_{nst}=\frac{\sum_{k=1}^{K}\varpi_{nks}\varpi_{nkt}s_{k}^{2}-\kappa_{nst}}{\sqrt{\sum_{k=1}^{K}\varpi_{nks}^{2}s_{k}^{2} -\kappa_{nss}}\sqrt{\sum_{k=1}^{K}\varpi_{nkt}^{2}s_{k}^{2} -\kappa_{ntt}}},
	\end{align*} 
	\begin{align*}
		\kappa_{nst}
		&=\frac{1}{4 \pi^{2}} \oint_{\mathcal{C}_{1}} \oint_{\mathcal{C}_{2}} \frac{f_{s}\left(z_{1}\right) f_{t}\left(z_{2}\right)}{\left(\underline{m}_{2n0} \left(z_{1}\right)-\underline{m}_{2n0}\left(z_{2}\right)\right)^{2}} d \underline{m}_{2n0}\left(z_{1}\right) d \underline{m}_{2n0}\left(z_{2}\right) \\
		&+\frac{c_{nM} \beta_{x}}{4 \pi^{2}} \oint_{\mathcal{C}_{1}} \oint_{\mathcal{C}_{2}} \int \frac{f_{s}\left(z_{1}\right) f_{t}\left(z_{2}\right)t^2}{\left(\underline{m}_{2n0}\left(z_{1}\right) t+1\right)^{2}\left(\underline{m}_{2n0}\left(z_{2}\right) t+1\right)^{2}} d H_{2n}(t) d \underline{m}_{2n0}\left(z_{1}\right) d \underline{m}_{2n0}\left(z_{2}\right)\\
		&\quad+\frac{1}{4 \pi^{2}} \oint_{\mathcal{C}_{1}} \oint_{\mathcal{C}_{2}} f_{s}\left(z_{1}\right) f_{t}\left(z_{2}\right)\left[\frac{\partial^{2}}{\partial z_{1} \partial z_{2}} \log \left(1-a_n\left(z_{1}, z_{2}\right)\right)\right] d z_{1} d z_{2},
		\end{align*} and
\begin{align*}
a_{n}\left(z_{1}, z_{2}\right)=\alpha_{x}\left(1+\frac{\underline{m}_{2n0}\left(z_{1}\right) \underline{m}_{2n0}\left(z_{2}\right)\left(z_{1}-z_{2}\right)}{\underline{m}_{2n0}\left(z_{2}\right)-\underline{m}_{2n0}\left(z_{1}\right)}\right).
\end{align*}
}
	
\end{thm}

\begin{remark}\label{remark}

\textcolor{black}{Define
	$ \tilde{\psi}_{nst}=\dfrac{\sum_{k=1}^{K}\varpi_{nks}\varpi_{nkt}s_{k}^{2}-\frac{1}{4\pi^{2}}\oint_{\mathcal{C}_{1}}\oint_{\mathcal{C}_{2}}f_{s}\left(z_{1} \right)f_{t}\left(z_{2} \right)\vartheta_{n}^{2}dz_{1}dz_{2}}{\varsigma_{s}\varsigma_{t}}. $
	If $ \bm{\tilde{\Psi}}_n=(\tilde{\psi}_{nst})_{h\times h}$ is invertible for all sufficiently large $n$, we conjecture that, similar to Theorem \ref{thm1}, the convergence
 $$ \bm{\tilde{\Psi}}_n^{-1/2}\left( \frac{Y_{1}-\mu_{1}}{\sigma_{1}},\dots,\frac{Y_{h}-\mu_{h}}{\sigma_{h}}\right)^\top\stackrel{d}{\rightarrow}N_{h}\left(0, \bbI_h \right)    $$
holds without requiring Assumptions \ref{ass5} and \ref{ass6}. It should be noted that $  \bm{\tilde{\Psi}}_n $ is singular if the set of test functions is linearly dependent. However, determining the invertibility of $ \bm{\tilde{\Psi}}_n$ becomes challenging when the test functions are completely linearly independent. Hence, the extension to the removal of Assumptions \ref{ass5} and \ref{ass6} in Theorem \ref{thm2} is left for future work. }

\end{remark}

\begin{remark}\label{defineDP}
\textcolor{black}{
If $ \varpi_{nkl}\rightarrow 0$ as $n\rightarrow \infty$, Theorem \ref{thm2} coincides with Theorem 2.1 of \cite{10.1214/14-AOS1292}. 
If the test functions $ f_l=x $ and $ x^2 $, then Theorem \ref{thm2} reduces to Theorem 2.1 of \cite{yin2021spectral}. Notably, the results in \cite{yin2021spectral} required higher-order moment conditions.}
\end{remark}

{\color{black}{
\section{Application}\label{section 5}
In this section, we focus on a hypothesis test concerning whether the population covariance matrix $ \bSi $ is equal to the identity matrix or a spiked model, i.e.,
\begin{align} H_{0}:\bSi=\bbI_{p}\quad \text{vs.} \quad
	H_{1}:\bSi=\mathbf{V}\left(\begin{array}{cc}
		\bbD_{1} & 0 \\
		0 & \bbI_{p-M}
	\end{array}\right) \mathbf{V}^{\ast},\label{alter}
\end{align} 
where $\bbD_{1} $ is a diagonal matrix of the diverging spiked eigenvalues of $ \bSi $. There are several classical test statistics for this problem, but due to the limited length of this paper, we only consider the likelihood ratio (LR) test statistic \citep{Wilks38L} and the Nagao's trace (NT) test statistic \citep{Nagao1973} in this section. 
Specifically, the LR and NT statistics can be formulated as
$$
L=\operatorname{tr} \bbB-\log \left|\bbB\right|-p~~\mbox{and}~~W=\mathrm{tr}(\bbB-\bbI_{p})^{2}, 
$$
respectively.  Under the null hypothesis, the asymptotic properties of the LR and NT statistics for high-dimensional settings have been investigated extensively in the literature; here, we refer to \cite{10.1214/09-AOS694,JiangY13C,Ledoit02,wang2013sphericity,OnatskiM13A} for more details. 
Thus, in this section, we mainly focus on the alternative hypothesis. However,  to provide  better comparisons, we also present the asymptotic distributions under the null hypothesis in the following theorems.

\subsection{Asymptotic results for the LR and NT statistics}
In this subsection, we present the asymptotic results for LR and NT test statistics for the testing problem \eqref{alter}. 
\begin{thm}[CLT for the LR statistic]\label{thm3}
	Under Assumptions \ref{ass1}--\ref{ass4} with $c_{n}=p / n \rightarrow c \in(0,1)$, we have 
	\begin{itemize}
		\item (Under $H_0$) $$\dfrac{L-p\ell_l -\mu_{l}}{\varsigma_{l}} \stackrel{d}{\longrightarrow} N(0,1),  $$
		where 
		\begin{align*}
			\ell_l=1-\frac{c_{n}-1}{c_{n}} \log \left(1-c_{n}\right),~~~
			\mu_{l} = -\frac{\log \left(1-c_{n}\right)}{2}\al_{x}+\frac{c_{n}}{2}\beta_{x}
		\end{align*}
		and$$
		\varsigma_{l}^{2} =(\al_{x}+1)(-\log \left(1-c_{n}\right)- c_{n}).$$
		\item  (Under $H_1$) $$\dfrac{L-(p-M)\breve{\ell}_l  -\breve{\mu}_{l}}{\breve{\varsigma}_{l}} \stackrel{d}{\longrightarrow} N(0,1),  $$
		where 
		\begin{align*}
			\breve{\ell}_l=&1-\frac{c_{nM}-1}{c_{nM}} \log \left(1-c_{nM}\right),~
			\breve{\mu}_{l}= -\frac{\log \left(1-c_{nM}\right)}{2}\al_{x}+\frac{c_{nM}}{2} \beta_{x} \\
			&+\sum_{k=1}^{K}d_{k}\left( \phi_{n}\left({\al}_{k} \right)-\log\phi_{n}\left({\al}_{k} \right)-1 \right)-M(c_{nM}+\log(1-c_{nM})) \end{align*}	and\begin{align*}
			\breve{\varsigma}_{l}^{2} =&\sum_{k=1}^{K}\frac{\left(\phi_{n}\left(\al_{k} \right)-1  \right)^{2} }{n}s_{k}^{2}+(\al_{x}+1)\left(-\log(1-c_{nM})-c_{nM}\right).
		\end{align*}	

	\end{itemize}
	%
\end{thm}

\begin{remark}
	If $ c\geq1$, then $\bbB_n$ could be singular for large $n$, which would give rise to an undefined LR statistic $L$. Thus, the additional restriction $c<1$ is added in Theorem \ref{thm3}. 
\end{remark}
\begin{remark}
	Note that $\phi_{n}\left({\al}_{k} \right) $ and $s_k$ are defined in Section \ref{section 2}. Under the alternative hypothesis  $H_1$ in \eqref{alter}, we can adopt simplification  $\phi_n\left(\alpha_k \right)=\alpha_k+c_{nM}+o(1)$,  and $$s_{k}^{2}= {\left(\al_{x}+1 \right)d_{k} }+{ \beta_{x}\sum_{j_{1}, j_{2}\in J_{k}}\mathcal{U}_{j_{1}j_{1}j_{2}j_{2}} }+o(1). $$
\end{remark}

\begin{thm}[CLT for the NT statistic]\label{thm4}
	Under Assumptions \ref{ass1}--\ref{ass4}, we have
	\begin{itemize}
		\item (Under $H_0$) $$ \frac{W-pc_{n}-\mu_{w}}{ \varsigma_{w} }\stackrel{d}{\longrightarrow} N(0,1),   $$
		where 
		\begin{align*}
			\mu_{w} =c_{n}(\al_{x}+\beta_{x})~~\mbox{and}~~		\varsigma_{w}^{2} =(\al_{x}+1)(4c_{n}^{3}+2c_{n}^{2})+4\beta_{x}c_{n}^{3}.
		\end{align*}
		
		\item  (Under $H_1$)
		\begin{align*}
			\frac{W-(p-M)c_{nM} -\breve{\mu}_{w}}{\breve{\varsigma}_{w}}\stackrel{d}{\longrightarrow} N(0,1), 
		\end{align*}	
		where  \begin{align*}
			\breve{\mu}_{w}=&c_{nM}(\al_{x}+\beta_{x})+\sum_{k=1}^{K}d_{k}\left( \phi_{n}(\al_{k})-1 \right)^2-M c_{nM}^{2}\end{align*}and
			 \begin{align*}
			\breve{\varsigma}_{w}^{2} =&\sum_{k=1}^{K}\frac{4\phi_{n}^{2}\left(\al_{k} \right) \left(\phi_{n}\left(\al_{k} \right)-1  \right)^{2} }{n}s_{k}^{2}+(\al_{x}+1)\left(4c_{nM}^{3}+2c_{nM}^{2} \right)+4\beta_{x}c_{nM}^{3}.
		\end{align*}
		
	\end{itemize}

\end{thm}
\begin{remark}
		From the covariance terms $\breve{\varsigma}_{l}^{2}$ and $\breve{\varsigma}_{w}^{2} $, one can find that these CLTs are related to the components of the right singular vectors $\mathbf{U}$, but not to the left singular vectors  $\mathbf{V}$. Furthermore, since $\bbD_2$ is an identity matrix, $\bbU_2$ does not affect the asymptotic CLTs. Therefore, the only singular vectors of $\bbT$ affecting the results are $\bbU_1$, which are involved in $s_k^2$.
		\end{remark}
 The proofs of Theorems \ref{thm3} and \ref{thm4} are given in the supplementary material. To avoid confusion with the classical distributions of the LR test and NT test, we refer to the CLTs above as the corrected LR test (CLRT) and corrected NT test (CNTT) in the sequel. From Theorems \ref{thm3} and \ref{thm4}, we reject the null hypothesis $H_0$ in \eqref{alter} if 
	$$
	L>z_{\xi}{\varsigma_{l}} +p\ell_l +\mu_{l}
	$$ and
	$$
	W>z_{\xi}{\varsigma_{w}}+c_{n}(p+\alpha_{x}+\beta_{x}), 
	$$  where $ \xi $ is the significance level of the test and $ z_{\xi} $ is the $1-\xi$ quantile of the standard Gaussian distribution  $ \Phi$.  For the power functions of CLRT and CNTT, we have the following theorems.
	\begin{thm}[Power function of CLRT] \label{CLRT power}Under Assumptions
	 \ref{ass1}--\ref{ass4} with $c_{n}=p / n \rightarrow c \in(0,1)$  and $H_1$ in \eqref{alter}, we have that the power function of the CLRT $P_L= P(L>z_{\xi}{\varsigma_{l}} +p\ell_l +\mu_{l})$ satisfies 
		\begin{align} 
		P_{L}- \Phi\left(  \frac{  \sum_{k=1}^{K}d_{k}\left( \phi_{n}\left({\al}_{k} \right)-\log\phi_{n}\left({\al}_{k} \right) \right)-M(1+c)-z_{\xi}{\varsigma_{l}} }{\sqrt{\sum_{k=1}^{K}\frac{\left(\phi_{n}\left(\al_{k} \right)-1  \right)^{2} }{n}s_{k}^{2}+(\al_{x}+1)\left(-\log(1-c)-c\right)}}   \right)  \rightarrow 0, \label{powerl}
\end{align}
as $n\to\infty$.
	\end{thm}
\begin{thm}[Power function of CNTT] \label{CNTT power}Under Assumptions
	 \ref{ass1}--\ref{ass4}   and $H_1$ in \eqref{alter}, we have that the power function of the CNTT $P_W= P(W>z_{\xi}{\varsigma_{w}}+c_{n}(p+\alpha_{x}+\beta_{x}))$ satisfies 
		\begin{align} 
		P_{W}- \Phi\left(  \frac{  \sum_{k=1}^{K}d_{k}\left( \phi_{n}(\al_{k})-1 \right)^2-M c^{2}-2Mc-z_{\xi}{\varsigma_{w}} }{\sqrt{\sum_{k=1}^{K}\frac{4\phi_{n}^{2}\left(\al_{k} \right) \left(\phi_{n}\left(\al_{k} \right)-1  \right)^{2} }{n}s_{k}^{2}+(\al_{x}+1)\left(4c^{3}+2c^{2} \right)+4\beta_{x}c^{3}}}   \right)  \rightarrow 0, \label{powerc}
\end{align}
as $n\to\infty$.

	\end{thm}
\begin{remark}
Since $s_{k}^{2}$ is nonrandom and of order $O(1)$, $ P_{L}$ and $P_{W}$ tend to 1 as $\alpha_1\to\infty$.
The detailed analysis of the power functions of $ P_{L}$ and $P_{W}$ is discussed in the next subsection. 
\end{remark}

\subsection{Power analysis}\label{poweranalysis}
This subsection discusses the power functions of $ P_{L}$ and $P_{W}$. For simplicity, in this subsection, we assume that $\{x_{ij}\}$ are real, i.e., $\alpha_{x}=1$. We first derive the asymptotic power of Roy's largest root test (RLRT) to detect $ H_1 $ in \eqref{alter} for comparison. Recall RLRT statistic $\lambda_1.$ Under Assumptions
	 \ref{ass1}--\ref{ass4}   and $H_0$ in \eqref{alter}, it follows from Theorem 2.7 of \cite{DingY18N} that 
\begin{align*}
  \frac{\lambda_1-\mu_{r}}{\varsigma_{r}}\stackrel{d}{\rightarrow} F_{TW},
\end{align*}
where  $ \mu_r=\left(1+\sqrt{c_n}\right)^{2},  $ $ \varsigma_r=n^{-2/3}\left( 1+\sqrt{c_n}\right)\left(1+\sqrt{c_n^{-1}} \right) ^{1/3} $ and $F_{TW}$ is the Type 1 Tracy-Widom (TW) distribution.
 Let $ t_{\xi} $ be the $1-\xi$ quantile of TW distribution with significance level $\xi$. 
Then we have the following theorem regarding the power function of RLRT.
\begin{thm}[Power function of RLRT]\label{RLRT power} Under Assumptions
	 \ref{ass1}--\ref{ass4}   and $H_1$ in \eqref{alter}, if the multiplicity of $\alpha_1$ is one, then the power function of the RLRT $P_R= \mathbb{P}(\lambda_1>t_{\xi}{\varsigma_{r}}+\mu_r)$ satisfies 
		\begin{align} 
		P_{R}- \Phi\left( - \frac{t_{\xi}{\varsigma_r}+{\mu_r}-\phi_{n}\left(\al_{1}\right)}{s_1\phi_{n}\left(\al_{1}\right)/\sqrt{n}} \right) \rightarrow 0, \label{powerr}
\end{align}
as $n\to\infty$.
\end{thm}
The proof of this theorem is postponed to the supplementary material.  It is clear that if $\alpha_1>1+\sqrt{c}$ uniformly, then $P_{R}\to1$ as $n\to\infty$.  According to \cite{Anderson03I}, compared with the classical LSSs, RLRT has the highest asymptotic power to detect rank-one alternatives and under low dimensional settings. This property has also been demonstrated by \cite{Olson74}  and \cite{10.1093/biomet/asw060}.  In the following, we discuss the asymptotic power functions of CLRT, CNTT and RLRT. In particular,  we will show that except for the case in which the number of spikes is equal to 1, CLRT and CNTT may exhibit higher asymptotic power than RLRT in some scenarios.
 
 Define  
\begin{align}\label{varkappaLWR}
\begin{split}
	  \varkappa_L&=\frac{  \sum_{k=1}^{K}d_{k}\left( \phi_{n}\left({\al}_{k} \right)-\log\phi_{n}\left({\al}_{k} \right) \right)-M(1+c)-z_{\xi}{\varsigma_{l}} }{\sqrt{\sum_{k=1}^{K}\frac{\left(\phi_{n}\left(\al_{k} \right)-1  \right)^{2} }{n}s_{k}^{2}-2\left(\log(1-c)+c\right)}} \\
	   \varkappa_W&= \frac{  \sum_{k=1}^{K}d_{k}\left( \phi_{n}(\al_{k})-1 \right)^2-M c^{2}-2Mc-z_{\xi}{\varsigma_{w}} }{\sqrt{\sum_{k=1}^{K}\frac{4\phi_{n}^{2}\left(\al_{k} \right) \left(\phi_{n}\left(\al_{k} \right)-1  \right)^{2} }{n}s_{k}^{2}+2\left(4c^{3}+2c^{2} \right)+4\beta_{x}c^{3}}} \\
	   \varkappa_R&=\frac{\phi_{n}\left(\al_{1}\right)-{\mu_r}-t_{\xi}{\varsigma_r}}{s_1\phi_{n}\left(\al_{1}\right)/\sqrt{n}}.
\end{split}
\end{align}
Since  CLRT, CNTT and RLRT statistics are all asymptotically normally distributed under the alternative hypothesis, 
according to formulas (\ref{powerl})--(\ref{powerr}), comparing the convergence rates of power functions $P_L$, $P_W$, and $P_R$ is equivalent to comparing the divergence rates
 $\varkappa_L$,   $\varkappa_W$ and  $\varkappa_R$ tend to infinity. 
 Note that $\{z_{\xi}{\varsigma_{l}}$, $z_{\xi}{\varsigma_{w}}$, $t_{\xi}{\varsigma_r}$\} are all of order $O(1)$, $\{K,M\}$ are fixed, $0<c<1$,
 $ \phi_{n}(\al_{k})=\al_{k}+c+o(1)$ and $s_{k}^{2}= {2d_{k} }+{ \beta_{x}\sum_{j_{1}, j_{2}\in J_{k}}\mathcal{U}_{j_{1}j_{1}j_{2}j_{2}} }+o(1)$. In the sequel, we use the notations $A_n=\Omega(B_n)$, $A_n\simeq B_n$ and  $A_n\asymp B_n$ to denote $B_n=O(A_n)$,  $A_n= B_n+o(B_n)$ and $C^{-1}A_n< B_n<CA_n$, respectively, for some constant  $C> 1$.  Then, we have the following conclusions. 
 \begin{itemize}
 	\item ($M=1$) For $M=1$, i.e., there is only one diverging spike, we report the divergence rates of
 $\varkappa_L$,   $\varkappa_W$ and  $\varkappa_R$ in Table \ref{rateM1}. 
 	From these results, we can conclude that the RLRT is asymptotically more powerful than CLRT and CNTT whenever $\alpha_1\to\infty$.    
Here, one should note that $1+2c+ \beta_{x}c>0$ and $\log(1-c)+c<0$ provided that $c<1$. Moreover, if $\alpha_1=o(n^{1/2})$, then the divergence rate of  $\varkappa_W$ is higher than  $\varkappa_L$. However, when $\alpha_1=\Omega(n^{1/2})$, $\varkappa_L$ could be larger than $\varkappa_W$, such as $n=o(\alpha_1^2)$.
\begin{table}[htbp]
	\caption{Divergence rates of  $\varkappa_L$,   $\varkappa_W$ and  $\varkappa_R$ when $M=1$  }
	\label{table1}
		\begin{tabular}{c|c|c|c}
			\hline 
			&{$\varkappa_L$ }&
			{$\varkappa_W$ }&
		{$\varkappa_R$} \\
			\hline 
						$\alpha_1=o(n^{\frac1 4})$& $\asymp \alpha_1 $ & $\asymp \alpha_1^2 $ &$\simeq\frac{\sqrt{n}}{s_1} $\\
			\hline 
			$\Omega(n^{\frac14})=\alpha_1=o(n^{\frac12})$& $\asymp \alpha_1 $ & $\simeq\frac{\sqrt{n}}{2\sqrt{s_1^2+c^{2}(1+2c+ \beta_{x}c)n/\alpha_1^4}} $ &$\simeq\frac{\sqrt{n}}{s_1} $\\
			\hline 
			$\alpha_1=\Omega(n^{\frac1 2})$& $\simeq\frac{\sqrt{n}}{\sqrt{s_1^2-2(\log(1-c)+c)n/\alpha_1^2}} $   & $\simeq\frac{\sqrt{n}}{2s_1} $ & $\simeq\frac{\sqrt{n}}{s_1} $ \\
			\hline 

		\end{tabular}
\label{rateM1}
\end{table}
 	
 	\item ($M=2$) For  $M=2$, we assume that the two diverging spikes are not equal, i.e., $d_1=d_2=1$. In addition, for convenience of analysis, we assume that the two spikes have the same divergence rate, i.e., $\alpha_2=k_2\alpha_1$ with some $k_2<1$. The results are presented in Table \ref{rateM2}. Since the power function of RLRT is relevant only to the largest spikes and not to the other spikes, $\varkappa_R$ has the same result for $M= 2$ as it does for $M= 1$, that is $\varkappa_R\simeq{\sqrt{n}}/{s_1}$. Thus, we omit the results of $\varkappa_R$ in Table \ref{rateM2}.   We can conclude from Table \ref{rateM2} that RLRT is asymptotically more powerful than CNTT whenever  $\alpha_1\to\infty$, because $1+k_2^2<2$ and $s_2\geq0$. However, for CLRT, if $n=o(\alpha_1^2)$ and $(s_1^2+k_2^2 s_2^2)/(1+k_2)^2<s_1^2$, then $\varkappa_L$ could be larger than $\varkappa_R$. Since $s_{k}^{2}= 2+ \beta_{x}\sum_{t=1}^p|u_{tk}|^4+o(1)$, with suitable values of  $\beta_{x}\geq-2$, $\sum_{t=1}^p|u_{tk}|^4\in[1/p,1]$ and $k_2<1$, the inequality $(s_1^2+k_2^2 s_2^2)/(1+k_2)^2<s_1^2$ can be satisfied, such as choosing $\beta_x=0$. This property indicates that the LSSs could exhibit higher asymptotic power than RLRT statistic in some special scenarios.

 	\begin{table}[htbp]
 \caption{Divergence rates of  $\varkappa_L$ and  $\varkappa_W$  when $M=2$  and $\alpha_2=k_2\alpha_1$ }
 \label{table2}
 \begin{tabular}{c|c|c}
  \hline 
  &{$\varkappa_L$ }&
  {$\varkappa_W$ }\\
  \hline 
  $\alpha_1=o(n^{\frac1 4})$& $\asymp \alpha_1 $ & $\asymp \alpha_1^2 $ \\
  \hline 
  $\Omega(n^{\frac14})=\alpha_1=o(n^{\frac12})$& $\asymp \alpha_1 $ & $\simeq\frac{\sqrt{n}(1+k_2^2)}{2\sqrt{s_1^2+k_2^4s_2^2+c^{2}(1+2c+ \beta_{x}c)n/\alpha_1^4}} $ \\
  \hline 
  $\alpha_1=\Omega(n^{\frac1 2})$& $\simeq\frac{(1+k_2)\sqrt{n}}{\sqrt{s_1^2+k_2^2 s_2^2-2(\log(1-c)+c)n/\alpha_1^2}} $   & $\simeq\frac{\sqrt{n}(1+k_2^2)}{2\sqrt{s_1^2+k_2^4 s_2^2}} $  \\
  \hline 
  
 \end{tabular}
 \label{rateM2}
\end{table}

 \item ($M\geq3$) For  $M\geq3$, the discussion is analogous to the cases of $M=1$ and $M=2$; thus, we omit the details because of space limitations. We merely wish to emphasize here that when $M \geq 3$, CNTT is also potentially asymptotically more powerful than RLRT. Suppose that  $ \alpha_t=k_t\alpha_1$, $t=1,\dots,M$ and  $ 1=k_{1}>k_{2}>\dots>k_{M}>0. $ It is not difficult to find that if $n=o(\alpha_1^4)$, then $\varkappa_W\simeq{\sqrt{n}\sum_{t=1}^Mk_t^2}/{\sqrt{4\sum_{t=1}^Mk_t^4s_t^2}}$, which can be larger than $\varkappa_R\simeq{\sqrt{n}}/s_1$ with suitable values of $k_t$ and $s_t$, $t=1,\dots,M$. For example, if $M=3$, we can choose $k_2$ and $k_3$ close to 1, while choosing $s_2$ and $s_3$ close to 0.  Notably, if $\beta_x=0$, e.g., $\{x_{ij}\}$ are Gaussian, then $M$ must be at least 5 for $\varkappa_W>\varkappa_R$ to asymptotically hold.
 \end{itemize}
For illustration, we present some graphs of the functions $\varkappa_L$, $\varkappa_W$ and $\varkappa_R$ in \eqref{varkappaLWR} with different numbers of  diverging spikes in Figure \ref{figvarkappa}. 

\begin{figure}[htbp]
	\subfigure[$M=1$]{
		\includegraphics[width=4.5cm,height=4cm]{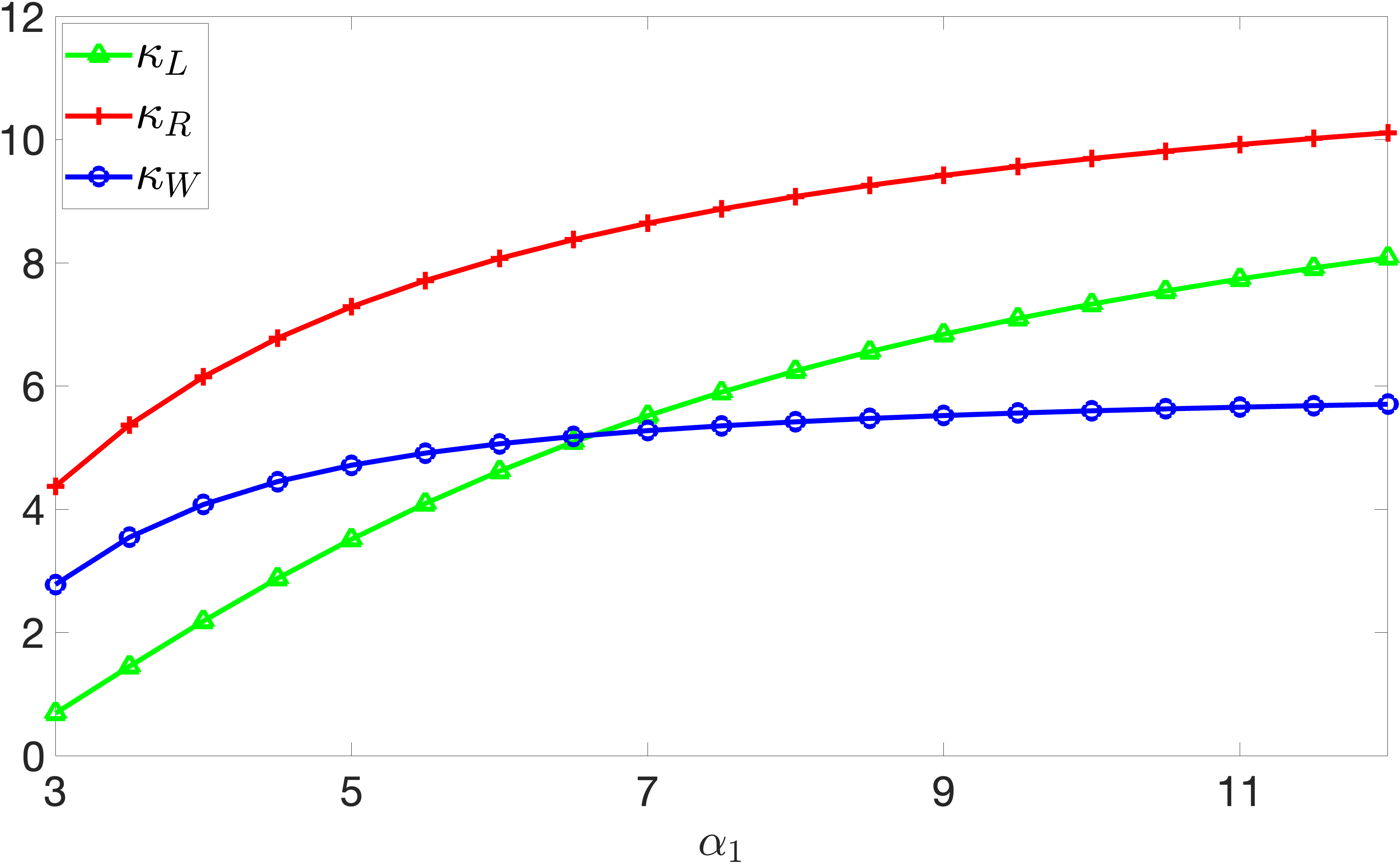}\label{figvarkappa:a}}
\subfigure[$M=2$]{
		\includegraphics[width=4.5cm,height=4cm]{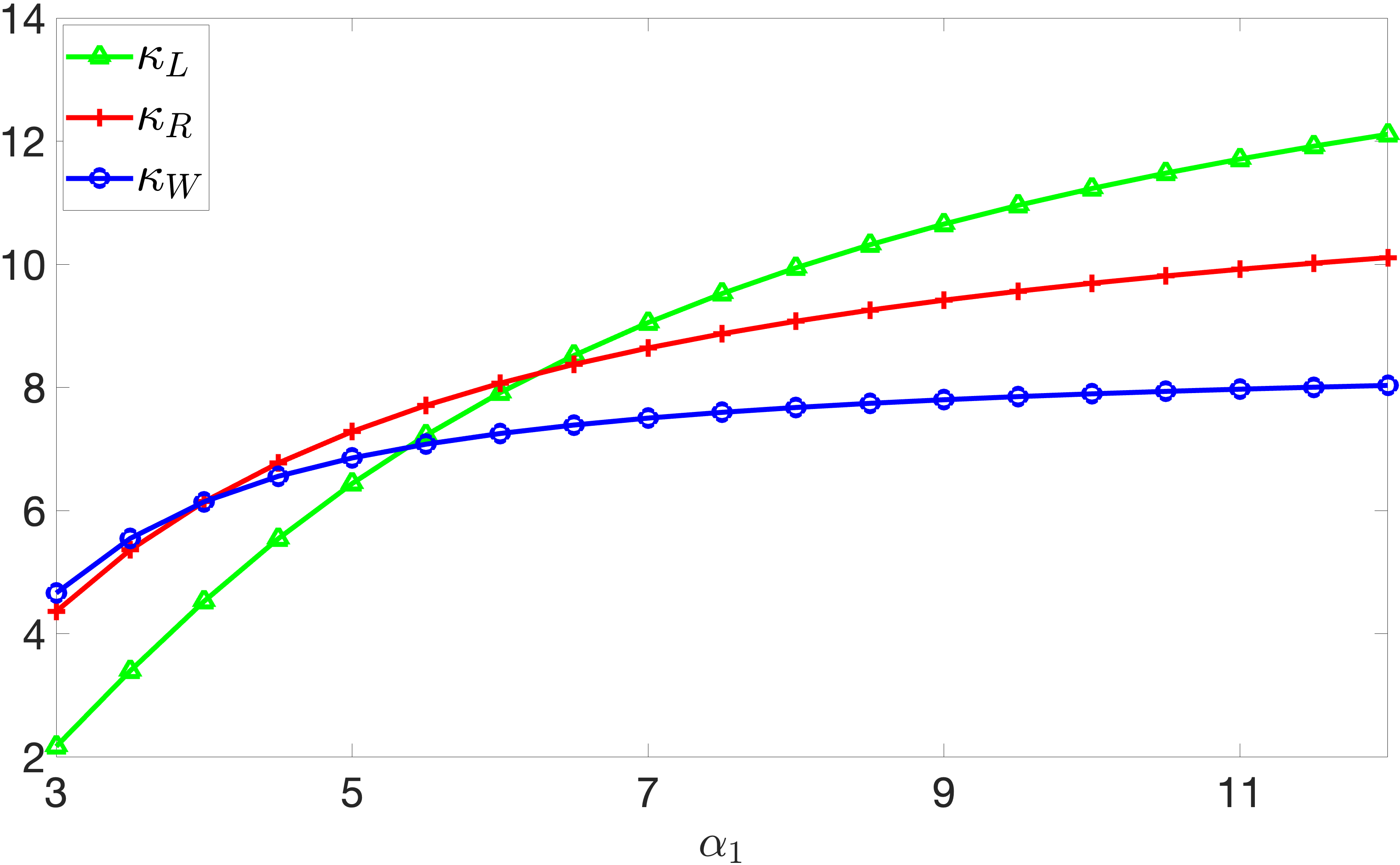}\label{figvarkappa:b}}
		\subfigure[$M=5$]{
			\includegraphics[width=4.5cm,height=4cm]{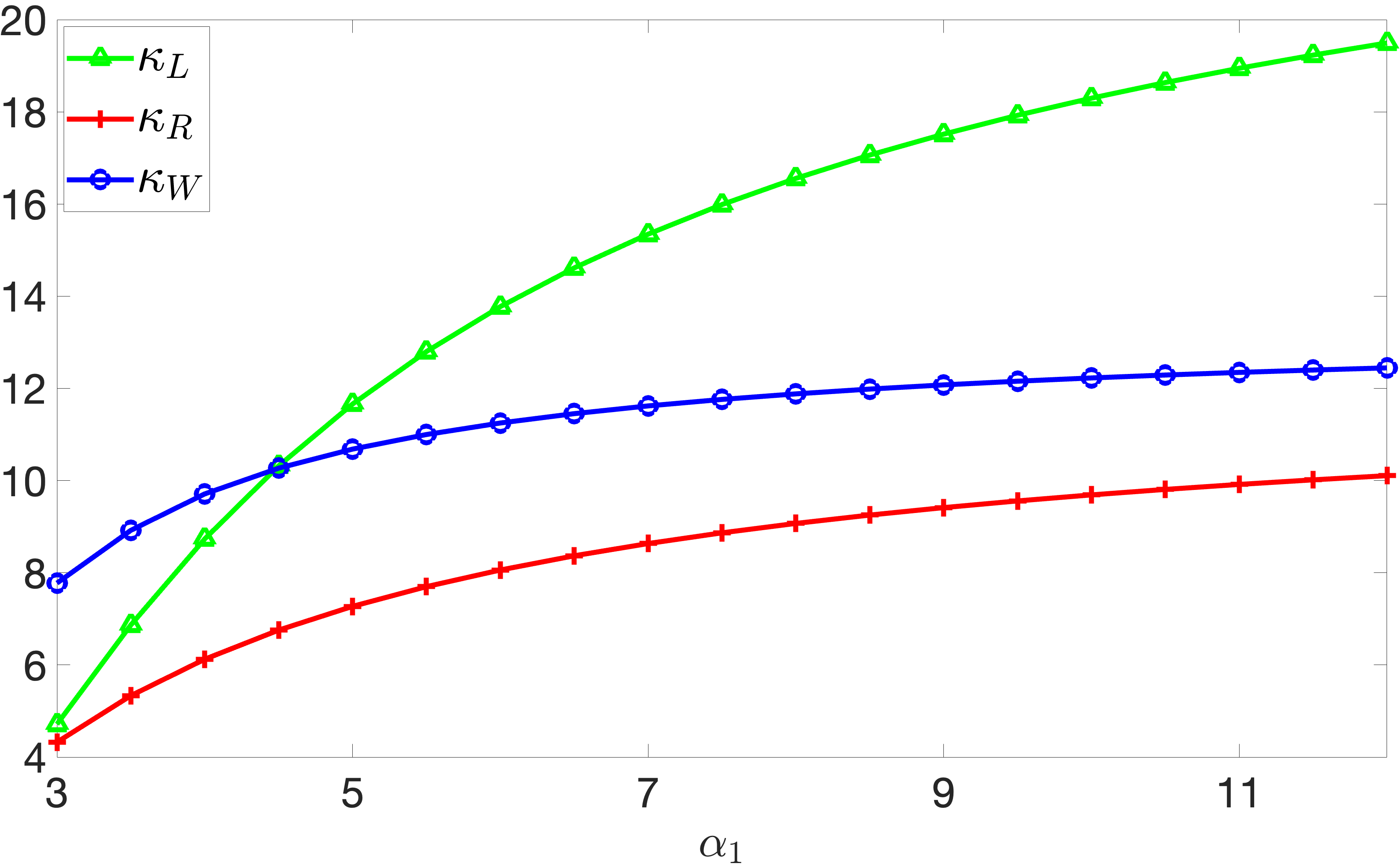}\label{figvarkappa:c}}

			\caption{Graphs of the functions $\varkappa_L$, $\varkappa_W$ and $\varkappa_R$. We fix $ (p,n,\xi,\beta_x)=(100,300,0.05,0)$. The left panel shows the curves for $M=1$.  The middle panel shows the curves for $M=2$ with $\alpha_2=0.9\alpha_1$. The right
			 panel shows the curves for $M=5$ with $\alpha_2=0.9\alpha_1$, $\alpha_3=0.85\alpha_1$, $\alpha_4=0.8\alpha_1$ and
			 $\alpha_5=0.75\alpha_1$.}
			\label{figvarkappa}
\end{figure}
}}

 {\color{black}{
 \section{Numerical studies }\label{simulation}
In this section,  we report short numerical studies as an illustration of our results. 
Our objective in the simulations is to examine the power analysis in subsection \ref{poweranalysis}.

We examine  the following three different distributions of $ x_{ij}: $
\begin{itemize}
	\item[$ Dt_1$:] $\left\lbrace x_{ij}\right\rbrace  $ are i.i.d. samples from a standard Gaussian population. 
	\item[$ Dt_2$:]$\left\lbrace x_{ij}\right\rbrace  $ are i.i.d. samples from  $ Gamma(4,0.5)-2 $.
	\item[$ Dt_3$:] $\left\lbrace x_{ij}\right\rbrace  $ are i.i.d. samples from Uniform population distribution $ U[-\sqrt{3},\sqrt{3}] $.
\end{itemize}
Note that in above settings, $ \beta_{x}=0, \frac{3}{2}, -\frac{6}{5}, $ respectively.

In the current numerical studies, the null hypothesis is defined as 
$ H_0: \bSi=\bbI_{p}. $
For the alternative hypothesis, we adopt the following six population covariance matrix structures:

\begin{itemize}
	\item[$H_1$:] $ \bSi=\bgL_{1}=diag(\al_{1},\underbrace{1,1,\dots,1}_{p-1}) $.
	\item[$H_2$:] $ \bSi=\bgL_{2}=diag(\al_{1},\al_{2},\underbrace{1,1,\dots,1}_{p-2}), \al_2=0.9\al_{1} $.
	\item[$H_3$:] $\bSi=\bgL_{3}=diag(\al_{1},\dots,\al_{5},\underbrace{1,1,\dots,1}_{p-5}) ,\al_2=0.9\al_{1},\al_3=0.85\al_{1},\al_4=0.8\al_{1},\al_5=0.75\al_{1}. $
	\item[$H_4$:] 
$\bSi=\bbU_{0}\bgL_1\bbU_{0}^{*}, $ 
where $ \bbU_{0} $ is the left singular vectors of a $ p\times p $ random matrix  with i.i.d. $ N(0,1)$
 entries
	\item[$H_5$:] 
$ \bSi=\bbU_{0}\bgL_2\bbU_{0}^{*}, $ 
and $ \bbU_{0} $ is defined in $ H_{4}. $
	\item[$H_6$:] $\bSi=\bbU_{0}\bgL_3\bbU_{0}^{*}, $ 
and $ \bbU_{0} $ is defined in $ H_{4}. $
\end{itemize}
Note that
in above settings, $ \bSi $ is diagonal and $\mathcal{U}_{j_{1}j_{1}j_{2}j_{2}}=1$ under $ H_1$--$H_3 $, whereas $ \bSi $ is nondiagonal and $\mathcal{U}_{j_{1}j_{1}j_{2}j_{2}}\asymp1/p$ under $ H_4$--$H_6 $. 

The settings for the significance level $\xi$ are constructed as follows:
$ 0.05,$ $ 0.01, $ and $1\times 10^{-4}.$
The empirical results are obtained based on 10,000 replications with dimension $ p=50,100,200 $, respectively.
We set the RDS $p/n=1/3 $.

In Tables \ref{tablesize}--\ref{tablepowerd3h3l3}, we list the empirical sizes and powers of CLRT, CNTT, and RLRT under different settings.
In the captions of these tables, $ ``(Dt_\ast, H_\ast) " $ stands for the setting $ Dt_\ast, H_\ast $. For the alternative hypothesis, due to space limitations, we present only some selected tables with significant properties in the paper, and the tables for other cases are provided in the supplementary material.
Below are our conclusions based on our simulation studies:
\begin{enumerate}[(1)]
	\item For the null hypothesis, the empirical sizes of CLRT and CNTT are closer to significance levels than that of RLRT overall. 
	This is because the rate of convergence for the largest eigenvalue distributions is slow. This property has been discussed extensively in the literature and we omit the details here. We refer interested readers to \cite{10.1214/aos/1009210544,Johnstone08}.

	\item For the alternative hypothesis: 
	\begin{itemize}
		\item From Table \ref{tablepowerd1h1l1}, it is easy to find that RLRT has the highest asymptotic power under $H_1$. When there are two spikes, as seen  from Table \ref{tablepowerd2h2l1},  CNTT and RLRT have higher asymptotic power than CLRT. As seen from Table \ref{tablepowerd3h3l3}, when there are five spikes, CNTT seems to have the highest asymptotic empirical power. This is consistent with our analysis in subsection \ref{poweranalysis} that when the number of spikes increases, CLRT and CNTT may exhibit higher asymptotic power than RLRT in some scenarios. To be noticed that, in Table \ref{tablepowerd3h3l3}, we only list the powers when $ \xi=1\times10^{-4} $ and consider smaller values of $ \al_{1} $ since the powers of three tests are all equal to 1 when $ \xi=0.05 $ and $ \xi=0.01, $ or when $ \al_1 $ is not small enough,  making comparisons infeasible.
		\item From Tables \ref{tablepowerd1h1l1}--\ref{tablepowerd3h3l3}, we can also find that in each row, the asymptotic power derived under $ (Dt_2,H_\ast) $ is smaller than that under $ (Dt_1,H_\ast) $, and power derived under $ (Dt_3,H_\ast) $ is larger than that under $ (Dt_1,H_\ast). $ This is because when $ \bSi $ is diagonal, a smaller $ \beta_{x} $ may result in higher asymptotic power, corresponding to Theorems \ref{CLRT power} and \ref{CNTT power}.
	\end{itemize} 
\end{enumerate}

\begin{table*}[htbp]
	\caption{Empirical probability of rejecting $ H_{0} $ at significance levels $ 0.05 $ and $ 0.01 $ under assumptions of Gaussian, Gamma, and Uniform distributions}
	\label{tablesize}
	 \begin{tabular}{@{}lrcccccc@{}}
		\hline && \multicolumn{2}{c}{$ Dt_1 $} & \multicolumn{2}{c}{$ Dt_2 $} & \multicolumn{2}{c}{$ Dt_3 $} \\
		\cline{3-8} 
		test &\multicolumn{1}{c}{
			(p,n) } &\multicolumn{1}{c}{$ \xi=0.05 $} & \multicolumn{1}{c}{$ \xi=0.01 $} & \multicolumn{1}{c}{$ \xi=0.05 $} &\multicolumn{1}{c}{$ \xi=0.01 $} & \multicolumn{1}{c}{$\xi=0.05 $} & \multicolumn{1}{c}{$ \xi=0.01 $}\\
		\hline
		$CLRT$  & (50,150)  &  0.0563 & 0.0132 & 0.0545  &0.0123&0.0538&0.0128\\
		&(100,300) & 0.0517  & 0.0105  &0.0567&0.0120&0.0525&0.0102 \\
		&(200,600)  & 0.0543  & 0.0117 &0.0502&0.0110&0.0527&0.0101 \\
		$CNTT$   & (50,150)  & 0.0591  & 0.0138 & 0.0784 &0.0219&0.0522&0.0107\\
		&(100,300)   &  0.0530 & 0.0106 & 0.0633 &0.0175&0.0513&0.0112\\
		&(200,600)   & 0.0530 & 0.0122 & 0.0542 &0.0142&0.0493&0.0117\\
		$RLRT$    & (50,150) & 0.0426  & 0.0089 & 0.1104 &0.0328&0.0215&0.0033\\
		&(100,300)   & 0.0459  & 0.0082 & 0.0969 &0.0260&0.0276&0.0046\\
		&(200,600) & 0.0495 & 0.0095 & 0.0779 &0.0170&0.0279&0.0068\\
		\hline
	\end{tabular}
\end{table*}

\begin{table*}[htbp]
	\caption{Empirical probability of rejecting $ H_{1} $ at significance level $ \xi=0.05 $ under assumptions of Gaussian, Gamma, and Uniform distributions }
	\label{tablepowerd1h1l1}
	\begin{tabular}{@{}lrrrrrrrrrc@{}}
		\hline &&\multicolumn{3}{c}{($ Dt_{1},H_1 $)} & \multicolumn{3}{c}{($ Dt_{2},H_1 $)} & \multicolumn{3}{c}{($ Dt_{3},H_1 $)}\\
		\cline{3-11}
		test & \multicolumn{1}{c}{
			(p,n) } 
		 &\multicolumn{1}{c}{$\alpha_1=3$} & \multicolumn{1}{c}{$\alpha_1=5$} & \multicolumn{1}{c}{$\alpha_1=7$} &
		\multicolumn{1}{c}{$\alpha_1=3$} & \multicolumn{1}{c}{$\alpha_1=5$} & \multicolumn{1}{c}{$\alpha_1=7$} &
		\multicolumn{1}{c}{$\alpha_1=3$} & \multicolumn{1}{c}{$\alpha_1=5$} & \multicolumn{1}{c}{$\alpha_1=7$}
		\\
		\hline
		$CLRT$  & (50,150)  &  0.7236 & 0.9989 & 1  &0.6920&0.9981&1&0.7511&1&1 \\
		&(100,300) & 0.7439  & 0.9999 & 1 &0.7251&0.9998&1&0.7499&1&1 \\
		&(200,600)  & 0.7544  & 0.9999 & 1&0.7468&0.9998&1&0.7632&1&1 \\
		$CNTT$   & (50,150)  & 0.9919  & 1 & 1 &0.9702&1&1&0.9996&1&1\\
		&(100,300)   &  0.9990 & 1 & 1 &0.9898&1&1&1&1&1\\
		&(200,600)   & 0.9998 & 1 & 1 &0.9973&1&1&1&1&1\\
		$RLRT$    & (50,150) & 0.9998  & 1 & 1 &0.9988&1&1&1&1&1\\
		&(100,300)   & 1  & 1 & 1 &1&1&1&1&1&1\\
		&(200,600) & 1 & 1 & 1 &1&1&1&1&1&1\\
		\hline
	\end{tabular}
\end{table*}

\begin{table*}[htbp]
	\caption{Empirical probability of rejecting $ H_{2} $ at significance level $ \xi=0.05 $ under assumptions of Gaussian, Gamma, and Uniform distributions }
	\label{tablepowerd2h2l1}
	\begin{tabular}{@{}lrrrrrrrrrr@{}}
		\hline &&\multicolumn{3}{c}{($ Dt_{1},H_2$)} & \multicolumn{3}{c}{($ Dt_{2},H_2$)} & \multicolumn{3}{c}{($ Dt_{3},H_2 $)}\\
		\cline{3-11}
		test & \multicolumn{1}{c}{
			(p,n) } 
		&\multicolumn{1}{c}{$\alpha_1=3$} & \multicolumn{1}{c}{$\alpha_1=5$} & \multicolumn{1}{c}{$\alpha_1=7$} &
		\multicolumn{1}{c}{$\alpha_1=3$} & \multicolumn{1}{c}{$\alpha_1=5$} & \multicolumn{1}{c}{$\alpha_1=7$} &
		\multicolumn{1}{c}{$\alpha_1=3$} & \multicolumn{1}{c}{$\alpha_1=5$} & \multicolumn{1}{c}{$\alpha_1=7$}
		\\
		\hline
		$CLRT$  & (50,150)  &  0.9811 & 1 & 1&0.9633&1&1&0.9899&1&1  \\
		&(100,300) & 0.9894  & 1 & 1&0.9835&1&1&0.9929&1&1 \\
		&(200,600)  & 0.9928  & 1 & 1& 0.9894&1&1&0.9948&1&1\\
		$CNTT$   & (50,150)  & 1  & 1 & 1 & 1  & 1 & 1& 1  & 1 & 1\\
		&(100,300)   &  1 & 1 & 1 & 1  & 1 & 1& 1  & 1 & 1\\
		&(200,600)   & 1 & 1 & 1 & 1  & 1 & 1& 1  & 1 & 1\\
		$RLRT$    & (50,150) & 1  & 1 & 1 & 1  & 1 & 1& 1  & 1 & 1\\
		&(100,300)   & 1  & 1 & 1 & 1  & 1 & 1& 1  & 1 & 1\\
		&(200,600) & 1 & 1 & 1 & 1  & 1 & 1& 1  & 1 & 1\\
		\hline
	\end{tabular}
\end{table*}

\begin{table*}[htbp]
	\caption{Empirical probability of rejecting $ H_{3} $ at significance level $ \xi=1\times10^{-4} $ under assumptions of Gaussian, Gamma, and Uniform distributions  }
	\label{tablepowerd3h3l3}
	\begin{tabular}{@{}lrrrrrrrrrr@{}}
		\hline &&\multicolumn{3}{c}{($ Dt_{1},H_3 $)} & \multicolumn{3}{c}{($ Dt_{2},H_3 $)} & \multicolumn{3}{c}{($ Dt_{3},H_3 $)}\\
		\cline{3-11} &&
		\multicolumn{9}{c}{ $ \alpha_1 $}\\
		\cline{3-11}
		test & \multicolumn{1}{c}{
			(p,n) } 
		&\multicolumn{1}{c}{$2.2$} & \multicolumn{1}{c}{$2.5$} & \multicolumn{1}{c}{$2.8$} &
		\multicolumn{1}{c}{$2.2$} & \multicolumn{1}{c}{$2.5$} & \multicolumn{1}{c}{$2.8$} &
		\multicolumn{1}{c}{$2.2$} & \multicolumn{1}{c}{$2.5$} & \multicolumn{1}{c}{$2.8$}
		\\
		\hline
		$CLRT$  & (50,150)  &  0.3911 & 0.8636 & 0.9918 &0.3883&0.8232&0.9834&0.3825&0.8945&0.9976 \\
		&(100,300) & 0.3876  & 0.8946 & 0.9985&0.3804&0.8705&0.9951&0.3861&0.9110&0.9985 \\
		&(200,600)  & 0.3779  & 0.9032 & 0.9980&0.3846&0.8946&0.9976 &0.3825&0.9134&0.9997\\
		$CNTT$   & (50,150)  & 0.9972  & 1 & 1 &0.9972&0.9998&1&0.9999&1&1\\
		&(100,300)   &  0.9997 & 1 & 1 &0.9917&1&1&1&1&1\\
		&(200,600)   & 1 & 1 & 1 &0.9964&1&1&1&1&1\\
		$RLRT$    & (50,150) & 0.8933  & 0.9968 & 1 &0.8954&0.9943&0.9999&0.8909&0.9994&1\\
		&(100,300)   & 0.9821  & 1 & 1 &0.9756&0.9999&1&0.9904&1&1\\
		&(200,600) & 0.9996 & 1 & 1 &0.9989&1&1&1&1&1\\
		\hline
	\end{tabular}
\end{table*}}}

\section{Technical proofs}\label{section 6}
In this section, we present some lemmas that are needed in the proofs of the main results. The truncation and renormalization are postponed to the end of this paper. 

\subsection{Some primary definitions and lemmas}
In this section, we provide some useful results that are used later in the proofs of Theorems \ref{thm1} and  \ref{thm2}. For the population covariance matrix $\bSi=\bbT\bbT^{\ast}$, we consider the corresponding sample covariance matrix $ \bbB=\bbT\bbS_{x}\bbT^{\ast} $, where $ \bbS_{x}=\frac{1}{n}\bbX\bbX^{\ast}  $. 
By singular value decomposition of $ \bbT $ (see \eqref{decT}),
\begin{equation*}
	\bbB=\bbV\left(\begin{array}{cc}
		\bbD_{1}^\frac{1}{2}\bbU_{1}^{\ast}\bbS_{x}\bbU_{1}\bbD_{1}^\frac{1}{2},  &\bbD_{1}^\frac{1}{2}\bbU_{1}^{\ast}\bbS_{x}\bbU_{2}\bbD_{2}^\frac{1}{2}\\
		\bbD_{2}^\frac{1}{2}\bbU_{2}^{\ast}\bbS_{x}\bbU_{1}\bbD_{1}^\frac{1}{2},	&\bbD_{2}^\frac{1}{2}\bbU_{2}^{\ast}\bbS_{x}\bbU_{2}\bbD_{2}^\frac{1}{2}
	\end{array} \right)\bbV^{\ast}.
\end{equation*}
Note that
\begin{equation*}
	\bbS=\left(\begin{array}{cc}
		\bbD_{1}^\frac{1}{2}\bbU_{1}^{\ast}\bbS_{x}\bbU_{1}\bbD_{1}^\frac{1}{2},  &\bbD_{1}^\frac{1}{2}\bbU_{1}^{\ast}\bbS_{x}\bbU_{2}\bbD_{2}^\frac{1}{2}\\
		\bbD_{2}^\frac{1}{2}\bbU_{2}^{\ast}\bbS_{x}\bbU_{1}\bbD_{1}^\frac{1}{2},	&\bbD_{2}^\frac{1}{2}\bbU_{2}^{\ast}\bbS_{x}\bbU_{2}\bbD_{2}^\frac{1}{2}
	\end{array} \right)	\triangleq\left(\begin{array}{cc}\bbS_{11},&\bbS_{12}\\\bbS_{21},&\bbS_{22}\end{array} \right). 
\end{equation*}
Moreover, $ \bbB $ and $ \bbS $ have the same eigenvalues.

Recall that $\underline{\bbB}=\frac{1}{n}\bbX^{\ast}\bbT^{\ast}\bbT\bbX $ (the spectrum of which differs from that of $ \bbB $ by $ \left|  n-p\right|  $ zeros). Its LSD is $ \underline{F}^{c,H} $, $\underline{F}^{c,H}\equiv\left(1-c \right)\mathbbm{1}_{\left[0,\infty \right) }+cF^{c,H}   $, and its Stieltjes transform is $ \underline{m}\left(z \right)  $.
Let $ \widetilde{\bgl}_{j} $ be the eigenvalues of $ \bbS_{22} $ so that the LSS of $ \bbS_{22} $ is $ \sum_{j=1}^{p-M}f( \widetilde{\bgl}_{j}) $. Correspondingly, recall that $ c_{nM}:=\frac{p-M}{n} $, $ H_{2n}:=F^{\bbD_{2}} $ and $ m_{2n0}:=m_{2n0}(z) $ is the 
Stieltjes transform
 of $ F^{c_{nM},H_{2n}} $. 
 First, in Lemma \ref{lemma1} we derive that the difference between the two centers is $ 0 $.

\begin{lemma}\label{lemma1}
	Under Assumptions \ref{ass1}--\ref{ass4},  $$ \left( p-M\right)\int f\left(x \right)dF^{c_{nM},H_{2n}} =p\int f\left( x\right)dF^{c_{n},H_{n}}\left( x\right). $$
	
\end{lemma}
\begin{proof}
	By the Cauchy integral formula, 
	$$ p\int f(x)dF^{c_{n},H_{n}}=-\frac{p}{2 \pi i}\oint_\mathcal{C}f(z)m_{1n0}dz=-\frac{n}{2 \pi i}\oint_\mathcal{C}f(z)\underline{m}_{1n0}dz, $$ $$ (p-M)\int f(x)dF^{c_{nM},H_{2n}}=-\frac{p-M}{2 \pi i}\oint_\mathcal{C}f(z)m_{2n0}dz=-\frac{n}{2 \pi i}\oint_\mathcal{C}f(z)\underline{m}_{2n0}dz, $$
	where $\underline{m}_{1n0} $ and $ \underline{m}_{2n0}$ are the Stieltjes transforms of $\underline{F}^{c_{n},H_{n}} $ and $ \underline{F}^{c_{nM},H_{2n}}$, respectively.
	Then,  $$ (p-M)\int f(x)dF^{c_{nM},H_{2n}}-p\int f(x)dF^{c_{n},H_{n}}=\frac{n}{2 \pi i}\oint_\mathcal{C}f(z)\left(\underline{m}_{1n0}-\underline{m}_{2n0}\right) dz. $$
	Next, we prove that $\underline{m}_{1n0}=\underline{m}_{2n0}$.
	
	Note that
	$ m_{1n0} $ and $ m_{2n0} $ are the unique solutions to
	\begin{align}
		z=-\frac{1}{\underline{m}_{1n0}}+c_{n}\int\frac{tdH_{n}\left( t\right) }{1+t\underline{m}_{1n0}}\label{1}\\     
		z=-\frac{1}{\underline{m}_{2n0}}+c_{nM}\int\frac{tdH_{2n}\left( t\right) }{1+t\underline{m}_{2n0}} \label{2}, 
	\end{align}
	respectively,
	where $ \underline{m}_{1n0}=-\frac{1-c_{n}}{z}+c_{n}m_{1n0} $ and $ \underline{m}_{2n0}=-\frac{1-c_{nM}}{z}+c_{nM}m_{2n0} $. 
	Since $$  H_{n}(t)=\frac{1}{p}\left[\sum_{i=1}^{M}\mathbbm{1}_{\left\lbrace 0\leq t\right\rbrace }+\sum_{i=M+1}^{p}\mathbbm{1}_{\left\lbrace \alpha_{i}\leq t\right\rbrace } \right]=\frac{M}{p}+\frac{1}{p}\sum_{i=M+1}^{p}\mathbbm{1}_{\left\lbrace \alpha_{i}\leq t\right\rbrace }  $$ and $ H_{2n}(t)=\frac{1}{p-M}\sum_{i=M+1}^{p}\mathbbm{1}_{\left\lbrace\alpha_{i}\leq t \right\rbrace } $,
	  (\ref{1}) can be written as
	\begin{align}
		\nonumber z&=-\frac{1}{\underline{m}_{1n0}}+\frac{p}{n}\int\frac{td\left( \frac{M}{p}+\frac{1}{p}\sum_{i=M+1}^{p}\mathbbm{1}_{\left\lbrace \alpha_{i}\leq t\right\rbrace }\left( t\right)\right)  }{1+t\underline{m}_{1n0}}=-\frac{1}{\underline{m}_{1n0}}+\frac{p}{n}\int\frac{td\left( \frac{1}{p}\sum_{i=M+1}^{p}\mathbbm{1}_{\left\lbrace \alpha_{i}\leq t\right\rbrace }\left( t\right)\right)  }{1+t\underline{m}_{1n0}}\\
		&=-\frac{1}{\underline{m}_{1n0}}+\frac{1}{n}\sum_{i=M+1}^{p}\frac{\alpha_{i}}{1+\al_{i}\underline{m}_{1n0}}.\label{3}
	\end{align}	
	Similarly, equation (\ref{2}) can be written as 
$
		z=-\frac{1}{\underline{m}_{2n0}}+\frac{1}{n}\sum_{i=M+1}^{p}\frac{\al_{i}}{1+\al_{i}\underline{m}_{2n0}}.\label{4}
$
Thus, according to the fact that $ m_{1n0} $ and $ m_{2n0} $ are the unique solutions of \eqref{3} and \eqref{4}, respectively, we have $ m_{1n0} = m_{2n0} $, which completes the proof of this lemma.

\end{proof}

\textcolor{black}{Note that for the bounded part of the LSS, $ \sum_{j=M+1}^{p}f\left(\lambda_{j} \right)  $, the BST cannot be used directly since it is not an LSS of a sample covariance matrix. In fact, it approximates the LSS of $ \bbS_{22} $, that is $ \sum_{j=1}^{p-M} f\left( \tilde{\lambda}_{j}\right)  $,  but they are not equal since the off-diagonal blocks of the sample covariance matrix are not null. The following lemma measures the difference between $ \sum_{j=M+1}^{p}f\left(\lambda_{j} \right)  $ and $ \sum_{j=1}^{p-M} f\left( \tilde{\lambda}_{j}\right)  $.  }\begin{lemma}\label{lemma2} \textit{Under Assumptions \ref{ass1}--\ref{ass4}}, 
$$\sum_{j=M+1}^{p}f\left( \lambda_{j}\right)   - \sum_{j=1}^{p-M}f\left( \widetilde{\bgl}_{j}\right)- \frac{M}{2\pi i}\oint_\mathcal{C}f\left(z \right)\frac{\underline{m}_{2n0}^{'}(z)}{\underline{m}_{2n0}(z)}dz=o_p(1).$$
\end{lemma}
\begin{proof}
Note that
$
		L_{1}:=\sum_{j=M+1}^{p}f\left( \lambda_{j}\right).
$
	By the Cauchy integral formula, we have
$
	L_{1}=-\frac{p}{2\pi i}\oint_{\mathcal C}f\left(z \right)m_{n}\left(z \right)dz,
$
	where $m_{n}=\frac{1}{p}\mathrm{tr}\left(\bbS-z\bbI_{p} \right)^{-1}=\frac{1}{p}\mathrm{tr}\left(\bbB-z\bbI_{p} \right)^{-1} $. Analogously, we have
$
		L_{2}:=\sum_{j=1}^{p-M}f\left( \widetilde{\bgl}_{j}\right)=-\frac{p-M}{2\pi i}\oint_{\mathcal C}f\left(z \right)m_{2n}\left(z \right)dz,	
$
	where $m_{2n}=\frac{1}{p-M}\mathrm{tr}\left(\bbS_{22}-z\bbI_{p-M} \right)^{-1}$. By applying the block matrix inversion formula to $m_n$, we can obtain
	\begin{align}\label{L12}
		L_{1}-L_{2}=-\frac{1}{2\pi i}\oint_{\mathcal C}f\left(z \right)\left(T_{1}-T_{2} \right) dz,
	\end{align}
	where 
	\begin{align*}
		T_{1} &=\mathrm{tr}\left( \bbS_{11}-z\bbI_{M}-\bbS_{12}\left( \bbS_{22}-z\bbI_{p-M}\right)^{-1}\bbS_{21} \right)^{-1},\\ 	
		T_{2} &=-\mathrm{tr}\left[\left( \bbS_{11}-z\bbI_{M}-\bbS_{12}\left(\bbS_{22}-z\bbI_{p-M} \right)^{-1}\bbS_{21} \right)^{-1}\bbS_{12}\left(\bbS_{22}-z\bbI_{p-M} \right)^{-2}\bbS_{21}\right].   
	\end{align*}
	Note that for any matrix $\bbZ$, 
	\begin{align*}
		\bbZ\left(\bbZ^{\ast}\bbZ-\lambda \bbI \right)^{-1}\bbZ^{\ast}=\bbI+\lambda\left(\bbZ\bbZ^{\ast}-\lambda \bbI \right)^{-1},
	\end{align*}	
	which, together with the notation $\bUps_n:=\frac{1}{n}\bbD_{1}^{\frac{1}{2}}\bbU_{1}^{\ast}\bbX\left(\frac{1}{n}\bbX^{\ast}\bbU_{2}\bbD_{2}\bbU_{2}^{*}\bbX-z\bbI_{n}\right) ^{-1}  \bbX^{\ast}\bbU_{1}\bbD_{1}^{\frac{1}{2}} $, implies that
	\begin{align*}
		 T_{1}&=-z^{-1}\mathrm{tr}\left(\bbI_{M}+\bUps_n \right)^{-1}\\
		T_{2} &=z^{-1}\mathrm{tr}\left[\left(\bbI_{M}+\bUps_n \right)^{-1}\bbS_{12}\left(\bbS_{22}-z\bbI_{p-M} \right)^{-2}\bbS_{21}\right]
	\end{align*}
	$ \underline{m}_{2n}= \underline{m}_{2n}(z)$ denotes the Stieltjes transform of $ F^{\frac{1}{n}\bbX^{\ast}\bbU_{2}\bbD_{2}\bbU_{2}^{*}\bbX} $. Thus,  we have that $\underline{m}_{2n}(z)-\underline{m}(z)=o_p(1)$ for any $z\in\mathcal{C}$. 
	From Theorem 3.1 of \citep{JiangB21G}, we know that
	\begin{align} 
		\frac{1}{n}\bbU_{1}^{\ast}\bbX\left(\frac{1}{n}\bbX^{\ast}\bbU_{2}\bbD_{2}\bbU_{2}^{*}\bbX-z\bbI_{n}\right) ^{-1}\bbX^{\ast}\bbU_{1}
		=\underline{m}_{2n}\left(z \right)\bbI_{M}+O_{p}(n^{-\frac{1}{2}}).
		\label{6}
	\end{align} 
Thus, under Assumption \ref{ass3}, we find that
\begin{align} 
	\bbD_{1}^{1/2}\left(\bbI_{M}+\bUps_n \right)^{-1}\bbD_{1}^{1/2}=\frac{1}{\underline{m} \left(z \right)}\bbI_{M}+o_p(1),
	\label{7}
\end{align}
which yields
\begin{align}\label{T1op1}
	T_{1}=o_p(1).
\end{align}

It follows that
\begin{align} 
	&\bbD_{1}^{-1/2}\bbS_{12}\left(\bbS_{22}-z\bbI_{p-M} \right)^{-2}\bbS_{21}\bbD_{1}^{-1/2}\nonumber\\
	\nonumber=&\frac{1}{n}\mathrm{tr}\left[ \left(\bbS_{22}-z\bbI_{p-M} \right)^{-2}\bbS_{22}\right]\bbI_{M}+O_{p}(n^{-\frac{1}{2}}) \\
	\nonumber=&\frac{1}{n}\mathrm{tr}\left(\bbS_{22}-z\bbI_{p-M} \right)^{-1}\bbI_{M}+\frac{z}{n}\mathrm{tr}\left(\bbS_{22}-z\bbI_{p-M} \right)^{-2}\bbI_{M}+O_{p}(n^{-\frac{1}{2}})\\
	=& cm_{2n}\left(z \right)\bbI_{M}+zcm_{2n}'\left(z \right)\bbI_{M}+o_p(1)\nonumber\\
	=& cm_{2n0}\left(z \right)\bbI_{M}+zcm_{2n0}'\left(z \right)\bbI_{M}+o_p(1)\nonumber\\
	=&\underline{m}_{2n0}\left(z \right)\bbI_{M}+z\underline{m}_{2n0}'\left(z \right)\bbI_{M}+o_p(1), \label{8} 
\end{align}
where the last equality is derived from $ \underline{m}_{2n0}(z)=-\frac{1-c}{z}+cm_{2n0}\left(z \right)  $. Therefore, according to (\ref{7}) and (\ref{8}), we obtain
\begin{align*}
	T_{2}=M\frac{\underline{m}_{2n0}\left(z \right)+z\underline{m}_{2n0}'\left(z \right)}{z\underline{m}_{2n0}\left(z \right)}+o_p(1),
\end{align*}
which, together with \eqref{L12} and \eqref{T1op1}, implies that
\begin{align*}
	L_{1}-L_{2}=\frac{M}{2\pi i}\oint_{\mathcal C}f\left(z \right)\frac{\underline{m}_{2n0}(z)+z\underline{m}_{2n0}^{'}(z)}{z\underline{m}_{2n0}(z)}dz+o_p(1)=\frac{M}{2\pi i}\oint_{\mathcal C}f\left(z \right)\frac{\underline{m}_{2n0}^{'}(z)}{\underline{m}_{2n0}(z)}dz+o_p(1).
\end{align*}
Therefore, the proof of this lemma is complete.
\end{proof}

Define random vector $\boldsymbol\gamma_{k}=\left( \gamma_{kj} \right)^\top = \left( \sqrt{n} \frac{\lambda_{j}-\phi_{n}\left(\al_{k} \right) }{\phi_{n}\left(\al_{k} \right)},    j\in J_{k} \right)^\top$, where $ J_{k} $ is the indicator set of a packet of $ d_{k} $ consecutive sample eigenvalues. Then, we present the following lemma, which is borrowed from \cite{JiangB21G} and characterizes the limiting distribution of the spiked eigenvalues of the sample covariance matrix. 
\begin{lemma} (\cite{JiangB21G})\label{lemma4}
	 Under Assumptions \ref{ass1}--\ref{ass4}, \textit{ random vector} $ \boldsymbol\gamma_{k} $ \textit{converges weakly to the joint distribution of $ d_{k} $ eigenvalues of a Gaussian random matrix} $$ -\frac{1}{\theta_{k}}\left[\bgO_{\phi_{k}} \right]_{kk},   $$
	\textit{where} 
	$$\theta_{k}=\phi_{k}^{2}\underline{m}_{2}\left( \phi_{k}\right), ~~\underline{m}_{2}\left( \lambda\right)=\int\frac{1}{\left( \lambda-x\right) ^{2}}d\underline{F}^{c,H}\left( x\right)   $$ \textit{with} $ \underline{F}^{c,H} $ \textit{being the} LSD \textit{of matrix} $ n^{-1}\bbX^{\ast}\bbU_{2}\bbD_{2}\bbU_{2}^{*}\bbX,$ $ \phi_{k}=\al_{k}\left(1+c\int\frac{t}{\al_{k}-t}dH\left(t \right)  \right)  $.
	$ \left[\bgO_{\phi_{k}} \right]_{kk} $ \textit{is the} $ k $\textit{th diagonal block of matrix} $ \bgO_{\phi_{k}} $. \textit{The variances and covariances of the elements} $ \omega_{ij} $ \textit{of} $ \bgO_{\phi_{k}} $ \textit{are:}
	$$
	\operatorname{Cov}\left(\omega_{i_{1}, j_{1}}, \omega_{i_{2}, j_{2}}\right)=\left\{\begin{array}{cc}
		(\al_{x}+1) \theta_{k}+\beta_{x}\mathcal{U}_{ i i i i} \nu_{k}, & i_{1}=j_{1}=i_{2}=j_{2}=i \\
		\theta_{k}+\beta_{x}\mathcal{U}_{ i j i j} \nu_{k}, & i_{1}=i_{2}=i \neq j_{1}=j_{2}=j \\
		\beta_{x}\mathcal{U}_{ i_{1} j_{1} i_{2} j_{2}} \nu_{k}, & \text { other cases }
	\end{array}\right.
	$$
	\textit{where} $\beta_{x}\mathcal{U}_{ i_{1} j_{1} i_{2} j_{2}}=\sum_{t=1}^{p} \bar{u}_{t i_{1}} u_{t j_{1}} u_{t i_{2}} \bar{u}_{t j_{2}} \beta_{x}$, $\boldsymbol u_{i}=\left(u_{1 i}, \ldots, u_{p i}\right)^{\top}$ \textit{are the} $i$ \textit{th column of the matrix} $\mathbf{U}_{1}$, $\nu_{k}=\phi_{k}^{2} \underline{m}^{2}\left(\phi_{k}\right)$.
\end{lemma}

Recall that $ \lambda_{j} $ are the eigenvalues of $ \bbB $, and $ \widetilde{\bgl}_{j} $ are the eigenvalues of $ \bbS_{22} $. The following lemma shows the asymptotic independence between $ \sum_{j=1}^{M}f\left( \lambda_{j}\right) $ and $ \sum_{j=1}^{p-M}f\left(\widetilde{\bgl}_{j} \right) $.

\begin{lemma}\label{lemma5} Under Assumptions \ref{ass1}--\ref{ass4},  $ \sum_{j=1}^{M}f\left( \lambda_{j}\right) $ and $ \sum_{j=1}^{p-M}f\left(\widetilde{\bgl}_{j} \right) $  are asymptotically independent.
\end{lemma}
\begin{proof} 
	It is sufficient to prove that for a given $ \sum_{j=1}^{p-M}f\left( \widetilde{\lambda}_{j}\right) $, the asymptotic limiting distribution of $ \sum_{j=1}^{M}f\left( \lambda_{j}\right) $ does not depend on the random part of $ \sum_{j=1}^{p-M}f\left( \widetilde{\lambda}_{j}\right) $, that is, it only depends on its limit. 
	
	First, we consider $ f(x)=x $.
	From the proof of Theorem 3.1 in \cite{JiangB21G}, \textcolor{black}{we have the following determinant equation}
	$$ 0=\left| \left[\bgO_{M}\left( \phi_{k}\right) \right] _{kk}+\mathrm{ lim}\gamma_{kj}\left\lbrace \phi_{k}^{2}\underline m_{2}(\phi_{k})\right\rbrace \bbI_{d_{k}}\right|,$$ 
	where $ \bgO_{M}\left( \phi_{k}\right) $ 
	\begin{align*}
=\frac{\phi_{k}}{\sqrt{n}}\left[ \mathrm{tr}\left\lbrace \left(  \phi_{k}\bbI-\frac{1}{n}\bbX^{*}\bbU_{2}\bbD_{2}\bbU_{2}^{*} \bbX\right) ^{-1}\right\rbrace \bbI-\bbU_{1}^{*}\bbX\left(\phi_{k}\bbI-\frac{1}{n}\bbX^{*}\bbU_{2}\bbD_{2}\bbU_{2}^{*} \bbX\right) ^{-1}\bbX^{*}\bbU_{1}\right],
\end{align*}
and
 $\underline m_{2}(\phi_{k}) $ is the limit of $ \mathrm{tr} \left(  \phi_{k}\bbI-\frac{1}{n}\bbX^{*}\bbU_{2}\bbD_{2}\bbU_{2}^{*} \bbX\right) ^{-2}$. Then, we know that $ \gamma_{kj} $ has the same asymptotic distribution with eigenvalues of
	$ -\frac{\left[\bgO_{M}\left( \phi_{k}\right) \right] _{kk}}{\phi_{k}^{2}\underline m_{2}(\phi_{k})} $ in order. \textcolor{black}{From \cite{JiangB21G}, we can obtain that the limiting distribution of $ \bm\Omega_{M}\left( \phi_{k}\right) $ does not change if $ \bbU_{2}^{*}\bbX $ is replaced by $ \bbU_{2}^{*}\bbY $ while $ \bbU_{1}^{*}\bbX $ remains unchanged. Here $ \bbY $ and $ \bbX $ are i.i.d.. Therefore  in  $\bm\Omega_{M}\left( \phi_{k}\right)$, we can assume that $  \bbU_{1}^{*} \bbX$ and $ \bbU_{2}^{*} \bbX $ are independent without loss of generality.}
 Then, given $ \bbU_{2}^{*}\bbX $, the limiting distribution of $ \gamma_{kj}  $ only depends on the limit of $\mathrm{tr}\left(  \phi_{k}\bbI-\frac{1}{n}\bbX^{*}\bbU_{2}\bbD_{2}\bbU_{2}^{*} \bbX\right) ^{-1} $, that is, $ \underline{m}_{2}(\phi_{k}) $, and has nothing to do with the random part. Therefore, it is found that $ \sum_{j=1}^{M}f\left( \lambda_{j}\right) $  and $ \sum_{j=1}^{p-M}f\left( \widetilde{\lambda}_{j}\right) $ are asymptotically independent when $ f(x)=x. $ 
	
	When $ f(x)\neq x $, 
	by using the Newton-Leibniz formula, we have 
\begin{align}
		&\nonumber\sum_{j=1}^{M}f\left(\lambda_{j}\right)-\sum_{k=1}^{K}d_{k}f\left(\phi_{n}\left(\bbalp_{k} \right) \right)
		=\nonumber\sum_{k=1}^{K}\sum_{ j\in J_{k}}(f\left(\lambda_{j}\right) - f\left(\phi_{n}\left(\bbalp_{k} \right) \right))  \\
		&=\nonumber\sum_{k=1}^{K}\sum_{j\in J_{k}}\int_{0}^{\frac{\phi_{n}\left(\bbalp_{k} \right)}{\sqrt{n}}\gamma_{kj}}f'\left(t+\phi_{n}\left(\bbalp_{k} \right) \right)dt\\
		&=\nonumber\sum_{k=1}^{K}\sum_{j\in J_{k}}\int_{0}^{1}\frac{\phi_{n}\left(\bbalp_{k} \right)}{\sqrt{n}}\gamma_{kj}\frac{f'\left(\phi_{n}\left(\bbalp_{k} \right)\left(1+\frac{\gamma_{kj}}{\sqrt{n}}s\right)  \right)}{f'\left(\phi_{n}\left(\bbalp_{k} \right)\right) }f'\left( \phi_{n}\left(\bbalp_{k} \right)\right)ds\\
		&=\nonumber\sum_{k=1}^{K}\sum_{j\in J_{k}}\int_{0}^{1}\varpi_{nk}\gamma_{kj}\frac{f'\left(\phi_{n}\left(\bbalp_{k} \right)\left(1+\frac{\gamma_{kj}}{\sqrt{n}}s\right)  \right)}{f'\left(\phi_{n}\left(\bbalp_{k} \right)\right) }ds \\
		&\rightarrow\sum_{k=1}^{K}\sum_{j\in J_{k}}\int_{0}^{1}\gamma_{kj}\varpi_{nk} ds
		=\sum_{k=1}^{K}\sum_{j\in J_{k}}\varpi_{nk}\gamma_{kj},    \label{21}
	\end{align}
	where \eqref{21} is true due to Assumption \ref{ass4}, and we denote $ \varpi_{nk}=\dfrac{\phi_{n}(\al_{k})}{\sqrt{n}}f^{\prime}(\phi_{n}(\al_{k}) )$. Thus, we convert the above equation into a function of $ \gamma_{kj} $. \textcolor{black}{The above calculations represent the underlying idea of the generalized delta method we mentioned in the Introduction.}
	Since we have proven above that given $ \sum_{j=1}^{p-M}f\left(\widetilde{\lambda}_{j}\right) $, the limiting distribution of $ \gamma_{kj} $ is concerned only with the limit of $ \sum_{j=1}^{p-M}f\left( \widetilde{\lambda}_{j}\right) $, as is $ \sum_{k=1}^{K}\sum_{j\in J_{k}}\varpi_{nk}\gamma_{kj} $, accordingly, we can conclude that $ \sum_{j=1}^{M}f\left( \lambda_{j}\right) $  and $ \sum_{j=1}^{p-M}f\left( \widetilde{\lambda}_{j}\right) $ are asymptotically independent. The proof is complete.
\end{proof}

In the following lemma, we derive the asymptotic distribution of the LSS generated from submatrix $ \bbS_{22} $.

\begin{lemma}\label{lemma3}
	Define $ Q_{1}=\sum_{j=1}^{p-M}f_{1}(\widetilde{\lambda}_{j}) -(p-M)\int f_{1}(x)dF^{c_{nM},H_{2n}}$; then, under Assumptions \ref{ass1}--\ref{ass4}, we have $$ \kappa_{1}^{-1}\left(Q_{1}-\mu_{1} \right)\stackrel{d}\longrightarrow N\left(0,1 \right)  $$
	with mean function
	\begin{align*}
	\mu_{1}&=-\frac{\alpha_{x}}{2 \pi i}\cdot\oint_{\mathcal{C}}f_{1}(z)\frac{  c_{nM} \int \underline{m}_{2n0}^{3}(z)t^{2}\left(1+t \underline{m}_{2n0}(z)\right)^{-3} d H_{2n}(t)}{\left(1-c_{nM} \int \frac{\underline{m}_{2n0}^{2}(z) t^{2}}{\left(1+t \underline{m}_{2n0}(z)\right)^{2}} d H_{2n}(t)\right)\left(1-\alpha_{x} c_{nM} \int \frac{\underline{m}_{2n0}^{2}(z) t^{2}}{\left(1+t \underline{m}_{2n0}(z)\right)^{2}} d H_{2n}(t)\right) }dz \\
		&-\frac{\beta_{x}}{2 \pi i} \cdot \oint_{\mathcal{C}} f_{1}(z) \frac{c_{nM} \int \underline{m}_{2n0}^{3}(z) t^{2}\left(1+t \underline{m}_{2n0}(z)\right)^{-3} d H_{2n}(t)}{1-c_{nM} \int \underline{m}_{2n0}^{2}(z) t^{2}\left(1+t \underline{m}_{2n0}(z)\right)^{-2} d H_{2n}(t)} dz,  	
	\end{align*}
	and the covariance function is   
	\begin{align*}
		\kappa_{1}^{2}=-\frac{1}{4\pi^{2}}\oint_{\mathcal{C}_{1}}\oint_{\mathcal{C}_{2}}f_{1}\left(z_{1} \right)f_{1}\left(z_{2} \right)\vartheta_{n}^{2}dz_{1}dz_{2},
	\end{align*}
where 
$\vartheta_{n}^{2}=\Theta_{0,n}(z_{1},z_{2})+\al_{x}\Theta_{1,n}(z_{1},z_{2})+\beta_{x}\Theta_{2,n}(z_{1},z_{2}),$ 
\begin{align*}
	\Theta_{0,n}(z_{1},z_{2})&=\dfrac{\underline{m}_{2n0}^{\prime}(z_{1}) \underline{m}_{2n0}^{\prime}(z_{2})  }{(\underline{m}_{2n0}(z_{1})-\underline{m}_{2n0}(z_{2}))^{2}  }-\dfrac{1}{(z_{1}-z_{2})^{2}},\\
	\Theta_{1,n}(z_{1},z_{2})&=\frac{\partial}{\partial z_{2}}\left\lbrace \dfrac{\partial \mathcal{A}_{n}(z_{1},z_{2})}{\partial z_{1}}\dfrac{1}{1-\al_{x}\mathcal{A}_{n}(z_{1},z_{2})} \right\rbrace ,\\
	\mathcal{A}_{n}(z_{1},z_{2})&=\dfrac{z_{1}z_{2}}{n}\underline{m}_{2n0}(z_{1}) \underline{m}_{2n0}(z_{2})\mathrm{tr}{\bGma^{*}\bbP_{n}(z_{1})\bGma\bGma^{\top}\bbP_{n}(z_{2})^{\top} \bar{\bGma}},\\
	\Theta_{2,n}(z_{1},z_{2})&=\dfrac{z_{1}^{2}z_{2}^{2}\underline{m}_{2n0}^{\prime}(z_{1}) \underline{m}_{2n0}^{\prime}(z_{2})}{n}\sum_{i=1}^{p}\left[ \bGma^{*}\bbP_{n}^{2}(z_{1})\bGma\right] _{ii}\left[ \bGma^{*}\bbP_{n}^{2}(z_{2})\bGma\right] _{ii},
\end{align*}
and the definitions of $ \bbP_{n} $, $ \bGma $, and $ \underline{m}_{2n0} $ are defined in Section \ref{section 3}.
\end{lemma}
\begin{proof}
	From \cite{10.1214/14-AOS1292}, we have that
under Assumptions \ref{ass1}--\ref{ass4}, the random variable  $ \left( \kappa_{1}^{0}\right) ^{-1}\left( Q_{1}-\mu_{1}\right) \stackrel{d}\longrightarrow N\left(0,1 \right), $ with mean function
\begin{align*}
	\mu_{1}&=-\frac{\alpha_{x}}{2 \pi i}\cdot\oint_{\mathcal{C}}\frac{ f_{1}(z) c_{nM} \int \underline{m}_{2n0}^{3}(z)t^{2}\left(1+t \underline{m}_{2n0}(z)\right)^{-3} d H_{2n}(t)}{\left(1-c_{nM} \int \frac{\underline{m}_{2n0}^{2}(z) t^{2}}{\left(1+t \underline{m}_{2n0}(z)\right)^{2}} d H_{2n}(t)\right)\left(1-\alpha_{x} c_{nM} \int \frac{\underline{m}_{2n0}^{2}(z) t^{2}}{\left(1+t \underline{m}_{2n0}(z)\right)^{2}} d H_{2n}(t)\right) }dz \\
	&-\frac{\beta_{x}}{2 \pi i} \cdot \oint_{\mathcal{C}} \frac{f_{1}(z) c_{nM} \int \underline{m}_{2n0}^{3}(z) t^{2}\left(1+t \underline{m}_{2n0}(z)\right)^{-3} d H_{2n}(t)}{1-c_{nM} \int \underline{m}_{2n0}^{2}(z) t^{2}\left(1+t \underline{m}_{2n0}(z)\right)^{-2} d H_{2n}(t)} dz,  	
\end{align*}
and the covariance function is  \begin{align*}
			\left( \kappa_{1}^{0}\right)^{2}=&-\frac{1}{4\pi^{2}}\oint_{\mathcal{C}_{1}}\oint_{\mathcal{C}_{2}}f_{1}\left(z_{1} \right)f_{1}\left(z_{2} \right)(\vartheta_{n}^{0})^{2}dz_{1}dz_{2}, 
		\end{align*} 
		where
		\begin{align*}
			(\vartheta_{n}^{0})^{2}=&\frac{b_{n}\left(z_{1}\right) b_{n}\left(z_{2}\right)}{n^{2}} \sum_{j=1}^{n} \operatorname{tr} \mathbb{E}_{j} \bGma\bGma^{*} \mathbf{A}_{j}^{-1}\left(z_{1}\right) \mathbb{E}_{j}\left(\bGma\bGma^{*} \mathbf{A}_{j}^{-1}\left(z_{2}\right)\right)\\&+\frac{\alpha_{x} b_{n}\left(z_{1}\right) b_{n}\left(z_{2}\right)}{n^{2}} \sum_{j=1}^{n} \operatorname{tr} \mathbb{E}_{j} \bGma^{*} \mathbf{A}_{j}^{-1}\left(z_{1}\right)\bGma
			\mathbb{E}_{j}\left(\bGma^{\top}\left(\mathbf{A}_{j}^{\top}\right)^{-1}\left(z_{2}\right)\bar{\bGma}\right)\\
			&+
			\frac{\beta_{x} b_{n}\left(z_{1}\right) b_{n}\left(z_{2}\right)}{n^{2}} \sum_{j=1}^{n} \sum_{i=1}^{p} \boldsymbol{e}_{i}^{\top}  {\bGma}^{*} \mathbf{A}_{j}^{-1}\left(z_{1}\right) {\bGma} \boldsymbol{e}_{i} \cdot \boldsymbol{e}_{i}^{\top} {\bGma}^{*} \mathbf{A}_{j}^{-1}\left(z_{2}\right) {\bGma} \boldsymbol{e}_{i},
		\end{align*}
		where $ b_{n}\left(z \right)=\frac{1}{1+n^{-1}\mathbb{E}\rtr\bGma\bGma^{*}\bbA_{j}^{-1}\left(z \right) }.$ The symbols $ \bbA_{j}, \boldsymbol{e_{i}} $ are defined in Section \ref{section 2}. 
Moreover \cite{najim2016gaussian} provided an estimation $ \vartheta_{n}^{2} $ for $ (\vartheta_{n}^{0})^{2} $ and proved that $ (\vartheta_{n}^{0})^{2} $ is close to $ \vartheta_{n}^{2} $ in the Lévy--Prohorov distance, where 
$\vartheta_{n}^{2}=\Theta_{0,n}(z_{1},z_{2})+\al_{x}\Theta_{1,n}(z_{1},z_{2})+\beta_{x}\Theta_{2,n}(z_{1},z_{2}),$ 
\begin{align*}
	\Theta_{0,n}(z_{1},z_{2})&=\dfrac{\underline{m}_{2n0}^{\prime}(z_{1}) \underline{m}_{2n0}^{\prime}(z_{2})  }{(\underline{m}_{2n0}(z_{1})-\underline{m}_{2n0}(z_{2}) )^{2} }-\dfrac{1}{(z_{1}-z_{2})^{2}},\\
	\Theta_{1,n}(z_{1},z_{2})&=\frac{\partial}{\partial z_{2}}\left\lbrace \dfrac{\partial \mathcal{A}_{n}(z_{1},z_{2})}{\partial z_{1}}\dfrac{1}{1-\al_{x}\mathcal{A}_{n}(z_{1},z_{2})} \right\rbrace ,\\
	\mathcal{A}_{n}(z_{1},z_{2})&=\dfrac{z_{1}z_{2}}{n}\underline{m}_{2n0}(z_{1}) \underline{m}_{2n0}(z_{2})\mathrm{tr}{\bGma^{*}\bbP_{n}(z_{1})\bGma\bGma^{\top}\bbP_{n}(z_{2})^{\top} \bar{\bGma}},\\
	\Theta_{2,n}(z_{1},z_{2})&=\dfrac{z_{1}^{2}z_{2}^{2}\underline{m}_{2n0}^{\prime}(z_{1}) \underline{m}_{2n0}^{\prime}(z_{2})}{n}\sum_{i=1}^{p}\left[ \bGma^{*}\bbP_{n}^{2}(z_{1})\bGma\right] _{ii}\left[ \bGma^{*}\bbP_{n}^{2}(z_{2})\bGma\right] _{ii},
\end{align*}
and the definitions of $ \bbP_{n} $, $ \bGma $, and $ \underline{m}_{2n0} $ are given in Section \ref{section 3}. Notably, if $ \bGma $ is not real, then
the convergence of $ \Theta_{1,n}(z_{1},z_{2}) $ is not guaranteed. However, if $ \bGma $ and entries $ x_{ij} $ are real, that is, $ \al_{x}=1 $, then it can be easily proven that $ \Theta_{0,n}(z_{1},z_{2})= \Theta_{1,n}(z_{1},z_{2}) $. Similarly, the convergence of $ \Theta_{2,n}(z_{1},z_{2}) $ depends on the assumption that $ \bGma^{*}\bGma $ is diagonal; thus, under Assumptions \ref{ass1}--\ref{ass4}, $ \Theta_{1,n}(z_{1},z_{2}) $ and $ \Theta_{2,n}(z_{1},z_{2}) $ may have no limits.

Therefore, the covariance term $ \left( \kappa_{1}^{0}\right) ^{2} $ is estimable, and the estimate is $ \kappa_{1}^{2} $, with $$  \kappa_{1}^{2}=  -\frac{1}{4\pi^{2}}\oint_{\mathcal{C}_{1}}\oint_{\mathcal{C}_{2}}f_{1}\left(z_{1} \right)f_{1}\left(z_{2} \right)\vartheta_{n}^{2}dz_{1}dz_{2}. $$
 Thus, the proof is complete.
\end{proof}

\subsection{ Truncation and renormalization}
In this subsection, we truncate and renormalize the random variables to ensure the existence of their higher-order moments. 

{\color{black}{Similar to \cite{JiangB21G}, we may select $ \eta_{n}\rightarrow0 $ such that $ \eta_n^{-4}n^2P(|x_{ij}|\ge \eta_n\sqrt{n})\to 0 $.}} Let $\hat{x}_{i j}=x_{i j} \mathbbm{1}_{\left\lbrace \left|x_{i j}\right|<\eta_{n} \sqrt{n}\right\rbrace} $ and $\tilde{x}_{i j}=\frac{\hat{x}_{i j}-\mathbbm{E} \hat{x}_{i j}}{ \hat\sigma_{n}}$, where $\hat\sigma_{n}^{2}=\mathbbm{E}\left|\hat{x}_{i j}-\mathbbm{E} \hat{x}_{i j}\right|^{2}$. \textcolor{black}{Analogous to the discussion in \cite{Li15}, the sequence  $ \eta_{n}\rightarrow0 $ can be made arbitrarily slow, therefore, we may require it satisfying  
	$ \eta_nn^{t}\to\infty $ for any fixed $t>0$. } Correspondingly, we define $ \hat{\bbB}=\frac{1}{n}\bbT\hat{\bbX}\hat{\bbX}^{\ast}\bbT^{\ast} $ and $ \tilde{\bbB}=\frac{1}{n}\bbT\tilde{\bbX}\tilde{\bbX}^{\ast}\bbT^{\ast} $, where $ \hat{\bbX}=(\hat{x}_{i j})$ and $ \tilde{\bbX}=(\tilde{x}_{i j})$. $ \hat{G}_{n} $ and $ \tilde{G}_{n} $  denote the analogs of $ G_{n} $ with the matrix $ \bbB $ replaced by $ \hat{\bbB} $ and $ \tilde{\bbB} $, respectively. 
\textcolor{black}{Next, we demonstrate that the limiting distribution of the LSS is unchanged when the entries of $\mathbf{X}$ are replaced by the truncated and renormalized entries.}

From Supplement B in \cite{JiangB21G}, we have 
{\color{black}{ \begin{align*}
			&\mathrm{P}(\bbB \neq \hat{\bbB})
			\leq \sum_{i, j} \mathrm{P}\left(\left|x_{i j}\right| \geq \eta_{n} \sqrt{n}\right)\leq n p \mathrm{P}\left(\left|x_{11}\right| \geq \eta_{n} \sqrt{n}\right) \rightarrow 0, \quad \text { as } n, p \rightarrow \infty.
\end{align*}}
It follows from the definition of LSS and Lemma \ref{lemma2},  for any $ l=1,\ldots,h $, 
\begin{align*}  
	& \int f_{l}(x)d\hat{G}_{n}-\int f_{l}(x)d\tilde{G}_{n}(x) =\sum_{i=1}^{p} ( f_{l}(\lambda_{i}^{\hat{\bbB}})-f_{l}(\lambda_{i}^{\tilde{\bbB}}))\\
	&=\sum_{i=1}^{M} ( f_{l}(\lambda_{i}^{\hat{\bbB}})-f_{l}(\lambda_{i}^{\tilde{\bbB}}))+ \sum_{i=M+1}^{p} ( f_{l}(\lambda_{i}^{\hat{\bbB}})-f_{l}(\lambda_{i}^{\tilde{\bbB}})) \\
	&=\sum_{i=1}^{M} ( f_{l}(\lambda_{i}^{\hat{\bbB}})-f_{l}(\lambda_{i}^{\tilde{\bbB}}))+ \sum_{i=M+1}^{p}f_{l}(\lambda_{i}^{\hat{\bbB}})-\sum_{i=1}^{p-M}f_{l}(\lambda_{i}^{\hat{\bbS}_{22}}) + \sum_{i=1}^{p-M}(f_{l}(\lambda_{i}^{\hat{\bbS}_{22}})-f_{l}(\lambda_{i}^{\tilde{\bbS}_{22}})) 
	\\
	&+\sum_{i=1}^{p-M}f_{l}(\lambda_{i}^{\tilde{\bbS}_{22}})-\sum_{i=M+1}^{p}f_{l}(\lambda_{i}^{\tilde{\bbB}}) \\
	&=\sum_{i=1}^{M} ( f_{l}(\lambda_{i}^{\hat{\bbB}})-f_{l}(\lambda_{i}^{\tilde{\bbB}}))+\sum_{i=1}^{p-M}(f_{l}(\lambda_{i}^{\hat{\bbS}_{22}})-f_{l}(\lambda_{i}^{\tilde{\bbS}_{22}})) \\
	&+\frac{M}{2\pi i}\oint_\mathcal{C}f_{l}\left(z \right)\frac{\hat{\underline{m}}_{2n0}^{'}(z)}{\hat{\underline{m}}_{2n0}(z)}dz-\frac{M}{2\pi i}\oint_\mathcal{C}f_{l}\left(z \right)\frac{\tilde{\underline{m}}_{2n0}^{'}(z)}{\tilde{\underline{m}}_{2n0}(z)}dz.
\end{align*}
Here $ \hat{\bbS}_{22} $ and $ \tilde{\bbS}_{22} $  denote the analogs of $ \bbS_{22} $ with the matrix $ \bbX $ replaced by $ \hat{\bbX} $ and $ \tilde{\bbX} $, respectively.  $ \hat{\underline{m}}_{2n0}(z) $ and $ \tilde{\underline{m}}_{2n0}(z) $  denote the analogs of $\underline{m}_{2n0}(z) $ with the matrix $ \bbX $ replaced by $ \hat{\bbX} $ and $ \tilde{\bbX} $, respectively.

Therefore, $ \left|\hat{Y}_l-\tilde{Y}_l \right|=\left|\sum_{i=1}^{M} ( f_{l}(\lambda_{i}^{\hat{\bbB}})-f_{l}(\lambda_{i}^{\tilde{\bbB}}))+ \sum_{i=1}^{p-M}(f_{l}(\lambda_{i}^{\hat{\bbS}_{22}})-f_{l}(\lambda_{i}^{\tilde{\bbS}_{22}}))\right|\leq \\\left|\sum_{i=1}^{M} ( f_{l}(\lambda_{i}^{\hat{\bbB}})-f_{l}(\lambda_{i}^{\tilde{\bbB}})) \right| +\left|\sum_{i=1}^{p-M}(f_{l}(\lambda_{i}^{\hat{\bbS}_{22}})-f_{l}(\lambda_{i}^{\tilde{\bbS}_{22}})) \right|.   $}

\textcolor{black}{
	For $ \left|\sum_{i=1}^{p-M}(f_{l}(\lambda_{i}^{\hat{\bbS}_{22}})-f_{l}(\lambda_{i}^{\tilde{\bbS}_{22}})) \right|, $ we have 
	\begin{align*} 
		\left|\sum_{i=1}^{p-M}(f_{l}(\lambda_{i}^{\hat{\bbS}_{22}})-f_{l}(\lambda_{i}^{\tilde{\bbS}_{22}})) \right| \leq \sum_{i=1}^{p-M}\left| f_{l}(\lambda_{i}^{\hat{\bbS}_{22}})-f_{l}(\lambda_{i}^{\tilde{\bbS}_{22}})\right|
		\leq C\sum_{i=1}^{p-M}\left| \lambda_{i}^{\hat{\bbS}_{22}}-\lambda_{i}^{\tilde{\bbS}_{22}}\right|.
	\end{align*}
	Similar to  Lemma 2.7 in \cite{bai99}, we have
	\begin{align*}
		&\sum_{i=1}^{p-M}\left| \lambda_{i}^{\hat{\bbS}_{22}}-\lambda_{i}^{\tilde{\bbS}_{22}}\right|\leq \left(\frac{2}{n}\mathrm{tr}(\hat{\bbS}_{22}+\tilde{\bbS}_{22})  \mathrm{tr}\bbD_{2}^{\frac{1}{2}}\bbU_{2}^{*}(\hat{\bbX}-\tilde{\bbX})(\hat{\bbX}-\tilde{\bbX})^{*} \bbU_{2}\bbD_{2}^{\frac{1}{2}}\right) ^{1/2}\\
		&= \left(\frac{2}{n}\mathrm{tr}(\hat{\bbS}_{22}+\tilde{\bbS}_{22})  \mathrm{tr}(\hat{\bbX}-\tilde{\bbX})(\hat{\bbX}-\tilde{\bbX})^{*} \bbU_{2}\bbD_{2}\bbU_2^*\right) ^{1/2}\\
		&\leq \left(\frac{2C}{n}\mathrm{tr}(\hat{\bbS}_{22}+\tilde{\bbS}_{22})  \mathrm{tr}(\hat{\bbX}-\tilde{\bbX})(\hat{\bbX}-\tilde{\bbX})^{*} \right)^{1/2},
	\end{align*}
	where the last inequality is because  $\bbU_{2}\bbD_{2}\bbU_2^*  $ is bounded. It is easy to prove $ \frac{1}{n}\mathbb{E}\mathrm{tr}(\hat{\bbS}_{22}+\tilde{\bbS}_{22})<C, $ and
	\begin{align} 
		&\nonumber\mathbb{E}\mathrm{tr}(\hat{\bbX}-\tilde{\bbX})(\hat{\bbX}-\tilde{\bbX})^{*}=\sum_{i,j}\mathbb{E}\left|  (1-\frac{1}{\hat{\sigma}_{n}})\hat{x}_{ij}+\frac{\mathbb{E}\hat{x}_{ij}}{\hat{\sigma}_{n}}\right| ^{2} \\
		&=\sum_{i,j}\left(  \Var( (1-\frac{1}{\hat{\sigma}_{n}})\hat{x}_{ij}+\frac{\mathbb{E}\hat{x}_{ij}}{\hat{\sigma}_{n}} )+\left| \mathbb{E}\hat{x}_{ij}\right|^{2}\right) .\label{truncation1}     
	\end{align}
	Since $ \Var\left( (1-\frac{1}{\hat{\sigma}_{n}})\hat{x}_{ij}+\frac{\mathbb{E}\hat{x}_{ij}}{\hat{\sigma}_{n}} \right)=(\hat{\sigma}_n-1)^2\leq \frac{(\hat{\sigma}_{n}^{2}-1)^{2}}{\hat{\sigma}_{n}^{2}+1}\leq 2\left( \mathbb{E} \left|x_{ij} \right|^{2}\mathbbm{1}_{\left\lbrace \left|x_{ij} \right|\geq \eta_{n}\sqrt{n}\right\rbrace }\right) ^{2}, $ and by the selection of $ \eta_{n} $,
	$ \mathbb{E} \left|x_{ij} \right|^{2}\mathbbm{1}_{\left\lbrace \left|x_{ij} \right|\geq \eta_{n}\sqrt{n}\right\rbrace }=o(\eta_{n}^{2}n^{-1}), $ $ \mathbb{E}\hat{x}_{ij}=o(\eta_{n}n^{-3/2}), $ then we can obtain (\ref{truncation1}) is o(1). Therefore, $ 	\sum_{i=1}^{p-M}\left| \lambda_{i}^{\hat{\bbS}_{22}}-\lambda_{i}^{\tilde{\bbS}_{22}}\right| $ is $ o_p(1) $.
}

For the first term $ \left|\sum_{i=1}^{M} ( f_{l}(\lambda_{i}^{\hat{\bbB}})-f_{l}(\lambda_{i}^{\tilde{\bbB}})) \right| $, from the arguments in Supplement B of \cite{JiangB21G}, we know that 
\begin{align} \left|\lambda_{i}^{\hat{\bbB}}-\lambda_{i}^{\tilde{\bbB}}\right|=o_{p}(n^{-\frac{1}{2} } {\rho_{i}}). \label{truncation}
\end{align} 
We recall that $ \rho_{i} = \al_{k} $ if $ i \in J_{k} $. 
Then, for brevity, we denote $ \beta_{i}=({\lambda_{i}^{\hat{\bbB}}-\lambda_{i}^{\tilde{\bbB}}})/{\rho_{i}}$ and
obtain that
\begin{align}
	&\nonumber f_l(\lambda_{i}^{\hat{\bbB}})-f_l(\lambda_{i}^{\tilde{\bbB}})	=\int_{0}^{\beta_{i}\rho_{i}}f_l^{\prime}(t+\lambda_{i}^{\tilde{\bbB}})	dt\\
	=&\int_{0}^{1}\beta_{i}\rho_{i}f_l^{\prime}(\beta_{i}\rho_{i}s+\lambda_{i}^{\tilde{\bbB}})ds
	=\beta_{i}\rho_{i}f_l^{\prime}(\rho_{i})\int_{0}^{1}\frac{f_l^{\prime}(\rho_{i}(\beta_{i}s+\frac{\lambda_{i}^{\tilde{\bbB}}}{\rho_{i}}) )}{f_l^{\prime}(\rho_{i})}ds.
\end{align} 
Since $ \beta_{i}=o_{p}(n^{-\frac{1}{2}}) $, $ \frac{ \lambda_{i}^{\hat{\bbB}}}{ \rho_{i} } $ tends to 1,
then from Assumption \ref{ass4} we obtain
\begin{align}  \sqrt{n}\sum_{i=1}^{M}\dfrac{\left| f_{l}(\lambda_{i}^{\hat{\bbB}})-f_{l}(\lambda_{i}^{\tilde{\bbB}})\right|}{f_{l}^{\prime}(\rho_{i})\rho_{i}}= o_{p}(1). \label{trun}
\end{align} 
Then we have $ \frac{\left|\sum_{i=1}^{M} ( f_{l}(\lambda_{i}^{\hat{\bbB}})-f_{l}(\lambda_{i}^{\tilde{\bbB}})) \right|}{\varsigma_{l}}\leq \frac{\sum_{i=1}^{M}\left|  f_{l}(\lambda_{i}^{\hat{\bbB}})-f_{l}(\lambda_{i}^{\tilde{\bbB}}) \right|}{\varsigma_{l}}\leq \sum_{i=1}^{M}\frac{\left| f_{l}(\lambda_{i}^{\hat{\bbB}})-f_{l}(\lambda_{i}^{\tilde{\bbB}})\right|}{f_{l}^{\prime}(\rho_{i})\rho_{i}/\sqrt{n}}= o_{p}(1). $
Thus, it is concluded that the procedure of truncation does not affect the limiting distribution of LSS.

Therefore, in the following proofs, we can safely assume that $ \left|x_{ij} \right|<\eta_{n}\sqrt{n}$. 

\begin{acks}[Acknowledgments]
The authors would like to thank Professor Jeff Yao for many helpful suggestions and discussions.  Zhidong Bai was partially supported by NSFC Grants No.12171198, No.12271536, and Team Project of Jilin Provincial
Department of Science and Technology No.20210101147JC. Jiang Hu was partially supported by NSFC Grant No.12171078.
\end{acks}

\begin{supplement}
	\stitle{Supplement to ``A CLT for the LSS of large dimensional sample covariance matrices with diverging spikes''}
	\sdescription{This supplementary document contains proofs of Theorems \ref{thm1}--\ref{RLRT power}, and some useful lemmas. We also
	report some additional simulation results in this document.}
\end{supplement}	

\newpage

	\begin{frontmatter}
		\title{Supplement to ``A CLT for the LSS of large dimensional sample covariance matrices with diverging spikes''}
		\runtitle{CLT for LSS with diverging spikes}
		
		\begin{aug}
			\author{\fnms{Zhijun} \snm{Liu}\ead[label=e1,mark]
				{liuzj037@nenu.edu.cn}},
			\author{\fnms{Jiang} \snm{Hu}\ead[label=e2,mark]{huj156@nenu.edu.cn}},
			\author{\fnms{Zhidong} \snm{Bai}\ead[label=e3,mark]{baizd@nenu.edu.cn}},
			\and 
			\author{\fnms{Haiyan}
				\snm{Song}\ead[label=e4,mark]{songhy716@nenu.edu.cn}}
			\address{KLASMOE and School of Mathematics and Statistics, Northeast Normal University, China.
				\printead{e1,e2,e3,e4}}
			
		\end{aug}
		
		
		
		\begin{keyword}[class=MSC]
			\kwd[Primary ]{	60B20}
			\kwd[; secondary ]{60F05}
		\end{keyword}
		
		\begin{keyword}
			\kwd{Empirical spectral distribution}
			\kwd{linear spectral statistic}
			\kwd{random matrix}
			\kwd{Stieltjes transform}
		\end{keyword}
	\end{frontmatter}
	
	

	In this document we present some technical details involved in \cite{Liu2022}. More precisely, in Section \ref{postpone proof}, we prove Theorems \ref{thm1}--\ref{RLRT power} of the main file.  Some derivations and calculations in Section \ref{postpone proof} are postponed to Section \ref{sectionB}. In Section \ref{some useful lemmas} we provide some useful lemmas.  Finally, in Section \ref{other simulations}, we report some additional simulation results in this part.
	
	The number of scheme(equations,theorems,lemmas,etc.) is shared with the main document so that there are no misunderstandings with the use of references.

	\section{Proofs of Theorems \ref{thm1}--\ref{RLRT power}}\label{postpone proof}
	\subsection{Proof of Theorem \ref{thm1}}
	The proof of Theorem \ref{thm1} builds on the decomposition analysis of the LSSs and it is divided into part (\uppercase\expandafter{\romannumeral1}) $ \sum_{j=1}^{M}f\left( \bgl_{j}\right) $ and part (\uppercase\expandafter{\romannumeral2}) $ \sum_{j=M+1}^{p}f\left( \bgl_{j}\right) $. Enlightened by the BST in \cite{10.1214/aop/1078415845}, we have
	\begin{align*}
		&\sum_{j=1}^{p}f\left( \bgl_{j}\right)-p\int f(x)dF^{c_{n},H_{n}} \\
		&=\sum_{j=1}^{M}f\left( \bgl_{j}\right)+\sum_{j=M+1}^{p}f\left( \bgl_{j}\right)-p\int f(x)dF^{c_{n},H_{n}}\\
		&=\sum_{j=1}^{M}f\left( \bgl_{j}\right)+\sum_{j=1}^{p-M}f\left( \widetilde{\bgl}_{j}\right)-(p-M)\int f(x)dF^{c_{nM},H_{2n}}+\sum_{j=M+1}^{p}f\left( \bgl_{j}\right)-\sum_{j=1}^{p-M}f\left( \widetilde{\bgl}_{j}\right)\\
		&+(p-M)\int f(x)dF^{c_{nM},H_{2n}}-p\int f(x)dF^{c_{n},H_{n}}.
	\end{align*}
	Since Lemma \ref{lemma1} has shown the difference between $ (p-M)\int f(x)dF^{c_{nM},H_{2n}}$ and $p\int f(x)dF^{c_{n},H_{n}}$ is 0. 
	Moreover, in Lemma \ref{lemma2} we have proved $$ \sum_{j=M+1}^{p}f\left( \bgl_{j}\right)-\sum_{j=1}^{p-M}f\left( \widetilde{\bgl}_{j}\right)=\frac{M}{2\pi i}\oint_{C}f\left(z \right)\frac{\underline{m}_{2n0}^{'}(z)}{\underline{m}_{2n0}(z)}dz+o_{p}(1). $$
	It follows that
	\begin{align*}
		&\sum_{j=1}^{p}f\left( \bgl_{j}\right)-p\int f(x)dF^{c_{n},H_{n}}\\
		=&\sum_{j=1}^{M}f\left( \bgl_{j}\right)+\sum_{j=1}^{p-M}f\left( \widetilde{\bgl}_{j}\right)
		-(p-M)\int f(x)dF^{c_{nM},H_{2n}}+\frac{M}{2\pi i}\oint_{C}f\left(z \right)\frac{\underline{m}_{2n0}^{'}(z)}{\underline{m}_{2n0}(z)}dz+ o_{p}(1),
	\end{align*}
	which yields
	\begin{align}
		&\nonumber\sum_{j=1}^{p}f\left( \bgl_{j}\right)-p\int f(x)dF^{c_{n},H_{n}}-\sum_{k=1}^{K}d_{k}f\left( \phi_{n}\left(\al_{k} \right) \right)-\frac{M}{2\pi i}\oint_{C}f\left(z \right)\frac{\underline{m}_{2n0}^{'}(z)}{\underline{m}_{2n0}(z)}dz\\
		=&\sum_{j=1}^{M}f\left( \bgl_{j}\right)-\sum_{k=1}^{K}d_{k}f\left( \phi_{n}\left(\al_{k} \right) \right)+\sum_{j=1}^{p-M}f\left( \widetilde{\bgl}_{j}\right)-(p-M)\int f(x)dF^{c_{nM},H_{2n}}+ o_{p}(1). \label{100}
	\end{align}
	
	The analysis below is executed by dividing (\ref{100}) into two parts: (\uppercase\expandafter{\romannumeral1}) $ \sum_{j=1}^{M}f\left( \bgl_{j}\right)-\sum_{k=1}^{K}d_{k}f\left( \phi_{n}\left(\al_{k} \right) \right) $ and (\uppercase\expandafter{\romannumeral2}) $ \sum_{j=1}^{p-M}f\left( \widetilde{\bgl}_{j}\right)-(p-M)\int f(x)dF^{c_{nM},H_{2n}} $, where we ignore the impact of $ o_{p}(1) $ on the asymptotic distribution. Since we have derived the asymptotic distribution of part (\uppercase\expandafter{\romannumeral2}) in Lemma \ref{lemma3}, we only need to consider the asymptotic distribution of part (\uppercase\expandafter{\romannumeral1}) $ \sum_{j=1}^{M}f\left(\lambda_{j}\right)-\sum_{k=1}^{K}d_{k}f\left(\phi_{n}\left(\bbalp_{k} \right) \right) $. From the proof of Lemma \ref{lemma5}, $ \sum_{j=1}^{M}f\left( \bgl_{j}\right)-\sum_{k=1}^{K}d_{k}f\left(\phi_{n}\left(\bbalp_{k} \right)\right)  $ has the same limiting distribution as $ \sum_{k=1}^{K}\varpi_{nk}\sum_{j\in J_{k}}\gamma_{kj}  $. From Lemma \ref{lemma4}, we have $ \left(\gamma_{kj}, j\in J_{k} \right)^{\top} $ converges weakly to the joint distribution of the $ d_{k} $ eigenvalues of Gaussian random matrix $ -\frac{1}{\theta_{k}}\left[\bgO_{\phi_{k}} \right]_{kk}  $, so $ \sum_{j\in J_{k}}\gamma_{kj}\stackrel{d}\rightarrow -\frac{1}{\theta_{k}}\mathrm {tr}\left[\bgO_{\phi_{k}} \right]_{kk}    $. Recall that $ \omega_{ij} $ is the element of $ \bgO_{\phi_{k}} $, and $ \mathrm {tr}\left[\bgO_{\phi_{k}} \right]_{kk} $ is the summation of the diagonal element, that is, $ \sum_{j\in J_{k}} \omega_{jj} $. Since the diagonal elements are i.i.d., then $ \mathbb{E}\left(\sum_{j\in J_{k}} \omega_{jj}\right) =0  $,  $ \mathrm{Var}\left(\sum_{j\in J_{k}} \omega_{jj}\right)=\sum_{j\in J_{k}}\mathrm{Var}\left(\omega_{jj} \right)+ \sum_{j_{1}\neq j_{2}}\mathrm{cov}\left(\omega_{j_{1}j_{1}}, \omega_{j_{2}j_{2}} \right)= \sum_{j\in J_{k}}( \left(\al_{x}+1 \right)\theta_{k}+\beta_{x}\mathcal{U}_{jjjj}\nu_{k}) +\sum_{j_{1}\neq j_{2}}\beta_{x}\mathcal{U}_{j_{1}j_{1}j_{2}j_{2}}\nu_{k} =\left( \al_{x}+1\right)\theta_{k}d_{k}+\sum_{j_{1},j_{2}\in J_{k}} \mathcal{U}_{j_{1}j_{1}j_{2}j_{2}}\beta_{x}\nu_{k}     $ 
	
	Therefore,
	from Lemma \ref{lemma4}, we have that the asymptotic distribution of $ \sum_{j\in J_{k}}\gamma_{kj} $ is a Gaussian distribution with $$
	\mathbb{E}\sum_{j\in J_{k}}\gamma_{kj}=0,$$ $$
	s_{k}^{2}\triangleq \mathrm{Var}\left(\sum_{j\in J_{k}}\gamma_{kj}\right)=\frac{ \left( \al_{x}+1\right)\theta_{k}d_{k}+\sum_{j_{1},j_{2}\in J_{k}} \mathcal{U}_{j_{1}j_{1}j_{2}j_{2}}\beta_{x}\nu_{k}}{\theta_{k}^{2}}, $$ and then, we directly derive that the mean function of $ \sum_{k=1}^{K}\varpi_{nk}\sum_{j\in J_{k}}\gamma_{kj}  $ is 0 and that its covariance function is
	$$
	\mathrm{Var}\left( Y_{f_{1}} \right)=\sum_{k=1}^{K}\varpi_{nk1}^{2}s_{k}^{2}. 
	$$

	Finally, we focus on the asymptotic distribution of equation (\ref{100}). Because of Lemma \ref{lemma5}, the two LSSs are asymptotically independent; thus, the random variable 
	$$ \varsigma_{1}^{-1}\left(Y_{1}-\mathbb{E}Y_{1} \right)\stackrel{d}\longrightarrow N\left(0,1 \right)  $$
	with mean function $\mathbb{E}Y_{1}=\mu_{1}$ being
	\begin{align*}
		&-\frac{\alpha_{x}}{2 \pi i}\cdot\oint_{\mathcal{C}}f_{1}(z)\frac{  c_{nM} \int \underline{m}_{2n0}^{3}(z)t^{2}\left(1+t \underline{m}_{2n0}(z)\right)^{-3} d H_{2n}(t)}{\left(1-c_{nM} \int \frac{\underline{m}_{2n0}^{2}(z) t^{2}}{\left(1+t \underline{m}_{2n0}(z)\right)^{2}} d H_{2n}(t)\right)\left(1-\alpha_{x} c_{nM} \int \frac{\underline{m}_{2n0}^{2}(z) t^{2}}{\left(1+t \underline{m}_{2n0}(z)\right)^{2}} d H_{2n}(t)\right) }dz \\
		&-\frac{\beta_{x}}{2 \pi i} \cdot \oint_{\mathcal{C}} f_{1}(z) \frac{c_{nM} \int \underline{m}_{2n0}^{3}(z) t^{2}\left(1+t \underline{m}_{2n0}(z)\right)^{-3} d H_{2n}(t)}{1-c_{nM} \int \underline{m}_{2n0}^{2}(z) t^{2}\left(1+t \underline{m}_{2n0}(z)\right)^{-2} d H_{2n}(t)} dz,  	
	\end{align*}
	and covariance function $\varsigma_{1}^{2}$ being 
	\begin{align*}
		\sum_{k=1}^{K}\varpi_{nk1}^{2}s_{k}^{2} -\frac{1}{4\pi^{2}}\oint_{\mathcal{C}_{1}}\oint_{\mathcal{C}_{2}}f_{1}\left(z_{1} \right)f_{1}\left(z_{2} \right)\vartheta_{n}^{2}dz_{1}dz_{2},
	\end{align*}
	where $ \vartheta_{n}^{2} $ is defined in Lemma \ref{lemma3}. Therefore, the proof is finished.

	\subsection{Proof of Theorem \ref{thm2}}
	Similar to the proof of Theorem \ref{thm1}, we divide the LSSs into two parts. Different from the above analysis, in this section, we focus on the multidimensional case under Assumptions \ref{ass1}--\ref{ass6}. Recall that we defined
	$$ G_{n}\left( x\right)=p\left[F^{\bbB}\left(x \right)-F^{c_{n},H_{n}}\left(x \right)   \right],  $$ 
	$$ Y_{l}=  \int f_{l}\left(x \right)dG_{n}\left( x\right)-\sum_{k=1}^{K}d_{k}f_{l}\left(\phi_{n}\left(\al_{k} \right)  \right)-\frac{M}{2\pi i}\oint_{\mathcal C}f_{l}\left(z \right)\frac{\underline{m}_{2n0}'(z)}{\underline{m}_{2n0}(z)}dz. $$
	Because of equation (\ref{100}), the random vector $\left(  Y_{1},\dots, Y_{h} \right)  $ shares the same asymptotic distribution with the summation of two random vectors  $$ ( \sum_{k=1}^{K}\varpi_{nk1}\sum_{j\in J_{k}}\gamma_{kj} ,\dots,\sum_{k=1}^{K}\varpi_{nkh}\sum_{j\in J_{k}}\gamma_{kj}  )$$ and  $$(   \sum_{j=1}^{p-M}f_{1}\left( \widetilde{\bgl}_{j}\right)-(p-M)\int f_{1}(x)dF^{c_{nM},H_{2n}} ,\dots, \sum_{j=1}^{p-M}f_{h}\left( \widetilde{\bgl}_{j}\right)-(p-M)\int f_{h}(x)dF^{c_{nM},H_{2n}} ).    $$ 
	
	First, we focus on the first random vector. Similar to the proof of Theorem \ref{thm1}, we derive that the mean function of the first random vector is 0 and that the covariance function is
	$$
	\mathrm{Cov}\left( Y_{f_{s}},Y_{f_{t}} \right)=\sum_{k=1}^{K}\frac{\phi_{n}^{2}\left(\bbalp_{k} \right)}{n} f_{s}'\left( \phi_{n}\left(\bbalp_{k} \right)\right)f_{t}'\left( \phi_{n}\left(\bbalp_{k} \right)\right)s_{k}^{2}= \sum_{k=1}^{K}\varpi_{nks}\varpi_{nkt}s_{k}^{2}, 
	$$
	Moreover, the asymptotic distribution of the second random vector is derived in \cite{10.1214/14-AOS1292}. Because of Lemma \ref{lemma5}, two random vectors are asymptotically independent; thus, the random vector
	$$ \left( Y_{1}-\mathbb{E}Y_{1},\dots,Y_{h}-\mathbb{E}Y_{h}\right)^{\top}\stackrel{d}{\rightarrow}N_{h}\left(0, \bgO\right),   $$ 
	with mean function $ \mathbb{E}Y_{l} $ is the same as $ \mu_{l} $, and the covariance matrix is $ \bgO $ with its entries
	\textcolor{black}{\begin{align*}
			\omega_{st}=\sum_{k=1}^{K}\varpi_{nks}\varpi_{nkt}s_{k}^{2}- \kappa_{nst},
		\end{align*} 
		where \begin{align*}
			\kappa_{nst}
			&=\frac{1}{4 \pi^{2}} \oint_{\mathcal{C}_{1}} \oint_{\mathcal{C}_{2}} \frac{f_{s}\left(z_{1}\right) f_{t}\left(z_{2}\right)}{\left(\underline{m}_{2n0} \left(z_{1}\right)-\underline{m}_{2n0}\left(z_{2}\right)\right)^{2}} d \underline{m}_{2n0}\left(z_{1}\right) d \underline{m}_{2n0}\left(z_{2}\right) 
			+\frac{c_{nM} \beta_{x}}{4 \pi^{2}} \oint_{\mathcal{C}_{1}} \oint_{\mathcal{C}_{2}} f_{s}\left(z_{1}\right) \\&f_{t}\left(z_{2}\right)
			\left[\int \frac{t}{\left(\underline{m}_{2n0}\left(z_{1}\right) t+1\right)^{2}}\right.
			\left.\times \frac{t}{\left(\underline{m}_{2n0}\left(z_{2}\right) t+1\right)^{2}} d H_{2n}(t)\right] d \underline{m}_{2n0}\left(z_{1}\right) d \underline{m}_{2n0}\left(z_{2}\right)\\
			&\quad+\frac{1}{4 \pi^{2}} \oint_{\mathcal{C}_{1}} \oint_{\mathcal{C}_{2}} f_{s}\left(z_{1}\right) f_{t}\left(z_{2}\right)\left[\frac{\partial^{2}}{\partial z_{1} \partial z_{2}} \log \left(1-a_{n}\left(z_{1}, z_{2}\right)\right)\right] d z_{1} d z_{2},\\
			a_{n}\left(z_{1}, z_{2}\right)&=\alpha_{x}\left(1+\frac{\underline{m}_{2n0}\left(z_{1}\right) \underline{m}_{2n0}\left(z_{2}\right)\left(z_{1}-z_{2}\right)}{\underline{m}_{2n0}\left(z_{2}\right)-\underline{m}_{2n0}\left(z_{1}\right)}\right).
		\end{align*} 
		Then, we obtain the random vector $$ \left( \frac{Y_{1}-\mathbb{E}Y_{1}}{\sigma_{1}},\dots,\frac{Y_{h}-\mathbb{E}Y_{h}}{\sigma_{h}}\right)^{\top}\stackrel{d}{\rightarrow}N_{h}\left(0, \bPsi \right), 
		$$ which has a mean function that is the same as that in Theorem \ref{thm1}, and variance function 
		\begin{align*}
			\sigma_{l}^{2}&=\sum_{k=1}^{K}\varpi_{nkl}^{2} s_{k}^{2}-\kappa_{nll}, \quad l=1,\dots,h,
		\end{align*} 
		and the covariance matrix $  \bPsi =\left( \psi_{st}\right) _{h\times h} $ is the correlation coefficient matrix of random vector $ (Y_{1},\dots,Y_{h})^{\top} $  with its entries $ \psi_{st}=\lim_{n\rightarrow\infty}\psi_{nst}, $
		\begin{align*}
			\psi_{nst}=\frac{\sum_{k=1}^{K}\varpi_{nks}\varpi_{nkt}s_{k}^{2}- \kappa_{nst}}{\sqrt{\sum_{k=1}^{K}\varpi_{nks}^{2}s_{k}^{2}- \kappa_{nss}}\sqrt{\sum_{k=1}^{K}\varpi_{nkt}^{2}s_{k}^{2}- \kappa_{ntt}}},
	\end{align*} }
	Note that renormalization is necessary to guarantee that elements in the correlation coefficient matrix $ \bPsi $ are limited. Therefore, the proof is finished.

	\subsection{Proof of Theorem \ref{thm3}}
	
	The result under $ H_{0} $ is a direct result of Theorem 4.1 in \cite{10.1214/14-AOS1292}  using the substitution principle. Therefore, we omit the proof here. Next, we focus on the result under $ H_{1}. $
	Recall that $$ G_{n}\left( x\right)=p\left[F^{\bbB}\left(x \right)-F^{c_{n},H_{n}}\left(x \right)   \right],  $$ $$  Y= \int f_{L}\left(x \right)dG_{n}\left( x\right)-\sum_{k=1}^{K}d_{k}f_{L}\left(\phi_{n}\left(\al_{k} \right)  \right)-\frac{M}{2\pi i}\oint_{\mathcal C}f_{L}\left(z \right)\frac{\underline{m}_{2n0}'(z)}{\underline{m}_{2n0}(z)}dz,$$
	when $ f_{L}(x)=x-\log x-1. $ After some calculations, we obtain
	\begin{gather}
		\nonumber\int f_{L}\left(x \right)dG_{n}\left( x\right)=\mathrm{tr}\bbB-\log \left|\bbB \right|-p-p\int f_{L}(x)dF^{c_{n},H_{n}}(x)=L-p\int f_{L}(x)dF^{c_{n},H_{n}}(x),\\
		p\int f_{L}(x)dF^{c_{n},H_{n}}(x)=(p-M)(1-\frac{c_{nM}-1}{c_{nM}} \log \left(1-c_{nM}\right)),\label{32}\\	
		\nonumber\sum_{k=1}^{K}d_{k}f_{L}\left(\phi_{n}\left(\al_{k} \right)  \right)=\sum_{k=1}^{K}d_{k}\left(\phi_{n}\left(\al_{k} \right)-\log \phi_{n}\left(\al_{k} \right)-1 \right),\\
		\frac{M}{2\pi i}\oint_{\mathcal C}f_{L}\left(z \right)\frac{\underline{m}_{2n0}'(z)}{\underline{m}_{2n0}(z)}dz= -M(c_{nM}+\log(1-c_{nM})),\label{31}
	\end{gather} 
	where (\ref{32}) is obtained from Lemma \ref{lemma1} and \cite{10.1214/09-AOS694}. For consistency, we present the proof of (\ref{31}) in Section \ref{sectionB}.
	According to Theorem \ref{thm1}, since $ f_{L}(x)=x-\log x-1 $, $ \bbD_{2}=\bbI_{p-M}, $ $ \bGma=\bbV_{2}\bbU_{2}^{*} $, then we have
	\begin{align*}
		\frac{L-p\int f_{L}(x)dF^{c_{n},H_{n}}(x)-\breve{\mu}_{l}}{ \breve{\varsigma}_{l}}\stackrel{d}{\rightarrow}N(0,1),
	\end{align*}	where the mean function is $\breve{\mu}_{l}=-\frac{\log \left(1-c_{nM}\right)}{2}\al_{x}+\frac{c_{nM}}{2} \beta_{x} 
	+\sum_{k=1}^{K}d_{k}( \phi_{n}\left({\al}_{k} \right)-\log\phi_{n}\left({\al}_{k} \right)-1 )-M(c_{nM}+\log(1-c_{nM}))$. \textcolor{black}{ For covariance term, $ \breve{\varsigma}_{l}^{2} $ equals \\ $ \sum_{k=1}^{K}\frac{\left( \phi_{n}\left(\al_{k} \right)-1 \right) ^{2}}{n}s_{k}^{2} -\frac{1}{4\pi^{2}}\oint_{\mathcal{C}_{1}}\oint_{\mathcal{C}_{2}}\left(z_{1}-\log\left(z_{1} \right)-1  \right) \left(z_{2}-\log\left(z_{2} \right)-1  \right)\vartheta_{n}^{2}dz_{1}dz_{2}, $ 
		where $ s_{k}^{2}= \frac{\left(\al_{x}+1 \right)d_{k} }{\theta_{k}}+\frac{\sum_{j_{1}, j_{2}\in J_{k}}\mathcal{U}_{j_{1}j_{1}j_{2}j_{2}}  \beta_{x}\nu_{k} }{\theta_{k}^{2}}  $, $ \vartheta_{n}^{2}=\Theta_{0,n}(z_{1},z_{2})+\al_{x}\Theta_{1,n}(z_{1},z_{2})+\beta_{x}\Theta_{2,n}(z_{1},z_{2}), $ and
		\begin{align*}
			\Theta_{0,n}(z_{1},z_{2})=& \dfrac{\underline{m}_{2n0}^{\prime}(z_{1}) \underline{m}_{2n0}^{\prime}(z_{2})  }{(\underline{m}_{2n0}(z_{1})-\underline{m}_{2n0}(z_{2}) )^{2} }-\dfrac{1}{(z_{1}-z_{2})^{2}}.
		\end{align*}	
		For $ \Theta_{1,n}(z_{1},z_{2}), $ since
		\begin{align*}
			\bbP_{n}\left(z \right)=&\left(\left(1-c_{nM}\right)\bbV_{2}\bbV_{2}^{*}-zc_{nM}m_{2n0}\left(z \right)\bbV_{2}\bbV_{2}^{*}-z\bbI_{p}   \right) ^{-1}, 
		\end{align*}
		then
		\begin{align*}
			{\mathcal{A}_{n}(z_{1},z_{2})}=&\frac{z_{1}z_{2}}{n}\underline{m}_{2n0}(z_{1})\underline{m}_{2n0}(z_{2})\mathrm{tr}\bbU_{2}\bbV_{2}^{*}\bbP_{n}\left(z_{1} \right)\bbV_{2}\bbU_{2}^{*}\bar{\bbU}_{2}\bbV_{2}^{\top}\bbP_{n}^{\top}\left(z_{2} \right)\bar{\bbV}_{2}\bbU_{2}^{\top} ,\\
			=&\frac{z_{1}z_{2}\underline{m}_{2n0}(z_{1})\underline{m}_{2n0}(z_{2})\mathrm{tr}\bbU_{2}\bbU_{2}^{*}\bar{\bbU}_{2}\bbU_{2}^{\top}}{n\left( 1-c_{nM}-z_{1}c_{nM}m_{2n0}(z_{1})-z_{1}\right) \left(1-c_{nM}-z_{2}c_{nM}m_{2n0}(z_{2})-z_{2} \right) },\\
			=&\frac{\underline{m}_{2n0}(z_{1})\underline{m}_{2n0}(z_{2})}{n\left(1+\underline{m}_{2n0}(z_{1}) \right)\left(1+\underline{m}_{2n0}(z_{2}) \right) }\mathrm{tr}\bbU_{2}\bbU_{2}^{*}\bar{\bbU}_{2}\bbU_{2}^{\top}.
		\end{align*}
		For $ \mathrm{tr}\bbU_{2}\bbU_{2}^{*}\bar{\bbU}_{2}\bbU_{2}^{\top}, $	since $ \bbU\bbU^{*}=\bbI_{p}, $ therefore $ \mathrm{tr}\bbU_{2}\bbU_{2}^{*}\bar{\bbU}_{2}\bbU_{2}^{\top}=\mathrm{tr}\left(\bbI_{p}-\bbU_{1}\bbU_{1}^{*} \right)\left(\bbI_{p}-\bbU_{1}\bbU_{1}^{*} \right)^{\top}=p-\mathrm{tr}\left( \bbU_{1}\bbU_{1}^{*}\right) ^{\top}-\mathrm{tr}\bbU_{1}\bbU_{1}^{*}+\mathrm{tr}\bbU_{1}\bbU_{1}^{*}\left( \bbU_{1}\bbU_{1}^{*}\right) ^{\top}.  $ Moreover, since $ \mathrm{tr}\bbU_{1}\bbU_{1}^{*}=M, $ $\mathrm{tr}\bbU_{1}\bbU_{1}^{*}\left( \bbU_{1}\bbU_{1}^{*}\right) ^{\top}\\=\sum_{s,t=1}^{p}\left(\sum_{i=1}^{M}u_{si}\bar{u}_{ti} \right)^{2}.   $ Therefore, $ {\mathcal A}_{n\left(z_{1},z_{2} \right) }=\frac{p-2M+\sum_{s,t=1}^{p}\left(\sum_{i=1}^{M}u_{si}\bar{u}_{ti} \right)^{2}}{n}\frac{\underline{m}_{2n0}(z_{1})\underline{m}_{2n0}(z_{2})}{\left(1+\underline{m}_{2n0}(z_{1}) \right)\left(1+\underline{m}_{2n0}(z_{2}) \right) }. $ Denote $ \tilde{c}=\frac{p-2M+\sum_{s,t=1}^{p}\left(\sum_{i=1}^{M}u_{si}\bar{u}_{ti} \right)^{2}}{n}, $ then $ {\mathcal A}_{n}\left(z_{1},z_{2} \right)= \frac{\tilde{c}\underline{m}_{2n0}(z_{1})\underline{m}_{2n0}(z_{2})}{\left(1+\underline{m}_{2n0}(z_{1}) \right)\left(1+\underline{m}_{2n0}(z_{2}) \right) }.  $ Therefore
		\begin{align*}
			\Theta_{1,n}(z_{1},z_{2})=&\frac{\partial}{\partial z_{2}}\left\lbrace \dfrac{\partial \mathcal{A}_{n}(z_{1},z_{2})}{\partial z_{1}}\dfrac{1}{1-\al_{x}\mathcal{A}_{n}(z_{1},z_{2})} \right\rbrace,\\=&\dfrac{\tilde{c} \underline{m}^{\prime}_{2n0}(z_{1})\underline{m}^{\prime}_{2n0}(z_{2}) }{\left(\left(1+\underline{m}_{2n0}(z_{1}) \right)\left(1+\underline{m}_{2n0}(z_{2}) \right)-\al_{x}\tilde{c}\underline{m}_{2n0}(z_{1})\underline{m}_{2n0}(z_{2}) \right)^{2}   } .
		\end{align*}
		For $ \Theta_{2,n}(z_{1},z_{2}), $ since
		$	\Theta_{2,n}(z_{1},z_{2})=\dfrac{z_{1}^{2}z_{2}^{2}\underline{m}_{2n0}^{\prime}(z_{1}) \underline{m}_{2n0}^{\prime}(z_{2})}{n}\sum_{i=1}^{p}\left[ \bGma^{*}\bbP_{n}^{2}(z_{1})\bGma\right]  _{ii}\left[ \bGma^{*}\bbP_{n}^{2}(z_{2})\bGma\right]  _{ii}, $
		and
		\begin{align*}
			\bGma^{*}\bbP_{n}^{2}(z_{1})\bGma=\bbU_{2}\bbV_{2}^{*}\left((1-c_{nM})\bbV_{2}\bbV_{2}^{*}-z_{1}c_{nM}m_{2n0}(z_{1})\bbV_{2}\bbV_{2}^{*}-z\bbI_{p} \right) ^{-2}\bbV_{2}\bbU_{2}^{*},
		\end{align*}	
		by using lemma \ref{inoutexchange}, we have
		\begin{align*}
			\bbP_{n}\left(z_{1} \right) =\dfrac{\underline{m}_{2n0}(z_{1})}{z\left(1+\underline{m}_{2n0}(z_{1}) \right) }\bbV_{2}\bbV_{2}^{*}-\frac{1}{z_1}\bbI_{p},
		\end{align*}
		then 
		\begin{align*}
			\bGma^{*}\bbP_{n}^{2}(z_{1})\bGma=&\bbU_{2}\bbV_{2}^{*}(\dfrac{\underline{m}_{2n0}(z_{1})}{z_{1}\left(1+\underline{m}_{2n0}(z_{1}) \right) }\bbV_{2}\bbV_{2}^{*}-\frac{1}{z_{1}}\bbI_{p} ) (\dfrac{\underline{m}_{2n0}(z_{1})}{z_{1}\left(1+\underline{m}_{2n0}(z_{1}) \right) }\bbV_{2}\bbV_{2}^{*}-\frac{1}{z_{1}}\bbI_{p} )\bbV_{2}\bbU_{2}^{*},\\
			=&\bbU_{2}\bbV_{2}^{*}( \dfrac{\underline{m}^{2}_{2n0}(z_{1})}{z_{1}^{2}\left(1+\underline{m}_{2n0}(z_{1}) \right)^{2} }  \bbV_{2}\bbV_{2}^{*}-\dfrac{2\underline{m}_{2n0}(z_{1})}{z_{1}^{2}\left(1+\underline{m}_{2n0}(z_{1}) \right) }  \bbV_{2}\bbV_{2}^{*}+\frac{1}{z_{1}^{2}}\bbI_{p} ) \bbV_{2}\bbU_{2}^{*},\\
			=&\bbU_{2}( \dfrac{\underline{m}^{2}_{2n0}(z_{1})}{z_{1}^{2}\left(1+\underline{m}_{2n0}(z_{1}) \right)^{2} }\bbI_{p-M}  -\dfrac{2\underline{m}_{2n0}(z_{1})}{z_{1}^{2}\left(1+\underline{m}_{2n0}(z_{1}) \right) } \bbI_{p-M} +\frac{1}{z_{1}^{2}}\bbI_{p-M}) \bbU_{2}^{*},\\
			=&\frac{1}{z_{1}^{2}\left(1+\underline{m}_{2n0}(z_{1}) \right)^{2} }\bbU_{2}\bbU_{2}^{*}.
		\end{align*}
		Therefore,
		\begin{align*}
			\sum_{i=1}^{p}[ \bGma^{*}\bbP_{n}^{2}(z_{1})\bGma]  _{ii}[ \bGma^{*}\bbP_{n}^{2}(z_{2})\bGma]_{ii}=\sum_{i=1}^{p}[  \frac{1}{z_{1}^{2}\left(1+\underline{m}_{2n0}(z_{1}) \right)^{2} }\bbU_{2}\bbU_{2}^{*}]_{ii} [  \frac{1}{z_{2}^{2}\left(1+\underline{m}_{2n0}(z_{2}) \right)^{2} }\bbU_{2}\bbU_{2}^{*}]_{ii}. 
		\end{align*}
		Since
		\begin{align*}
			&	[  \frac{1}{z_{1}^{2}\left(1+\underline{m}_{2n0}(z_{1}) \right)^{2} }\bbU_{2}\bbU_{2}^{*}]_{ii}=[ \frac{1}{z_{1}^{2}\left(1+\underline{m}_{2n0}(z_{1}) \right)^{2}}\bbI_{p}- \frac{1}{z_{1}^{2}\left(1+\underline{m}_{2n0}(z_{1}) \right)^{2} }\bbU_{1}\bbU_{1}^{*}]_{ii}\\&=\dfrac{1-\sum_{j=1}^{M}\left|u_{ij} \right|^{2}  }{z_{1}^{2}\left(1+\underline{m}_{2n0}(z_{1}) \right)^{2}},
		\end{align*}
		then 
		\begin{align*}
			&\sum_{i=1}^{p}\left[ \bGma^{*}\bbP_{n}^{2}(z_{1})\bGma\right]_ {ii}\left[ \bGma^{*}\bbP_{n}^{2}(z_{2})\bGma\right]_{ii}=\sum_{i=1}^{p}\dfrac{\left( 1-\sum_{j=1}^{M}\left|u_{ij} \right|^{2} \right)^{2}   }{z_{1}^{2}z_{2}^{2}\left(1+\underline{m}_{2n0}(z_{1}) \right)^{2}\left(1+\underline{m}_{2n0}(z_{2}) \right)^{2}},\\
			=& \dfrac{p-2M+\sum_{i=1}^{p}\left( \sum_{j=1}^{M}\left|u_{ij} \right|\right) ^{2}    }{z_{1}^{2}z_{2}^{2}\left(1+\underline{m}_{2n0}(z_{1}) \right)^{2}\left(1+\underline{m}_{2n0}(z_{2}) \right)^{2}}
			=\dfrac{p-2M+\sum_{j_{1},j_{2}=1}^{M}{\mathcal U_{j_{1}j_{1}j_{2}j_{2}}}  }{z_{1}^{2}z_{2}^{2}\left(1+\underline{m}_{2n0}(z_{1}) \right)^{2}\left(1+\underline{m}_{2n0}(z_{2}) \right)^{2}}.
		\end{align*}
		Then 
		\begin{align*}
			\Theta_{2,n}(z_{1},z_{2})=&\dfrac{z_{1}^{2}z_{2}^{2}\underline{m}_{2n0}^{\prime}(z_{1}) \underline{m}_{2n0}^{\prime}(z_{2})}{n}\sum_{i=1}^{p}\left[ \bGma^{*}\bbP_{n}^{2}(z_{1})\bGma\right]  _{ii}\left[ \bGma^{*}\bbP_{n}^{2}(z_{2})\bGma\right]  _{ii},\\
			=&
			\dfrac{z_{1}^{2}z_{2}^{2}\underline{m}_{2n0}^{\prime}(z_{1}) \underline{m}_{2n0}^{\prime}(z_{2})}{n}\dfrac{p-2M+\sum_{j_{1},j_{2}=1}^{M}{\mathcal U_{j_{1}j_{1}j_{2}j_{2}}}  }{z_{1}^{2}z_{2}^{2}\left(1+\underline{m}_{2n0}(z_{1}) \right)^{2}\left(1+\underline{m}_{2n0}(z_{2}) \right)^{2}}\\
			=&\dfrac{p-2M+\sum_{j_{1},j_{2}=1}^{M}{\mathcal U_{j_{1}j_{1}j_{2}j_{2}}} }{n}\dfrac{\underline{m}_{2n0}^{\prime}(z_{1}) \underline{m}_{2n0}^{\prime}(z_{2})}{\left(1+\underline{m}_{2n0}(z_{1}) \right)^{2}\left(1+\underline{m}_{2n0}(z_{2}) \right)^{2}}.
		\end{align*}
		Since the covariance of bulk part is $ -\frac{1}{4\pi^{2}}\oint_{\mathcal{C}_{1}}\oint_{\mathcal{C}_{2}}\left(z_{1}-\log\left(z_{1} \right)-1  \right) \left(z_{2}-\log\left(z_{2} \right)-1  \right)\vartheta_{n}^{2}dz_{1}dz_{2} $, where
		$ \vartheta_{n}^{2}=\Theta_{0,n}(z_{1},z_{2})+\al_{x}\Theta_{1,n}(z_{1},z_{2})+\beta_{x}\Theta_{2,n}(z_{1},z_{2}). $ By  contour integral calculations, we obtain the covariance equals
		$  -\log(1-c_{nM})-c_{nM}+ \al_{x}\left(-\log(1-\tilde{c})-\tilde{c} \right), $ where $\tilde{c}=\dfrac{p-2M+\sum_{s,t=1}^{p}\left(\sum_{i=1}^{M}u_{si}\bar{u}_{ti} \right)^{2} }{n}. $ Since $ \tilde{c}-c_{nM}\rightarrow0 $ as $ n\rightarrow\infty, $
		therefore 
		\begin{align*}	\breve{\varsigma}_{l}^{2} =&\sum_{k=1}^{K}\frac{\left(\phi_{n}\left(\al_{k} \right)-1  \right)^{2} }{n}s_{k}^{2}+(\al_{x}+1)\left(-\log(1-c_{nM})-c_{nM} \right),
		\end{align*} 
		and 
		\begin{align*}
			\frac{L-p\int f_{L}(x)dF^{c_{n},H_{n}}(x)-\breve{\mu}_{l}}{ \breve{\varsigma}_{l}}\stackrel{d}{\rightarrow}N(0,1).
		\end{align*}
		The proof of Theorem \ref{thm3} is finished. }
	
	
	
	\subsection{Proof of Theorem \ref{thm4}}
	First, we focus on the results under $ H_{0} $. From Lemma \ref{simple result}, we have
	\begin{align*}
		I_{1}(f_{W})&=c,	\\
		I_{2}(f_{W})&=c, \\
		J_{1}(f_{W},f_{W})&=4c^{3}+2c^{2},\\
		J_{2}(f_{W},f_{W})&=4c^{3},
	\end{align*}	
	which then yields
	\begin{align*}
		\mu_{w}&=\al_{x}I_{1}(f_{W})+\beta_{x}I_{2}(f_{W})=\al_{x}c+\beta_{x}c,\\
		\varsigma_{w}^{2}&=(\al_{x}+1)J_{1}(f_{W},f_{W})+\beta_{x}J_{2}(f_{W},f_{W})= (\al_{x}+1)(4c^{3}+2c^{2})+4\beta_{x}c^{3}.
	\end{align*}
	The results are still valid if $ c $ is replaced by $ c_{n} $. Moreover, the center term
	\begin{align}\label{center term}
		\int f_{W}(x) dF^{c_{n},H_{n}} = c_{n},	
	\end{align}	
	is a direct result of Lemma 2.2 in \cite{wang2013sphericity}.
	Therefore, from \cite{10.1214/14-AOS1292} or \cite{wang2013sphericity}, we have
	\begin{align*}
		\frac{W-p\int f_{W}(x)dF^{c_{n},H_{n}}-\mu_{w}  }{\varsigma_{w}}\stackrel{d}{\longrightarrow} N(0,1).
	\end{align*}
	Then, we focus on the results under $ H_{1} $. Note that
	\begin{align*}
		Y= \int f_{W}\left(x \right)dG_{n}\left( x\right)-\sum_{k=1}^{K}d_{k}f_{W}\left(\phi_{n}\left(\al_{k} \right)  \right)-\frac{M}{2\pi i}\oint_{\mathcal C}f_{W}\left(z \right)\frac{\underline{m}_{2n0}'(z)}{\underline{m}_{2n0}(z)}dz.
	\end{align*}	
	After some calculations, we obtain
	\begin{gather}
		\nonumber\int f_{W}\left(x \right)dG_{n}\left( x\right)=\mathrm{tr}(\bbB-\bbI_{p})^{2}-p\int f_{W}(x)dF^{c_{n},H_{n}}=W-p\int f_{W}(x)dF^{c_{n},H_{n}},\\
		\nonumber p\int f_{W}(x)dF^{c_{n},H_{n}}=(p-M)\int f_{W}(x)dF^{c_{nM},H_{2n}}=(p-M)c_{nM},\\
		\nonumber\sum_{k=1}^{K}d_{k}f_{W}\left(\phi_{n}\left(\al_{k} \right)  \right)=\sum_{k=1}^{K}d_{k}\left(\phi_{n}^{2}\left(\al_{k} \right)-2\phi_{n}\left(\al_{k} \right)+1 \right),\\
		\frac{M}{2\pi i}\oint_{\mathcal C}f_{W}\left(z \right)\frac{\underline{m}_{2n0}'(z)}{\underline{m}_{2n0}(z)}dz= -M c_{nM}^{2}.  
		\label{l_{n}}
	\end{gather} 
	For consistency, we present the proof of (\ref{l_{n}}) in Section \ref{sectionB}. Therefore, from Theorem \ref{thm1}, we have
	\begin{align*}
		\dfrac{W-(p-M)\breve{\ell}_{w}-\breve{\mu}_{w}}{\breve{\varsigma}_{w}}\stackrel{d}{\longrightarrow} N(0,1),
	\end{align*}	
	where 
	\begin{align*}
		\breve{\ell}_{w}&=c_{nM},~~~
		\breve{\mu}_{w}=\al_{x}c_{nM}+\beta_{x}c_{nM}+\sum_{k=1}^{K}d_{k}\left( \phi_{n}^{2}\left({\al}_{k} \right)-2\phi_{n}\left({\al}_{k} \right)+1 \right)-M c_{nM}^{2}, \\
		\breve{\varsigma}_{w}^{2} =&-\frac{1}{4\pi^{2}}\oint_{\mathcal{C}_{1}}\oint_{\mathcal{C}_{2}}\left( z_{1}-1\right)^{2} \left( z_{2}-1\right)^{2}\vartheta_{n}^{2}dz_{1}dz_{2}+ \sum_{k=1}^{K}\frac{4\phi_{n}^{2}\left(\al_{k} \right) \left(\phi_{n}\left(\al_{k} \right)-1  \right)^{2} }{n}s_{k}^{2}.
	\end{align*}
	\textcolor{black}{Since $ \vartheta_{n}^{2}=\Theta_{0,n}(z_{1},z_{2})+\al_{x}\Theta_{1,n}(z_{1},z_{2})+\beta_{x}\Theta_{2,n}(z_{1},z_{2}), $ and $ \Theta_{0,n}(z_{1},z_{2}), \Theta_{1,n}(z_{1},z_{2}),$ and $ \Theta_{2,n}(z_{1},z_{2}) $ are calculated in the proof of Theorem \ref{thm3}, then by some calculations we obtain $ -\frac{1}{4\pi^{2}}\oint_{\mathcal{C}_{1}}\oint_{\mathcal{C}_{2}}\left( z_{1}-1\right)^{2} \left( z_{2}-1\right)^{2}\vartheta_{n}^{2}dz_{1}dz_{2} $ equals $4c_{nM}^{3}+2c_{nM}^{2}+\al_{x}\left(4\tilde{c}^{3}+2\tilde{c}^{2} \right)+4\beta_{x}\check{c}^{3}, $ where
		$\check{c}=\dfrac{p-2M+\sum_{j_{1},j_{2}=1}^{M}{\mathcal U_{j_{1}j_{1}j_{2}j_{2}}}}{n}$. Since $ \tilde{c}-c_{nM}\rightarrow0, $ $ \check{c}-c_{nM}\rightarrow0 $ as $ n\rightarrow\infty, $
		therefore 
		\begin{align*} \breve{\varsigma}_{w}^{2} =&\sum_{k=1}^{K}\frac{4\phi_{n}^{2}\left(\al_{k} \right) \left(\phi_{n}\left(\al_{k} \right)-1  \right)^{2} }{n}s_{k}^{2}+(\al_{x}+1)\left(4c_{nM}^{3}+2c_{nM}^{2} \right)+4\beta_{x}c_{nM}^{3},
		\end{align*} and
		\begin{align*}
			\dfrac{W-(p-M)\breve{\ell}_{W}-\breve{\mu}_{w}}{\breve{\varsigma}_{w}}\stackrel{d}{\longrightarrow} N(0,1),
		\end{align*}
		then the proof is finished. }
	
	\subsection{Proof of Theorem \ref{CLRT power} }
	\textcolor{black}{
		Let $ \xi $ be the significance level, and  $ z_{\xi} $ is the  $ 1-\xi $ quantile of the standard Gaussian distribution $ \Phi$. Since
		$$ \xi=P\left( L >z_{\xi}\varsigma_{l}+p\ell_{l}+\mu_{l}   \right), 
		$$ 
		for brevity, we denote $ L_{0}=p\ell_{l}+\mu_{l} $, 
		$ L_{1}=(p-M)\breve{\ell}_{l}   +\breve{\mu}_{l}. $ Therefore, the power to detect the hypothesis is
		\begin{align*}
			P_L=P\left( L>z_{\xi}\varsigma_{l}+L_0 \right)&=P\left(\frac{ L -L_{1}}{\breve{\varsigma}_{l}} >\frac{z_{\xi}\varsigma_{l}+L_0-L_{1}}{\breve{\varsigma}_{l}}\right)
		\end{align*}
		Since $ \frac{ L -L_{1}}{\breve{\varsigma}_{l}} $ is asymptotically normal distributed, then $ P_L $ is approxiamte to $ \Phi\left( \frac{L_{1}-L_{0}}{\breve{\varsigma}_{l}}-z_{\xi}\frac{\varsigma_{l}}{\breve{\varsigma}_{l}}\right). $
		After some elementary calculations,   we obtain as $ n\rightarrow \infty, $ $$ L_{1}-L_{0}\rightarrow-Mc+\sum_{k=1}^{K}d_{k}(\phi_{k}-\log \phi_{k}-1 ), $$
		$$ \varsigma_{l}\rightarrow\sqrt{(\alpha_x+1) (-\log(1-c)-c)},   $$
		$$ \breve{\varsigma}_{l}-\sqrt{(\al_{x}+1) (-\log(1-c)-c) +\sum_{k=1}^{K}\frac{\left( \phi_{n}\left(\al_{k} \right)-1\right) ^{2} }{n}}s_k^2\rightarrow0.     $$
		Therefore, we have as $ n $ tends to infinity,
		\begin{align} 
			P_{L}- \Phi\left(  \frac{  \sum_{k=1}^{K}d_k\left( \phi_{n}(\alpha_{k})-\log\phi_{n}(\alpha_{k}) \right)-M-Mc-z_{\xi}\varsigma_{l}  }{ \sqrt{(\al_{x}+1) (-\log(1-c)-c) +\sum_{k=1}^{K}\frac{\left( \phi_{n}\left(\al_{k} \right)-1\right) ^{2} }{n}}s_k^2 }   \right)  \rightarrow 0, 
		\end{align}
		then the proof of Theorem \ref{CLRT power} is finished.} 
	\\
	
	\subsection{Proof of Theorem \ref{CNTT power}}
	\textcolor{black}{
		Since
		$ 
		\xi=P\left( W >z_{\xi}\varsigma_{w}+c_{n}\left(p+\al_{x}+\beta_{x} \right)  \right),$
		for brevity, we use the notation $ W_{0}=pc_{n}+c_n(\al_x+\beta_{x}), $ $ W_{1}=(p-M)\int f_{W}(x)dF^{c_{nM},H_{2n}}   +\breve{\mu}_{w}. $ Therefore, the power to detect the hypothesis is
		\begin{align*}
			P_W=P\left( W >z_{\xi}\varsigma_{w}+W_0 \right)&=P\left(\frac{W -W_{1}}{\breve{\varsigma}_{w}} >\frac{z_{\xi}\varsigma_{w}+W_0-W_{1}}{\breve{\varsigma}_{w}}\right)
		\end{align*}
		Since $ \frac{W -W_{1}}{\breve{\varsigma}_{w}} $ is asymptotically normal distributed, then $ P_W $ is approxiamted to $ \Phi( \frac{W_{1}-W_{0}}{\breve{\varsigma}_{w}}-z_{\xi}\frac{\varsigma_{w}}{\breve{\varsigma}_{w}}). $
		Here
		$
		W_{1}-W_{0}=(p-M)c_{nM}-pc_{n}+( \beta_{x}+\al_x)c_{nM}-( \beta_{x}+\al_x)c_{n}+\sum_{k=1}^{K}d_{k}( \phi_{n}\left(\al_{k} \right)-1 )^{2}-Mc_{nM}^{2}.     
		$
		After some elementary calculations, we obtain, as $ n\rightarrow \infty $,
		$$
		W_{1}-W_{0}\rightarrow \sum_{k=1}^{K}d_k\left( \phi_{n}\left(\al_{k} \right)-1 \right)^{2}-Mc^{2}-2Mc,$$
		$$	\varsigma_{w}\rightarrow \sqrt{(\al_x+1)(4c^3+2c^2)+4\beta_xc^3},$$
		$$	\breve{\varsigma}_{w}-\sqrt{(\al_x+1)(4c^3+2c^2)+4\beta_xc^3+ \sum_{k=1}^{K}\dfrac{4\phi_{n}^{2}\left(\al_{k} \right)\left(\phi_{n}\left(\al_{k} \right)-1 \right)^{2}   }{n}s_k^2}\rightarrow 0.$$
		Then we have as $ n $ tends to infinity,  
		\begin{align} 
			P_{W}- \Phi\left(  \frac{\sum_{k=1}^{K}d_k\left( \phi_{k}-1 \right)^{2}-Mc^{2}-2Mc-z_{\xi}\varsigma_w  }{ \sqrt{ (\al_x+1)(4c^3+2c^2)+4\beta_xc^3+ \sum_{k=1}^{K}\dfrac{4\phi_{n}^{2}\left(\al_{k} \right)\left(\phi_{n}\left(\al_{k} \right)-1 \right)^{2}   }{n}s_k^2  }   }   \right)  \rightarrow 0,
		\end{align}	
		then the proof is finished.}
	\\
	
	\subsection{Proof of Theorem \ref{RLRT power}}
	\textcolor{black}{
		From \cite{JiangB21G}, for spike $ \al_{1} $, we eliminate the multiplicity of it and then we have
		\begin{align*}
			\sqrt{\frac{n\theta_{1}^{2}}{2\theta_{1}+\sum_{t=1}^{p}\left| u_{t1}\right|^{4}\beta_{x}\nu_{1} }}\frac{\lambda_{1}-\phi_{n}\left(\al_{1} \right) }{\phi_{n}\left(\al_{1} \right)}\stackrel{d}{\longrightarrow}N\left(0,1 \right). 
		\end{align*}
		Then the power of test R  equals
		\begin{align*}
			P_{R} &= P\left(\lambda_{1}> t_{\xi}\varsigma_{r}+\mu_{r} \right) \\
			&= P\left( \sqrt{\frac{n\theta_{1}^{2}}{2\theta_{1}+\sum_{t=1}^{p}\left| u_{t1}\right|^{4}\beta_{x}\nu_{1} }}\frac{\lambda_{1}-\phi_{n}\left(\al_{1}\right)}{\phi_{n}\left(\al_{1}\right)} >\notag\right.
			\\
			\phantom{=\;\;}&
			\left.\sqrt{\frac{n\theta_{1}^{2}}{2\theta_{1}+\sum_{t=1}^{p}\left| u_{t1}\right|^{4}\beta_{x}\nu_{1} }} \frac{t_{\xi}\varsigma_{r}+\mu_{r}-\phi_{n}\left(\al_{1}\right)}{\phi_{n}\left(\al_{1}\right)} \right) 
		\end{align*}
		Since $ \sqrt{\frac{n\theta_{1}^{2}}{2\theta_{1}+\sum_{t=1}^{p}\left| u_{t1}\right|^{4}\beta_{x}\nu_{1} }}\frac{\lambda_{1}-\phi_{n}\left(\al_{1} \right) }{\phi_{n}\left(\al_{1} \right)} $ is asymptotically standard normal distributed, then $ P_R $ is approximate to  $ 1-\Phi\left(\sqrt{\frac{n\theta_{1}^{2}}{2\theta_{1}+\sum_{t=1}^{p}\left| u_{t1}\right|^{4}\beta_{x}\nu_{1} }} \frac{t_{\xi}\varsigma_{r}+\mu_{r}-\phi_{n}\left(\al_{1}\right)}{\phi_{n}\left(\al_{1}\right)} \right), $ and it equals \\$ \Phi( -\sqrt{n} \frac{t_{\xi}\varsigma_{r}+\mu_{r}-\phi_{n}\left(\al_{1}\right)}{s_{1}\phi_{n}\left(\al_{1}\right)} ), $
		then the proof is finished.}

	\section{Some deviations and calculations}\label{sectionB}
	This section contains proof of formulas stated in the proof of Theorems \ref{thm3} and \ref{thm4}. \textbf{We begin by deriving formula (\ref{31})}.
	First, we consider $ \oint_{\mathcal C}f_{L}\left(z \right)\frac{\underline{m}'(z)}{\underline{m}(z)}dz $. 
	\begin{align}
		&\nonumber\oint_{\mathcal C}f_{L}\left(z \right)\frac{\underline{m}'(z)}{\underline{m}(z)}dz
		=\nonumber\oint_{\mathcal C}f_{L}\left(z \right)d\log\underline{m}\left(z \right)=\nonumber-\oint_{\mathcal C}f_{L}^{'}\left(z \right)\log\underline{m}\left(z \right)dz \\
		=&\nonumber\int_{a(c)}^{b(c)}f_{L}^{'}\left(z \right)\left[\log\underline{m}(x+i\varepsilon)-\log\underline{m}(x-i\varepsilon) \right]dx\\
		=&2i 	\int_{a(c)}^{b(c)}f_{L}^{'}\left(z \right) \Im \log\underline{m}(x+i\varepsilon)dx \label{26}
	\end{align}	
	Here, $a(c)=(1-\sqrt{c})^{2}  $ and $b(c)=(1+\sqrt{c})^{2}  $. Since $$\underline{m}\left(z \right)=-\frac{1-c}{z}+cm\left(z \right),  $$
	under $ H_{1} $, we have $$ \underline{m}\left(z \right)=\frac{-(z+1-c)+\sqrt{(z-1-c)^{2}-4c}}{2z} . $$
	As $z \rightarrow x \in$ $[a(c), b(c)]$, we obtain
	$$ \underline{m}\left(x \right)= \frac{-(x+1-c)+\sqrt{4c-(x-1-c)^{2}}i}{2x}.$$
	Therefore,
	\begin{align*}
		&\int_{a(c)}^{b(c)}f_{L}^{'}\left(z \right) \Im \log\underline{m}(x+i\varepsilon)dx\\
		&=\int_{a(c)}^{b(c)} f_{L}^{'}(x)\tan ^{-1}\left(\frac{\sqrt{4 c-(x-1-c)^{2}}}{-(x+1-c)}\right) d x\\
		&=\left[\left.\tan ^{-1}\left(  \frac{\sqrt{4 c-(x-1-c)^{2}}}{-(x+1-c)}\right)  f_{L}(x)\right|_{a(c)} ^{b(c)}-\int_{a(c)}^{b(c)} f_{L}(x) d \tan ^{-1}\left(\frac{\sqrt{4 c-(x-1-c)^{2}}}{-(x+1-c)}\right)\right].
	\end{align*}
	It is easy to verify that the first term is $ 0 $, and we now focus on the second term,
	\begin{align}
		&\nonumber\int_{a(c)}^{b(c)} f_{L}(x) d \tan ^{-1}\left(\frac{\sqrt{4 c-(x-1-c)^{2}}}{-(x+1-c)}\right)\\
		&=\int_{a(c)}^{b(c)} \frac{\left(x-\log x-1 \right)}{1+\frac{4c-(x-1-c)^{2}}{(x+1-c)^{2}}}\cdot\frac{\sqrt{4c-(x-1-c)^{2}}+\frac{(x-1-c)(x+1-c)}{\sqrt{4c-(x-1-c)^{2}}}}{(x+1-c)^{2}}dx.\label{27}
	\end{align}
	By substituting $ x=1+c-2\sqrt{c}\cos(\theta) $, we obtain
	\begin{align}
		(\ref{27})&=\nonumber\frac{1}{2}\int_{0}^{2\pi}\left(1+c-2\sqrt{c}\cos(\theta)-\log\left(1+c-2\sqrt{c}\cos(\theta) \right)-1  \right)\frac{c-\sqrt{c}\cos(\theta) }{1+c-2\sqrt{c}\cos(\theta)} d\theta\\
		&=\nonumber\frac{1}{2}\int_{0}^{2\pi}\left[1-\frac{\log\left(1+c-2\sqrt{c}\cos(\theta) \right)+1}{1+c-2\sqrt{c}\cos(\theta)}  \right] \left(c-\sqrt{c}\cos(\theta) \right)d\theta\\
		&=\frac{1}{2}\int_{0}^{2\pi}\left(c-\sqrt{c}\cos(\theta) \right)d\theta-\frac{1}{2}\int_{0}^{2\pi}\frac{\log\left(1+c-2\sqrt{c}\cos(\theta) \right)}{1+c-2\sqrt{c}\cos(\theta)}\left( c-\sqrt{c}\cos(\theta)\right)  d\theta-\label{28}\\
		&\nonumber\frac{1}{2}\int_{0}^{2\pi}\frac{c-\sqrt{c}\cos(\theta)}{1+c-2\sqrt{c}\cos(\theta)}d\theta
	\end{align}
	It is easy to obtain that the first term of (\ref{28}) is $ \pi c $; then, we consider the second term. By substituting $ \cos\theta=\frac{z+z^{-1}}{2} $, we turn it into a contour integral on $ \left|z \right|=1  $
	\begin{align*}
		&\frac{1}{2}\int_{0}^{2\pi}\frac{\log\left(1+c-2\sqrt{c}\cos(\theta) \right)}{1+c-2\sqrt{c}\cos(\theta)}\left( c-\sqrt{c}\cos\theta\right)  d\theta\\
		&=\frac{1}{2} \oint_{\left| z\right| =1} \log |1-\sqrt{c} z|^{2} \cdot \frac{c-\sqrt{c} \frac{z+z^{-1}}{2}}{1+c-2\sqrt{ c} \cdot \frac{z+z^{-1}}{2}} \frac{d z}{i z}\\
		&=\frac{1}{4 i} \oint_{\left| z\right| =1} \log |1-\sqrt{c} z|^{2} \cdot \frac{2c z-\sqrt{c}\left(z^{2}+1\right)}{(z-\sqrt{c})(-\sqrt{c}z+1)z} d z
	\end{align*}
	When $ c<1 $, $ 0 $ and $\sqrt{c}$ are poles, by using the residue theorem, the integral is $ -\pi \log(1-c)$. The same argument also holds for the third term, and the integral is $ 0 $ after some calculation.

	Therefore, $$ \frac{M}{2\pi i}\oint_{\mathcal C}f_{L}\left(z \right)\frac{\underline{m}'(z)}{\underline{m}(z)}dz= -M(c+\log(1-c)), $$ 
	and the result is still valid if $ c $ is replaces $ c_{nM} $; therefore, formula (\ref{31}) holds.

	\textbf{Now, we prove (\ref{l_{n}})}. 
	Since $ z=-\frac{1}{\underline{m}}+\frac{c}{1+\underline{m}} $, we have, for $ c>1 $,
	\begin{align*}
		\oint_{\mathcal C}f_{W}\left(z \right)\frac{\underline{m}'(z)}{\underline{m}(z)}dz=\oint_{\mathcal C_{1}}f_{W}\left(z \right)\frac{\underline{m}'(z)}{\underline{m}(z)}dz+\oint_{\mathcal C_{2}}f_{W}\left(z \right)\frac{\underline{m}'(z)}{\underline{m}(z)}dz,
	\end{align*}
	where $\mathcal C_{1}  $ is a contour that includes the interval $ ( (1-\sqrt{c})^{2}, (1+\sqrt{c})^{2}) $, and $\mathcal C_{2}  $ is a contour that includes the origin. Using $ \mathcal C_{\underline{m}} $ to denote the contour of $ \underline{m} $, we obtain
	\begin{align*}
		&\oint_{\mathcal C_{1}}f_{W}\left(z \right)\frac{\underline{m}'(z)}{\underline{m}(z)}dz=\oint_{\mathcal C_{\underline{m}}}(-\frac{1}{\underline{m}}+\frac{c}{1+\underline{m}}-1)^{2}\frac{\underline{m}'(z)}{\underline{m}(z)}\frac{dz}{d\underline{m}}d\underline{m}\\
		=&\oint_{\mathcal C_{\underline{m}}}(-\frac{1+\underline{m}}{\underline{m}}+\frac{c}{1+\underline{m}})^{2}\frac{1}{\underline{m}}d\underline{m}=\oint_{\mathcal C_{\underline{m}}}( \frac{(1+\underline{m})^{2}}{\underline{m}^{3}}+\frac{c^{2}}{(1+\underline{m})^{2}\underline{m}}-\frac{2c}{\underline{m}^{2}}  ) d\underline{m}
	\end{align*}	 
	Since the $ z $ contour cannot enclose the origin, neither can the resulting $ \underline{m} $ contour. Thus, the only pole is $ -1 $, the residue is $ -c^{2} $ by residue theorem, and we obtain the integral as $ -2\pi ic^{2}. $ 
	
	Then, we focus on the second integral $ \oint_{\mathcal C_{2}}f_{W}\left(z \right)\frac{\underline{m}'(z)}{\underline{m}(z)}dz. $ When $ z=0, $ we obtain $ \underline{m}=\frac{1}{c-1} $; since $ c>1,$ $ \frac{1}{c-1}>0. $ Both $ \underline{m}=0 $ and $ \underline{m}=-1 $ are not in the contour. Thus, the integrand $ ( \frac{(1+\underline{m})^{2}}{\underline{m}^{3}}+\frac{c^{2}}{(1+\underline{m})^{2}\underline{m}}-\frac{2c}{\underline{m}^{2}}  ) $ is analytic in the contour. The integral is $ 0 $. 
	Therefore, when $ c>1 $,  $ \frac{M}{2\pi i}\oint_{\mathcal C}f_{W}\left(z \right)\frac{\underline{m}'(z)}{\underline{m}(z)}dz=-M c^{2}$. When $ c<1, $, the contour integral $ \oint_{\mathcal C}f_{W}\left(z \right)\frac{\underline{m}'(z)}{\underline{m}(z)}dz $ reduces to $ \oint_{\mathcal C_{1}}f_{W}\left(z \right)\frac{\underline{m}'(z)}{\underline{m}(z)}dz $, and the result is also the same as above. When $ c=1 $, the result is still true by continuity in $ c $. The results above are still valid if $ c $  is replaced by $ c_{nM}. $ Therefore, the proof of (\ref{l_{n}}) is complete.
	\\
	\\
	\\
	\\
	\\
	\\

	\section{Some useful lemmas}\label{some useful lemmas}
	\begin{lemma}\label{simple result}
		If $ \bbD_{2}=\bbI_{p-M} $, then the mean function $\mu_1$ and $ \kappa_{nst} $ in the covariance function of Theorem \ref{thm2}  can be simplified from the results in \cite{wang2013sphericity} and \cite{10.1214/14-AOS1292}, i.e.,    
		$$
		\mu_{1}=\alpha_x I_{1}(f_{1})+\beta_{x} I_{2}(f_{1}),
		$$
		$$\kappa_{nst}=(\alpha_x+1)J_{1}(f_{s}, f_{t})+\beta_{x} J_{2}(f_{s}, f_{t}),$$
		\begin{align*}
			I_{1}(f_{1})&=\lim _{r \downarrow 1} \frac{1}{2 \pi i} \oint_{|z|=1} f_{1}\left(\left|1+\sqrt{c_{nM}} z\right|^{2}\right)\left[\frac{z}{z^{2}-r^{-2}}-\frac{1}{z}\right] d z,\\
			I_{2}(f_{1})&=\frac{1}{2 \pi i} \oint_{|z|=1} f_{1}\left(\left|1+\sqrt{c_{nM}} z\right|^{2}\right) \frac{1}{z^{3}} d z,\\
			J_{1}(f_{s}, f_{t})&=\lim _{r \downarrow 1} \frac{-1}{4 \pi^{2}} \oint_{\left|z_{1}\right|=1} \oint_{\left|z_{2}\right|=1} \frac{f_{s}\left(\left|1+\sqrt{c_{nM}} z_{1}\right|^{2}\right) f_{t}\left(\left|1+\sqrt{c_{nM}} z_{2}\right|^{2}\right)}{\left(z_{1}-r z_{2}\right)^{2}} d z_{1} d z_{2},\\
			J_{2}(f_{s}, f_{t})&=-\frac{1}{4 \pi^{2}} \oint_{\left|z_{1}\right|=1} \frac{f_{s}\left(\left|1+\sqrt{c_{nM}} z_{1}\right|^{2}\right)}{z_{1}^{2}} d z_{1} \oint_{\left|z_{2}\right|=1} \frac{f_{t}\left(\left|1+\sqrt{c_{nM}} z_{2}\right|^{2}\right)}{z_{2}^{2}} d z_{2}.
		\end{align*}
		
	\end{lemma}
	
	\begin{lemma}\label{inoutexchange}
		Note that for any matrix $\bbZ$, 
		\begin{align*}
			\bbZ\left(\bbZ^{\ast}\bbZ-\lambda \bbI \right)^{-1}\bbZ^{\ast}=\bbI+\lambda\left(\bbZ\bbZ^{\ast}-\lambda \bbI \right)^{-1}.
		\end{align*}
	\end{lemma}
	
	\section{Tables for simulation studies}\label{other simulations}
	In this section, we present addtional simulation tables regarding empirical probability of rejecting alternative hypotheses in Section \ref{simulation} in the main file.
	\begin{table*}
		\caption{Empirical probability of rejecting $ H_{1} $ at significance level $ \xi=0.01 $ under assumptions of Gaussian, Gamma, and Uniform distributions }
		\label{tablepowerd1h1l2}
		\begin{tabular}{@{}lrrrrrrrrrrc@{}}
			\hline &&\multicolumn{3}{c}{($ Dt_{1},H_1 $)} & \multicolumn{3}{c}{($ Dt_{2},H_1 $)} & \multicolumn{3}{c}{($ Dt_{3},H_1 $)}\\
			\cline{3-11}
			test & \multicolumn{1}{c}{
				(p,n) } 
			&\multicolumn{1}{c}{$\alpha_1=3$} & \multicolumn{1}{c}{$\alpha_1=5$} & \multicolumn{1}{c}{$\alpha_1=7$} &
			\multicolumn{1}{c}{$\alpha_1=3$} & \multicolumn{1}{c}{$\alpha_1=5$} & \multicolumn{1}{c}{$\alpha_1=7$} &
			\multicolumn{1}{c}{$\alpha_1=3$} & \multicolumn{1}{c}{$\alpha_1=5$} & \multicolumn{1}{c}{$\alpha_1=7$}
			\\
			\hline
			$CLRT$  & (50,150)  &  0.5142 & 0.9954 & 1  &0.4896&0.9896&1&0.5130&0.9995&1 \\
			&(100,300) & 0.5056  & 0.9992 & 1 &0.4954&0.9972&1&0.5134&0.9999&1 \\
			&(200,600)  & 0.5178  & 0.9996 & 1&0.5085&0.9990&1&0.5210&1&1 \\
			$CNTT$   & (50,150)  & 0.9814  & 1 & 1 &0.9287&1&1&0.9984&1&1\\
			&(100,300)   &  0.9938 & 1 & 1 &0.9662&1&1&0.9995&1&1\\
			&(200,600)   & 0.9985 & 1 & 1 &0.9873&1&1&1&1&1\\
			$RLRT$    & (50,150) & 0.9983  & 1 & 1 &0.9947&1&1&1&1&1\\
			&(100,300)   & 1  & 1 & 1 &1&1&1&1&1&1\\
			&(200,600) & 1 & 1 & 1 &1&1&1&1&1&1\\
			\hline
		\end{tabular}
	\end{table*}	
	
	\begin{table*}
		\caption{Empirical probability of rejecting $ H_{2} $ at significance level $ \xi=0.01 $ under assumptions of Gaussian, Gamma, and Uniform distributions }
		\label{tablepowerd2h2l2}
		\begin{tabular}{@{}lrrrrrrrrrr@{}}
			\hline &&\multicolumn{3}{c}{($ Dt_{1},H_2 $)} & \multicolumn{3}{c}{($ Dt_{2},H_2 $)} & \multicolumn{3}{c}{($ Dt_{3},H_2 $)}\\
			\cline{3-11}
			test & \multicolumn{1}{c}{
				(p,n) } 
			&\multicolumn{1}{c}{$\alpha_1=3$} & \multicolumn{1}{c}{$\alpha_1=5$} & \multicolumn{1}{c}{$\alpha_1=7$} &
			\multicolumn{1}{c}{$\alpha_1=3$} & \multicolumn{1}{c}{$\alpha_1=5$} & \multicolumn{1}{c}{$\alpha_1=7$} &
			\multicolumn{1}{c}{$\alpha_1=3$} & \multicolumn{1}{c}{$\alpha_1=5$} & \multicolumn{1}{c}{$\alpha_1=7$}
			\\
			\hline
			$CLRT$  & (50,150)  &  0.9387 & 1 & 1&0.9045&1&1&0.9574&1&1  \\
			&(100,300) & 0.9496  & 1 & 1&0.9330&1&1&0.9666&1&1 \\
			&(200,600)  & 0.9624  & 1 & 1& 0.9553&1&1&0.9710&1&1\\
			$CNTT$   & (50,150)  & 1  & 1 & 1 & 0.9998  & 1 & 1& 1  & 1 & 1\\
			&(100,300)   &  1 & 1 & 1 & 1  & 1 & 1& 1  & 1 & 1\\
			&(200,600)   & 1 & 1 & 1 & 1  & 1 & 1& 1  & 1 & 1\\
			$RLRT$    & (50,150) & 1  & 1 & 1 & 1  & 1 & 1& 1  & 1 & 1\\
			&(100,300)   & 1  & 1 & 1 & 1  & 1 & 1& 1  & 1 & 1\\
			&(200,600) & 1 & 1 & 1 & 1  & 1 & 1& 1  & 1 & 1\\
			\hline
		\end{tabular}
	\end{table*}
	
	\begin{table*}
		\caption{Empirical probability of rejecting $ H_{4} $ at significance level $ \xi=0.05 $ under assumptions of Gaussian, Gamma, and Uniform distributions }
		\label{tablepowerd1h4l1}
		\begin{tabular}{@{}lrrrrrrrrrr@{}}
			\hline &&\multicolumn{3}{c}{($ Dt_{1},H_4 $)} & \multicolumn{3}{c}{($ Dt_{2},H_4 $)} & \multicolumn{3}{c}{($ Dt_{3},H_4 $)}\\
			\cline{3-11}
			test & \multicolumn{1}{c}{
				(p,n) } 
			&\multicolumn{1}{c}{$\alpha_1=3$} & \multicolumn{1}{c}{$\alpha_1=5$} & \multicolumn{1}{c}{$\alpha_1=7$} &
			\multicolumn{1}{c}{$\alpha_1=3$} & \multicolumn{1}{c}{$\alpha_1=5$} & \multicolumn{1}{c}{$\alpha_1=7$} &
			\multicolumn{1}{c}{$\alpha_1=3$} & \multicolumn{1}{c}{$\alpha_1=5$} & \multicolumn{1}{c}{$\alpha_1=7$}
			\\
			\hline
			$CLRT$  & (50,150)  &  0.7288 & 0.9991 & 1  &0.7074&0.9989&1&0.7269&0.9993&1 \\
			&(100,300) & 0.7397  & 0.9999 & 1 &0.7364 &0.9997&1&0.7411&0.9997&1 \\
			&(200,600)  & 0.7568  & 1 & 1&0.7484&1&1&0.7649&1&1 \\
			$CNTT$   & (50,150)  & 0.9917  & 1 & 1 &0.9821&1&1&0.9974&1&1\\
			&(100,300)   &  0.9985 & 1 & 1 &0.9951&1&1&0.9997&1&1\\
			&(200,600)   & 0.9997 & 1 & 1 &0.9983&1 &1&1&1&1\\
			$RLRT$    & (50,150) & 0.9997  & 1 & 1 &0.9992&1&1&0.9996&1&1\\
			&(100,300)   & 1  & 1 & 1 &1&1&1&1&1&1\\
			&(200,600) & 1 & 1 & 1 &1&1&1&1&1&1\\
			\hline
		\end{tabular}
	\end{table*}	
	
	\begin{table*}
		\caption{Empirical probability of rejecting $ H_{4} $ at significance level $ \xi=0.01 $ under assumptions of Gaussian, Gamma, and Uniform distributions }
		\label{tablepowerd1h4l2}
		\begin{tabular}{@{}lrrrrrrrrrr@{}}
			\hline &&\multicolumn{3}{c}{($ Dt_{1},H_4 $)} & \multicolumn{3}{c}{($ Dt_{2},H_4 $)} & \multicolumn{3}{c}{($ Dt_{3},H_4 $)}\\
			\cline{3-11}
			test & \multicolumn{1}{c}{
				(p,n) } 
			&\multicolumn{1}{c}{$\alpha_1=3$} & \multicolumn{1}{c}{$\alpha_1=5$} & \multicolumn{1}{c}{$\alpha_1=7$} &
			\multicolumn{1}{c}{$\alpha_1=3$} & \multicolumn{1}{c}{$\alpha_1=5$} & \multicolumn{1}{c}{$\alpha_1=7$} &
			\multicolumn{1}{c}{$\alpha_1=3$} & \multicolumn{1}{c}{$\alpha_1=5$} & \multicolumn{1}{c}{$\alpha_1=7$}
			\\
			\hline
			$CLRT$  & (50,150)  &  0.5068 & 0.9967 & 1  &0.5043 &0.9952 &1 &0.5101&0.9966&1 \\
			&(100,300) & 0.5137  & 0.9996 & 1 &0.5082 &0.9987&1&0.5193&0.9996&1 \\
			&(200,600)  & 0.5155  & 0.9995 & 1&0.5192 &0.9997&1&0.5100&0.9998&1 \\
			$CNTT$   & (50,150)  & 0.9782  & 1 & 1 &0.9469&1&1&0.9934&1&1\\
			&(100,300)   &  0.9935 & 1 & 1 &0.9803&1&1&0.9988&1&1\\
			&(200,600)   & 0.9984 & 1 & 1 &0.9921&1 &1&0.9997&1&1\\
			$RLRT$    & (50,150) & 0.9978  & 1 & 1 &0.9991&1&1&0.9981&1&1\\
			&(100,300)   & 1  & 1 & 1 &1&1&1&1&1&1\\
			&(200,600) & 1 & 1 & 1 &1&1&1&1&1&1\\
			\hline
		\end{tabular}
	\end{table*}	
	
	\begin{table*}
		\caption{Empirical probability of rejecting $ H_{5} $ at significance level $ \xi=0.05 $ under assumptions of Gaussian, Gamma, and Uniform distributions }
		\label{tablepowerd1h5l1}
		\begin{tabular}{@{}lrrrrrrrrrr@{}}
			\hline &&\multicolumn{3}{c}{($ Dt_{1},H_5 $)} & \multicolumn{3}{c}{($ Dt_{2},H_5 $)} & \multicolumn{3}{c}{($ Dt_{3},H_5 $)}\\
			\cline{3-11}
			test & \multicolumn{1}{c}{
				(p,n) } 
			&\multicolumn{1}{c}{$\alpha_1=3$} & \multicolumn{1}{c}{$\alpha_1=5$} & \multicolumn{1}{c}{$\alpha_1=7$} &
			\multicolumn{1}{c}{$\alpha_1=3$} & \multicolumn{1}{c}{$\alpha_1=5$} & \multicolumn{1}{c}{$\alpha_1=7$} &
			\multicolumn{1}{c}{$\alpha_1=3$} & \multicolumn{1}{c}{$\alpha_1=5$} & \multicolumn{1}{c}{$\alpha_1=7$}
			\\
			\hline
			$CLRT$  & (50,150)  &  0.9803 & 1 & 1  &0.9778 &1&1&0.9819&1&1 \\
			&(100,300) & 0.9889  & 1 & 1 &0.9864 &1&1&0.9861&1&1 \\
			&(200,600)  & 0.9928  & 1 & 1&0.9913&1&1&0.9908&1&1 \\
			$CNTT$   & (50,150)  & 1  & 1 & 1 &1&1&1&1&1&1\\
			&(100,300)   &  1 & 1 & 1 &1&1&1&1&1&1\\
			&(200,600)   & 1 & 1 & 1 &1&1 &1&1&1&1\\
			$RLRT$    & (50,150) & 1  & 1 & 1 &1&1&1&1&1&1\\
			&(100,300)   & 1  & 1 & 1 &1&1&1&1&1&1\\
			&(200,600) & 1 & 1 & 1 &1&1&1&1&1&1\\
			\hline
		\end{tabular}
	\end{table*}	
	
	\begin{table*}
		\caption{Empirical probability of rejecting $ H_{5} $ at significance level $ \xi=0.01 $ under assumptions of Gaussian, Gamma, and Uniform distributions }
		\label{tablepowerd1h5l2}
		\begin{tabular}{@{}lrrrrrrrrrr@{}}
			\hline &&\multicolumn{3}{c}{($ Dt_{1},H_5 $)} & \multicolumn{3}{c}{($ Dt_{2},H_5 $)} & \multicolumn{3}{c}{($ Dt_{3},H_5 $)}\\
			\cline{3-11}
			test & \multicolumn{1}{c}{
				(p,n) } 
			&\multicolumn{1}{c}{$\alpha_1=3$} & \multicolumn{1}{c}{$\alpha_1=5$} & \multicolumn{1}{c}{$\alpha_1=7$} &
			\multicolumn{1}{c}{$\alpha_1=3$} & \multicolumn{1}{c}{$\alpha_1=5$} & \multicolumn{1}{c}{$\alpha_1=7$} &
			\multicolumn{1}{c}{$\alpha_1=3$} & \multicolumn{1}{c}{$\alpha_1=5$} & \multicolumn{1}{c}{$\alpha_1=7$}
			\\
			\hline
			$CLRT$  & (50,150)  &  0.9321 & 1 & 1  &0.9218 &1 &1 &0.9418&1&1 \\
			&(100,300) & 0.9530 & 1 & 1 &0.9521 &1&1&0.9528&1&1 \\
			&(200,600)  & 0.9610  & 1 & 1&0.9586 &1&1&0.9638&1&1 \\
			$CNTT$   & (50,150)  & 1 & 1 & 1 &0.9998&1&1&1&1&1\\
			&(100,300)   &  1 & 1 & 1 &1&1&1&1&1&1\\
			&(200,600)   & 1 & 1 & 1 &1&1 &1&1&1&1\\
			$RLRT$    & (50,150) & 1 & 1 & 1 &1&1&1&1&1&1\\
			&(100,300)   & 1  & 1 & 1 &1&1&1&1&1&1\\
			&(200,600) & 1 & 1 & 1 &1&1&1&1&1&1\\
			\hline
		\end{tabular}
	\end{table*}	
	
	\begin{table*}
		\caption{Empirical probability of rejecting $ H_{6} $ at significance level $ \xi=1\times10^{-4} $ under assumptions of Gaussian, Gamma, and Uniform distributions  }
		\label{tablepowerd3h6l3}
		\begin{tabular}{@{}lrrrrrrrrrr@{}}
			\hline &&\multicolumn{3}{c}{($ Dt_{1},H_6 $)} & \multicolumn{3}{c}{($ Dt_{2},H_6 $)} & \multicolumn{3}{c}{($ Dt_{3},H_6 $)}\\
			\cline{3-11} &&
			\multicolumn{9}{c}{ $ \alpha_1 $}\\
			\cline{3-11}
			test & \multicolumn{1}{c}{
				(p,n) } 
			&\multicolumn{1}{c}{$2.2$} & \multicolumn{1}{c}{$2.5$} & \multicolumn{1}{c}{$2.8$} &
			\multicolumn{1}{c}{$2.2$} & \multicolumn{1}{c}{$2.5$} & \multicolumn{1}{c}{$2.8$} &
			\multicolumn{1}{c}{$2.2$} & \multicolumn{1}{c}{$2.5$} & \multicolumn{1}{c}{$2.8$}
			\\
			\hline
			$CLRT$  & (50,150)  &  0.3876 & 0.8702 & 0.9943 &0.6033&0.9347&0.9971&0.3881&0.8639&0.9946 \\
			&(100,300) & 0.3849  & 0.8870 & 0.9980&0.6069&0.9616 &0.9993&0.3856&0.8917&0.9985 \\
			&(200,600)  & 0.3761  & 0.9038 & 0.9978&0.6272&0.9698&0.9996 &0.3842&0.9026&0.9990\\
			$CNTT$   & (50,150)  & 0.9974  & 1 & 1 &0.9969&0.9999&1&0.9998&1&1\\
			&(100,300)   &  0.9996 & 1 & 1 &0.9994&1&1&1&1&1\\
			&(200,600)   & 0.9998 & 1 & 1 &1&1&1&1&1&1\\
			$RLRT$    & (50,150) & 0.8845 & 0.9972 & 1 &0.9111&0.9979&1&0.8718&0.9962&1\\
			&(100,300)   & 0.9833  & 1 & 1 &0.9858&1&1&0.9802&1&1\\
			&(200,600) & 0.9994 & 1 & 1 &0.9995&1&1&0.9995&1&1\\
			\hline
		\end{tabular}
	\end{table*}

\end{document}